%% file: TASEP-different_speeds-revision-v2-shortened-revised.tex
\documentclass[letterpaper,reqno,11pt,oneside,final]{amsart} 
\pdfoutput=1

\usepackage{amsmath,amsthm,amsfonts,amssymb}
\usepackage[extension=pdf]{hyperref}
\usepackage{enumitem,comment}
\usepackage[T1]{fontenc}
\usepackage{times}
\usepackage{soul}
\usepackage{mathtools}
\usepackage{mathrsfs}
\usepackage{cases}
\usepackage{needspace}
\usepackage{accents}
\usepackage{bm}
\usepackage[scr=boondoxo,scrscaled=1.05]{mathalfa}
\usepackage[dvipsnames]{xcolor}
\usepackage[multiple]{footmisc}
\usepackage{microtype}
\usepackage{import}
\usepackage{xspace}
\usepackage{tikz}
\usetikzlibrary{matrix,backgrounds,decorations.pathreplacing}
\usepackage{rotating}
\usepackage[totalheight=9.6in,totalwidth=6.15in,centering]{geometry}
\usepackage{graphics,graphicx}
\usepackage{tocvsec2} 
\usepackage{xr}
\usepackage{ifthen}

\newboolean{final}
\setboolean{final}{false}
\ifthenelse{\not{\boolean{final}}}{%
  \usepackage[color,notref,notcite]{showkeys}
  
}{%
}

\usepackage[style=alphabetic,natbib=true,maxnames=99,isbn=false,doi=false,url=false,firstinits=true,hyperref=auto,arxiv=abs,backend=bibtex]{biblatex}
\AtEveryBibitem{%
\clearlist{language}}
\DeclareFieldFormat[article,inbook,incollection,inproceedings,patent,thesis,unpublished]{title}{#1\isdot}
\renewbibmacro{in:}{%
\ifentrytype{article}{}{%
  \printtext{\bibstring{in}\intitlepunct}}} 
\defbibheading{apa}[\refname]{\section*{#1}}
\DeclareFieldFormat{sentencecase}{\MakeSentenceCase{#1}} \renewbibmacro*{title}{%
\ifthenelse{\iffieldundef{title}\AND\iffieldundef{subtitle}}{}
{\ifthenelse{\ifentrytype{article}\OR\ifentrytype{inbook}%
    \OR\ifentrytype{incollection}\OR\ifentrytype{inproceedings}%
    \OR\ifentrytype{inreference}} {\printtext[title]{%
      \printfield[sentencecase]{title}%
      \setunit{\subtitlepunct}%
      \printfield[sentencecase]{subtitle}}}%
  {\printtext[title]{%
      \printfield[titlecase]{title}%
      \setunit{\subtitlepunct}%
      \printfield[titlecase]{subtitle}}}%
  \newunit}%
\printfield{titleaddon}}

\usepackage[f]{esvect}
\usepackage{stmaryrd}

\DeclareMathOperator{\arccot}{arccot}

\DeclareFontFamily{U}{BOONDOX-calo}{\skewchar\font=45 }
\DeclareFontShape{U}{BOONDOX-calo}{m}{n}{
  <-> s*[1.05] BOONDOX-r-calo}{}
\DeclareFontShape{U}{BOONDOX-calo}{b}{n}{
  <-> s*[1.05] BOONDOX-b-calo}{}
\DeclareMathAlphabet{\mcb}{U}{BOONDOX-calo}{m}{n}
\SetMathAlphabet{\mcb}{bold}{U}{BOONDOX-calo}{b}{n}

\makeatletter
\newcommand{\sbullet}{%
  \hbox{\fontfamily{lmr}\fontsize{.6\dimexpr(\f@size pt)}{0}\selectfont\textbullet}}

\makeatother

\makeatletter
\newcommand*\bigcdot{\mathpalette\bigcdot@{.5}}
\newcommand*\bigcdot@[2]{\mathbin{\vcenter{\hbox{\scalebox{#2}{$\m@th#1\bullet$}}}}}
\makeatother

\definecolor{darkblue}{rgb}{0.13,0.13,0.39}%
\hypersetup{colorlinks=true,urlcolor=darkblue,citecolor=darkblue,linkcolor=darkblue,%
  pdftitle=TASEP variants with different speeds,pdfauthor={K. Matetski, D. Remenik}}

\mathtoolsset{showmanualtags,showonlyrefs}

\newtheorem{thm}{Theorem}[section] 
\newtheorem{lem}[thm]{Lemma}
 
\newtheorem{prop}[thm]{Proposition}

\theoremstyle{definition} 
\newtheorem{rem}[thm]{Remark} 
\newtheorem*{rem*}{Remark}

\newtheorem{assumption}[thm]{Assumption} 
\newcounter{assum}

\newcommand{\head}{{\uptext{head}}}

\newcommand{\BFS}{{\uptext{BFS}}}

\newcommand{\rB}{\uptext{r-B}}

\newcommand{\UC}{\uptext{UC}}

\newcommand{\set}[1]{\llbracket #1 \rrbracket}

\renewcommand{\d}{\mathrm{d}}

\newcommand{\T}{\mathbb{T}}
\newcommand{\V}[1]{\mathbb{V}_{\hspace{-0.1em}#1}}

\newcommand{\B}[1]{\mathbb{B}_{#1}}

\newcommand{\Q}{\mcb{V}}
\newcommand{\Qo}{\mcb{Q}}
\newcommand{\E}{\vartheta}

\newcommand{\G}{G}

\newcommand{\fh}{\mathfrak{h}}
\newcommand{\fn}{\mathfrak{n}}

\newcommand{\xx}{X}

\newcommand{\R}{\mathcal{R}}

\newcommand{\CO}{\mathcal{O}}

\newcommand{\CM}{\mathcal{M}}

\newcommand{\cK}{\mathcal{K}}

\newcommand{\cS}{\mathcal{S}}

\newcommand{\vT}{\vv{T}}

\newcommand{\arrowright}[1]{\parbox{#1}{\tikz{\draw[->](0,0)--(#1,0);}}}

\newcommand{\ra}{\vphantom{m}\arrowright{.15cm}}

\newcommand{\I}{{\rm i}} 
\newcommand{\pp}{\mathbb{P}}

\newcommand{\ee}{\mathbb{E}} 
\newcommand{\rr}{\mathbb{R}}
\newcommand{\nn}{\mathbb{N}} 
\newcommand{\zz}{\mathbb{Z}}

\newcommand{\K}{K}
\newcommand{\Id}{I}
\newcommand{\cc}{{\mathbb{C}}}
\newcommand{\p}{\partial}
\newcommand{\uno}[1]{\mathbf{1}_{#1}}
\newcommand{\ep}{\varepsilon}
\newcommand{\eps}{\varepsilon}
\newcommand{\vs}{\vspace{6pt}}

\newcommand{\qand}{\quad\text{and}\quad}
\newcommand{\qqand}{\qquad\text{and}\qquad}

\DeclareMathOperator{\Ai}{\uptext{Ai}}

\newcommand{\ts}{\hspace{0.1em}}
\newcommand{\tts}{\hspace{0.05em}}
\newcommand{\tsm}{\hspace{-0.1em}}

\newcommand{\bioneref}{\hyperlink{it:biorth4}{\normalfont ($\star$)}\xspace}
\newcommand{\bitworef}{\hyperlink{it:poly4}{\normalfont ($\star\star$)}\xspace}

\newcommand{\bionerefbar}{{\normalfont ($\bar\star$)}\xspace}
\newcommand{\bitworefbar}{{\normalfont ($\bar\star\bar\star$)}\xspace}

\makeatletter
\newcommand\RedeclareMathOperator{%
  \@ifstar{\def\rmo@s{m}\rmo@redeclare}{\def\rmo@s{o}\rmo@redeclare}%
}
\newcommand\rmo@redeclare[2]{%
  \begingroup \escapechar\m@ne\xdef\@gtempa{{\string#1}}\endgroup
  \expandafter\@ifundefined\@gtempa
     {\@latex@error{\noexpand#1undefined}\@ehc}%
     \relax
  \expandafter\rmo@declmathop\rmo@s{#1}{#2}}
\newcommand\rmo@declmathop[3]{%
  \DeclareRobustCommand{#2}{\qopname\newmcodes@#1{#3}}%
}
\@onlypreamble\RedeclareMathOperator
\makeatother
\newcommand{\uptext}[1]{\text{\upshape{#1}}}

\RedeclareMathOperator{\det}{\mathop{\uptext{det}}}
\RedeclareMathOperator{\ker}{\mathop{\uptext{ker}}}
\RedeclareMathOperator{\exp}{\mathop{\uptext{exp}}}
\RedeclareMathOperator{\log}{\mathop{\uptext{log}}}
\RedeclareMathOperator*{\lim}{\mathop{\uptext{lim}}}
\RedeclareMathOperator*{\sup}{\mathop{\uptext{sup}}}
\RedeclareMathOperator*{\limsup}{\mathop{\uptext{lim\hspace{1pt}sup}}}
\RedeclareMathOperator*{\liminf}{\mathop{\uptext{lim\hspace{1pt}inf}}}
\RedeclareMathOperator*{\max}{\mathop{\uptext{max}}}
\RedeclareMathOperator*{\inf}{\mathop{\uptext{inf}}}
\RedeclareMathOperator*{\min}{\mathop{\uptext{min}}}
\RedeclareMathOperator*{\cos}{\mathop{\uptext{cos}}}
\RedeclareMathOperator*{\sin}{\mathop{\uptext{sin}}}
\RedeclareMathOperator*{\tan}{\mathop{\uptext{tan}}}
\RedeclareMathOperator*{\sec}{\mathop{\uptext{sec}}}
\RedeclareMathOperator*{\arccot}{\mathop{\uptext{arccot}}}
\RedeclareMathOperator*{\arg}{\mathop{\uptext{arg}}}
\RedeclareMathOperator*{\mod}{\mathop{~~~\uptext{mod}}~~}
\RedeclareMathOperator{\Re}{\mathop{\uptext{Re}}}
\RedeclareMathOperator{\Im}{\mathop{\uptext{Im}}}

\DeclareMathOperator{\spanning}{\uptext{span}}

\newcommand{\twopii}[1]{\ifthenelse{#1=1}{2\pi\I}{(2\pi\I)^{#1}}}

\newcommand{\SLP}{\mathbb{S}}

\newcommand{\Bpl}{B^+}
\newcommand{\Qpl}{Q^+}
\newcommand{\alphapl}{\alpha^+}
\newcommand{\barQpl}{\bar{Q}^+}
\newcommand{\taupl}{{\tau^+}}

\newcommand{\SM}{\mathcal{S}}
\newcommand{\SN}{\bar{\mathcal{S}}}

\newcommand{\fT}{\mathbf{S}}

\newcommand{\ft}{\mathbf{t}}

\newcommand{\fx}{\mathbf{x}}

\newcommand{\fy}{\mathbf{y}}

\newcommand{\fa}{\mathbf{a}}
\newcommand{\fb}{\mathbf{b}}

\newcommand{\fB}{\mathbf{B}}
\newcommand{\fQ}{\mathbf{Q}}
\newcommand{\fK}{\mathbf{K}}
\newcommand{\fI}{\mathbf{I}}
\newcommand{\ftau}{\bm{\tau}}

\renewcommand{\P}{\chi}
\newcommand{\bP}{\bar{\P}}

\newcommand{\rin}{r}
\newcommand{\rout}{\bar\rin}
\newcommand{\rrin}{\rin}

\newcommand{\rhoin}{\rho}
\newcommand{\rhoout}{\bar\rho}

\DeclareMathOperator{\hypo}{\uptext{hypo}}
\DeclareMathOperator{\epi}{\uptext{epi}}

\input{catp-ch}

\def\dash---{\kern.16667em---\penalty\exhyphenpenalty\hskip.16667em\relax}

\numberwithin{equation}{section}

\let\oldmarginpar\marginpar
\renewcommand\marginpar[1]{\-\oldmarginpar[\raggedleft\footnotesize #1]%
  {\raggedright{\small\textsf{#1}}}}

\renewenvironment{quote}{%
  \list{}{%
    \leftmargin0.5cm 
    \rightmargin\leftmargin 
  }%
  \item\relax
}{%
  \endlist
}

\allowdisplaybreaks[2]

\begin{document}

\setcounter{tocdepth}{2}
\setcounter{secnumdepth}{3}
\hypersetup{bookmarksdepth=3}

\protected\def\ignorethis#1\endignorethis{}
\let\endignorethis\relax
\def\TOCstop{\addtocontents{toc}{\ignorethis}}
\def\TOCstart{\addtocontents{toc}{\endignorethis}}

\title{Exact solution of TASEP and variants with inhomogeneous speeds and memory lengths}

\author{Konstantin Matetski} \address[K.~Matetski]{
  Department of Mathematics\\
  Michigan State University\\
  619 Red Cedar Road\\
  East Lansing, MI 48824\\
  USA} \email{matetski@msu.edu}

\author{Daniel Remenik} \address[D.~Remenik]{
  Departamento de Ingenier\'ia Matem\'atica and Centro de Modelamiento Matem\'atico (IRL-CNRS 2807)\\
  Universidad de Chile\\
  Av. Beauchef 851, Torre Norte, Piso 5\\
  Santiago\\
  Chile} \email{dremenik@dim.uchile.cl}

\begin{abstract}
    In \cite{fixedpt,TASEPgeneral} an explicit biorthogonalization method was developed that applies to a class of determinantal measures which describe the evolution of several variants of classical interacting particle systems in the KPZ universality class.
    The method leads to explicit Fredholm determinant formulas for the multipoint distributions of these systems which are suitable for asymptotic analysis.
    In this paper we extend the method to a broader class of determinantal measures which is applicable to systems where particles have different jump speeds and different memory lengths.
    As an application of our results we study three particular examples: some variants of TASEP with two blocks of particles having different speeds, a version of discrete time TASEP which mixes particles with sequential and parallel update, and a version of sequential TASEP with a block of long memory particles placed at the bulk of the system.
\end{abstract}

\maketitle
\tableofcontents

\vskip-40pt
  
\section{Introduction}
\label{sec:intro}

In the (continuous time, one-dimensional) \emph{totally asymmetric simple exclusion process (TASEP)}, particles perform totally asymmetric nearest neighbour random walks on the integer lattice $\zz$ subject to the exclusion rule: each particle independently attempts jumps to the neighbouring site to its right at rate $1$, the jumps being allowed only when the destination site is empty.
Despite its simplicity, TASEP presents a very rich asymptotic behavior, and due to its tractability it has become a paradigmatic model in out-of-equilibrium statistical physics.

Much of the interest in TASEP arises from the central role it plays as a member of the \emph{KPZ universality class}, a broad collection of physical and probabilistic models including particle systems, one-dimensional random growth models, directed polymers, stochastic reaction-diffusion equations, and random stirred fluids.
Models in the KPZ class share a common asymptotic fluctuation behavior, identified by their (in general, conjectural) convergence, under the characteristic KPZ \emph{1:2:3 scaling}, to a universal, scale-invariant Markov process known as the \emph{KPZ fixed point}, which was first constructed in \cite{fixedpt} as the scaling limit of TASEP.
This 1:2:3 scaling refers to the ratios between the exponents used to rescale the fluctuations, space and time: for KPZ models, as $t\to\infty$ one has fluctuations growing like $t^{1/3}$ with non-trivial spatial correlations arising at a scale of $t^{2/3}$.

What makes TASEP special in this context is that its distribution can be expressed as a marginal of an (in general, signed) determinantal point process.
For general initial data, this was first discovered in \cite{sasamoto,borFerPrahSasam} (building on exact determinantal formulas for the transition probabilities of the system derived in \cite{MR1468391} using the coordinate Bethe ansatz), where it was used to study the special case of \emph{periodic} initial data, with particles initially occupying sites at $2\zz$.
There the associated spatial fluctuations in the long time 1:2:3 scaling limit were derived; they lead to the \emph{Airy$_1$ process}, whose marginals are given by the Tracy-Widom GOE distribution from random matrix theory \cite{tracyWidom2}.
For another choice of special initial data known as \emph{step}, where particles initially occupy sites at $\zz_{<0}$, there is an even richer algebraic structure, and the analogous scaling limit had been known since the early 2000's \cite{johanssonShape,prahoferSpohn,johansson}, leading to the \emph{Airy$_2$ process} and Tracy-Widom GUE marginals \cite{tracyWidom}.

The method employed in \cite{sasamoto,borFerPrahSasam} is based on a representation of TASEP through a biorthogonal ensemble, which leads to an expression for the multipoint distribution of TASEP as the Fredholm determinant of a kernel defined implicitly as the solution of a certain biorthogonalization problem which depends on the initial data of the system.
For step initial data, the biorthogonalization turns out to be, in a certain sense, trivial (the functions one needs to biorthogonalize are essentially already orthogonal), while for periodic initial data the authors were able to solve it explicitly.
The solution of the biorthogonalization problem for general initial data was discovered in \cite{fixedpt}, and leads to a kernel which can be expressed in terms of the hitting time of a certain random walk to a curve defined by the initial data.
In the 1:2:3 scaling limit, this kernel naturally rescales to an analogous kernel defined in terms of Brownian hitting times, whose Fredholm determinants yield the finite dimensional distributions of the KPZ fixed point.

\TOCstop

TASEP is part of a family of exactly solvable models in the KPZ class for which a description in terms of biorthogonal ensembles is available.
Besides continuous time TASEP, this family includes discrete time TASEP with both sequential and parallel update, with pushing and blocking dynamics, and with Bernoulli and geometric jumps, as well as several generalizations.
 In \cite{TASEPgeneral} we extended the explicit biorthogonalization method of \cite{fixedpt} to a general class of determinantal measures which includes these models and several others (the method has also been applied in continuous space to systems of one-sided reflected Brownian motions \cite{nqr-RBM}, and more generally in the recent paper \cite{assiotis} to a certain family of one-sided reflected diffusions with polynomial drift and diffusion coefficients).

\subsection*{Description of the main result}

The purpose of this paper is to develop a further generalization of the explicit biorthogonalization method to cover extensions of these models to the case where particles have different speeds and different memory lengths.
By the \emph{speed} of a particle we mean, in the context of continuous time TASEP, simply its jump rate.
The \emph{memory length} of a particle, on the other hand, is easier to interpret in the case of discrete time TASEP: it refers to the amount of time a site remains blocked after a particle occupying it leaves.
Memory lengths $0$ and $1$ translate respectively into the standard discrete time TASEPs with sequential and parallel update.
For more general memory lengths, the system is no longer Markovian, but it can be reinterpreted as a Markovian system of \emph{interacting caterpillars}, which occupy a variable number of sites in the lattice.

Our starting point is the biorthogonal ensemble representation for such systems.
Such a representation for TASEP-like systems in the case of inhomogeneous speeds is well known \cite{bp-push,borodFerSas}.
Those papers focus on two particle systems, PushASEP (a combination of TASEP with blocking and pushing dynamics) and TASEP with parallel update, for which they compute scaling limits in the case of periodic initial data.
They actually obtain more general multipoint distributions along ``space-like paths'' (i.e., the distribution of collections of particles at different times, but subject to a certain ordering in space-time).
As we will explain in the next section, the case of TASEP with general memory lengths, or caterpillars, can be recovered by considering an extension of this setting to one where particles are also allowed to start evolving at different times.
The biorthogonal ensemble representation in this setting was obtained in some generality in \cite{TASEPgeneral}, but the explicit solution of the biorthogonalization problem in that paper was restricted to the case corresponding to equal caterpillar lengths and equal speeds.

Our main goal is thus to complete this program in the general setting of inhomogeneous speeds and memory lengths, by providing an explicit formula for the biorthogonal kernel appearing in these formulas which is amenable to asymptotic analysis.
Our main result, Thm.~\ref{thm:kernel-explicit}, will actually be proved in a more general, abstract setting involving a broad class of kernels constructed out of the solution of a biorthogonalization problem.
This abstract theorem will cover all of the examples we have mentioned so far (and several more).
In the next section we describe the result as it applies to the concrete setting of variants of TASEP with inhomogeneous speeds and memory lengths.

The explicit solution of the biorthogonalization problem in this general setting is still given in terms of a kernel involving a hitting problem for a random walk.
In the case of memory length $0$ (e.g. continuous time TASEP) and equal speeds, the jumps of this random walk have a geometric distribution with some fixed parameter.
The effect of introducing different speeds in the particle system is, perhaps not surprisingly, to make the parameters of the geometric jump distribution of the random walk appearing inside the kernel be time-inhomogeneous, and chosen according to the speed profile of the system.
The effect of considering inhomogeneous memory lengths is more delicate.
In fact, changing the memory length of the system changes, in a certain sense, the nature of the associated biorthogonalization problem, which leads to random walks with a modified jump distribution.
Having inhomogeneous memory lengths ends up leading to a kernel in terms of a random walk with a time-dependent distribution which varies accordingly, but it was far from clear a priori (to us, at least) that this would be so.
In particular, there are some structural difficulties in the biorthogonalization problem which need to be overcome (see Rem. \ref{rem:pfcat}) but, notably, the solution can still be expressed in terms of a random walk hitting problem with the same structure that had appeared earlier.

The availability of these explicit formulas, and in particular the fact that they are given in terms of relatively familiar kernels, opens up the possibility of studying several examples and, in particular, performing asymptotic analysis.
In Sec.~\ref{sec:application} we include three applications to variants of TASEP in continuous and discrete time.
Our goal is mainly to illustrate how the formulas can be applied in a few relatively simple cases, while leaving the more detailed exploration of other interesting examples, which require a more delicate asymptotic analysis, for future work; we will briefly mention a few such possibilities along the way.

Before briefly discussing those examples, we mention that in the particular case of discrete time TASEP with right Bernoulli jumps, the explicit kernel for inhomogeneous speeds was obtained recently, and independently, by \citet{bisiLiaoSaenzZigouras}.
In that paper they also provide a new derivation of the biorthogonal ensemble representation of the system which uses combinatorial properties of the Robinson-Schensted-Knuth correspondence together with intertwining relations to express the transition kernel of the system in terms of an ensemble of non-intersecting lattice paths.

\subsection*{Applications}

Our first application of Thm. \ref{thm:kernel-explicit} is to the two-speed setting studied in \cite{TwoSpeed}.
In that paper, the authors considered continuous time TASEP with a leading block of particles with a different speed.
They obtained explicit contour integral formulas in the case of periodic initial data, for which they were able to perform the asymptotic analysis necessary to describe its limiting behavior depending on the parameters of the model (the length and speed of the leading block), which in the most interesting cases lead to shocks.
Here we will obtain similar formulas for systems in a slightly more general setting which includes continuous and discrete time TASEP with both sequential and parallel update.
The result which we will obtain is for general initial data, and the resulting formulas can be used to perform asymptotic analysis in the general case, but we leave this for future work.

Instead of modifying the particle speeds in a way which leads to shocks as the previous example, one could also perturb them microscopically according to a smooth profile.
It is natural to expect that if the perturbation is chosen appropriately and the scaling is adjusted suitably, the system should still present KPZ fixed point limiting behavior.
This can actually be derived from our formulas, but the inhomogeneity in the speeds makes the asymptotic analysis more challenging, so we also leave this for future work.

In the next application we consider a system based on discrete time TASEP which mixes particles updating sequentially and in parallel (with equal speeds).
We will show that, under the appropriate scaling, the system converges to the KPZ fixed point.
We will do this at the level of pointwise convergence of the kernels, for which we provide a detailed argument; the proof can be upgraded to obtain the full convergence by adapting the arguments in \cite{fixedpt}.
It is worth noting that in this example we work with initial data with general particle density, in contrast with the standard choice of density $\frac12$ usually made in the literature.

The final application we consider is also based on discrete time TASEP.
Now we start with a system updating sequentially, and modify the memory lengths of a macroscopic block of particles located in the bulk. As one might expect, this modification introduces a delay in the system but, once this delay is accounted for in the scaling, the system still converges to the KPZ fixed point. 
A perhaps more natural variant is to place the block of long caterpillars at the beginning of the system; this leads to a more drastic perturbation of the system which introduces additional complications in the asymptotic analysis, and is also left for future work.

\subsection*{Outline}

In Sec.~\ref{sec:examples} we describe several interacting particle systems (and some of their generalizations to systems of caterpillars) in the KPZ universality class, whose distributions are particular cases of the determinantal measure considered in Sec.~\ref{sec:biorth-det}. Under quite general assumptions, Thm.~\ref{thm:biorth_general} states that a marginal of this measure can be written as a Fredholm determinant of a kernel described implicitly through the solution of a biorthogonalization problem. Sec.~\ref{sec:biorth} is devoted to the explicit solution of this problem in an abstract setting, leading to an explicit formula for the kernel in Thm.~\ref{thm:kernel-explicit}. Finally, in Sec.~\ref{sec:application} we study this kernel and its KPZ scaling limit for the particular examples mentioned above.

\subsection*{Notation}

We will use the same notation and conventions employed in \cite{TASEPgeneral}.
We use the standard notation $\nn$ for the set of natural numbers $\{1, 2, \dotsc\}$, and we use $\nn_0 = \nn \cup \{0\}$. For $n \in \nn$ we define the set $\set{n} = \{1, \dotsc, n\}$. For $N \geq 2$ the \emph{Weyl chamber} is
\begin{equation*}
\Omega_{N} = \{\vec x \in \zz^N\!: x_N < x_{N-1} <\dotsm < x_1 \}.
\end{equation*}

Throughout the paper we consider various kernels $K\!:\zz\times\zz\longrightarrow\rr$, which we identify with integral operators acting on suitable families of functions $f\!: \zz \to \cc$ as
\begin{equation}\label{eq:kernel-general}
Kf(x) = \sum_{y \in \zz} K(x,y)f(y).
\end{equation}
We prefer not to specify precisely the domains of such operators and always interpret them in terms of absolutely convergent sums \eqref{eq:kernel-general}. The composition of two such operators $K$ and $L$ is defined as $KL(x,y)=\sum_{z\in\zz}K(x,z)L(z,y)$, provided that the sum is absolutely convergent. Then we say that $K^{-1}$ is an inverse of an operator $K$ if $KK^{-1}(x,y)=K^{-1}K(x,y)=I(x,y)$, where $I$ is the identity operator $I(x,y)=\uno{x=y}$. 
We use the standard notation $K^*(x_1, x_2) = K(x_2,x_1)$ for the adjoint. 

Our kernels will often be defined in terms of functions written as contour integrals. The contours in these integrals will be usually $\gamma_{r}$, a circle in the complex plane with radius $r$ and centered at the origin. Whenever the contour is different, it will be specified explicitly. 

For a closed subset $U$ of $\cc$ we say that a complex function $f$ is analytic on $U$ if it is analytic on some open domain which contains $U$.
A particular case of interest will be when $U$ is the closed annulus on the complex plane centered at the origin and with radii $0<r<R$, which we denote by $A_{r,R}$. 

Finally, for a fixed vector $a\in\rr^m$ and indices $n_1<\dotsm<n_m$ we let 
\begin{equation}\label{eq:defChis}
\P_a(n_j,x)=\uno{x>a_j} \qqand \bP_a(n_j,x)=\uno{x\leq a_j},
\end{equation}
which we also regard as multiplication operators acting on the space $\ell^2(\{n_1,\dotsc,n_m\}\times\zz)$.

\TOCstart
\section{Motivating examples: interacting particle systems}
\label{sec:examples} 

The main result of this paper will be stated in Sec.~\ref{sec:biorth} in an abstract setting which, in general, does not necessarily originate from determinantal measures connected to particle systems.
In order to motivate that setting and to provide some physical intuition, we begin by presenting in this section some particular cases of that result, stated in the context of the variants of TASEP and systems of interacting caterpillars with inhomogeneous jump speeds and lengths (equivalently, discrete time TASEPs with inhomogeneous speeds and memory lengths) discussed in the introduction.
We will begin by presenting the general formula which we will obtain for the multipoint distribution of this type of systems.
At this stage we will not be precise about the details of and assumptions on the systems to which this result will apply.
Later on we will introduce particular cases corresponding to several particle systems and systems of caterpillars to which the result will apply.
The precise setting for our general result will be provided in Secs.~\ref{sec:biorth-det} and \ref{sec:biorth}.

A \emph{(forward) caterpillar} of length $L\geq 1$ is an element $\xx$ of the set
\begin{equation*}
\cK^{\ra}_L=\big\{(\xx^1,\dotsc,\xx^L)\in\zz^L\!: \xx^{i}-\xx^{i+1}\in\{0,1\},\,i \in \set{L-1}\big\}.
\end{equation*}
A caterpillar thus has $L$ ordered sections $\xx^{1} \geq \xx^{2} \geq \dotsm \geq \xx^{L}$; we will call $\xx^{1}$ the \emph{head} of the caterpillar.
A system of $N \geq 2$ interacting caterpillars of lengths $\vec L=(L_1, \ldots, L_N) \geq 1$ takes values in the set
\begin{equation*}%
\Omega^{\catp}_{N, \vec L}=\big\{\xx = (\xx(1), \dotsc, \xx(N))\!: \xx(i)\in\cK^{\ra}_{L_i}\!:\xx^1(i+1) < \xx^{L_i}(i),\,i \in \set{N-1}\big\},
\end{equation*}
i.e., the caterpillar $\xx(i)$ has length $L_i$ and no two caterpillars overlap. For $\xx\in\Omega^{\catp}_{N, \vec L}$ we define $\xx^{\head} = (\xx^1(i) : i \in \set{N}) \in \Omega_N$ to be the vector of heads of the caterpillars, which can be thought of as $N$ particles located at the sites $\xx^1(i)$ for $i \in \set{N}$.

Now for fixed speeds $v_i > 0$, $i\in\set{N}$, we will consider certain specific dynamics for caterpillars $\xx_t \in \Omega^{\catp}_{N, \vec L}$ in time $t$, which is either in $\rr_+$ or in $\nn_0$.
The simplest example is the case of continuous time TASEP, where all caterpillars have length $1$ and the $i$-th one jumps to the right at rate $v_i$ except that jumps onto already occupied sites are forbidden.
We provide below other examples of dynamics of caterpillars to which our results are applicable; those with lengths $2$ or more all evolve in discrete time (we remark that there is also a generalized version of continuous time TASEP which has the flavor of a length-$2$ system, but its definition does not quite fit the setting of this section, although it can be analyzed in the framework of Sec.~\ref{sec:biorth}, see \cite[Sec. 3.3]{TASEPgeneral}).

We say that the system of caterpillars $\xx_t$ has \emph{initial condition} $\vec y \in \Omega_N$ if $\xx_0 \in \Omega^{\catp}_{N, \vec L}$ is given by $\xx^1_0(k) = \dotsm = \xx^{L_k}_0(k) = y_k$ for each $k \in\ \set{N}$; in words, the $k$-th caterpillar starts with all its sections at $y_k$.
With a little ambiguity, we will write in this case $\xx_0 = \vec y \in \Omega_N$.
Throughout the paper we will be restricted to work in the case when the initial condition $\vec y$ is in the set
\begin{equation}\label{eq:Omega-kappa}
\Omega_{N}(\vec L) = \{\vec x \in \zz^N\!:x_{i} - x_{i+1} \geq (L_i-1)\vee1\uptext{ for }i = 1, \dotsc, N-1\}.
\end{equation}

The following holds for all of the systems of caterpillars considered in this paper: for fixed initial condition $\vec y\in \Omega_{N}(\vec L)$, for any $t \geq 0$ and $1 \leq n_1 < \cdots < n_m \leq N$, and for any real $a_1, \ldots, a_m$, the distribution function of the heads of the system can be written in the form
\begin{subequations}\label{eqs:caterpillars}
\begin{equation}\label{eq:caterpillars}
\pp\bigl(\xx^{\head}_t(i) > a_i,\; i = 1, \ldots, m\bigr) = \det \bigl(\Id- \bP_{a}  \K_t \bP_{a}  \bigr)_{\ell^2(\{n_1, \ldots, n_m\} \times \zz)},
\end{equation}
where the kernel $\K_t$ is given implicitly in terms of the solution of a certain biorthogonalization problem which involves the initial data $\vec y$.
The determinant on the right hand side of \eqref{eqs:caterpillars} is the Fredholm determinant on the Hilbert space $\ell^2(\{n_1, \ldots, n_m\} \times \zz)$ (see \cite[Sec.~2]{quastelRem-review} for a brief overview of the Fredholm determinant, or \cite{simon} for a comprehensive treatment).

The precise form of the biorthogonal kernel $\K_t$ is presented in Sec.~\ref{sec:biorth-det}.
We will explain shortly one way to interpret the restriction to $\xx_0=\vec y$ with $\vec y\in\Omega_{N}(\vec L)$ in our setting.
The restriction actually arises as a requirement in the proof of this representation (which for systems of caterpillars can be found in \cite{TASEPgeneral}), but it appears in general to be necessary for it to hold.
The same restriction will be crucial in the proof of our main abstract result, Thm.~\ref{thm:kernel-explicit}, from which the results presented in this section will be corollaries.

Our main result will provide an explicit formula for the kernel $\K_t$ appearing in \eqref{eq:caterpillars}.
The formula follows from solving explicitly the biorthogonalization problem defining the kernel for general initial condition and representing the result in terms of a hitting problem for a certain random walk to a curve defined by the initial data.
This representation is such that the appropriate scaling limits can be obtained naturally, by computing the limits of the kernels involved in the formula and recognizing that the random walk hitting problem converges to a similar problem for a Brownian motion; we present examples of this in Secs.~\ref{sec:caterpillars-mix} and \ref{sec:caterpillar-block} (see also \cite{fixedpt} where the scheme was implemented in detail for continuous time TASEP).

In order to state our formula we first need to make several definitions. 
Consider a meromorphic function $\varphi\!:U\longrightarrow\cc$, where the domain $U\subseteq\cc$ contains $0$ and all values $v_i$, which is analytic and non-zero in an annulus $A_{\rrin,\rout}$ with radii $0<\rrin<\min v_i$ and $\rout > \max v_i$.
Fix also a real parameter $\theta\in(\rrin,\min v_i)$.
We introduce the kernels 
\begin{equation}
Q_{(\ell, n]}(x,y) = \frac{1}{2\pi\I}\oint_{\gamma_\rrin}\d w\,\frac{\theta^{x-y}}{w^{x-y - n + \ell + 1}} \prod_{i = \ell+1}^n \frac{\alpha_i \varphi(w)^{L_{i-1}-1}}{v_i-w}
\end{equation}
with $\alpha_i = \frac{v_{i}-\theta}{\theta} \varphi(\theta)^{1 - L_{i-1}}$, integer $0 \leq \ell < n$ and $L_0 = 1$, and 
\begin{equation}
\Qpl_{(\ell, n]}(x,y) = \frac{1}{2\pi\I}\oint_{\gamma_\rrin}\d w\,\frac{\theta^{x-y}}{w^{x-y - n + \ell + 1}} \prod_{i = \ell+1}^n \frac{\alphapl_i \varphi(w)^{L_{i}-1}}{v_i-w},
\end{equation}
with $\alphapl_i = \frac{v_{i}-\theta}{\theta}\varphi(\theta)^{1 - L_{i}}$. 
These two kernels are Markov.
We let $\Bpl_m$ be the time-inhomogeneous random walk which has transitions from time $m-1$ to time $m$, $m\geq1$, with step distribution $\Qpl_{(m-1, m]}$.
For a fixed initial condition $\vec y \in \Omega_{N}(\vec L)$ we define the stopping time
\begin{equation}
\taupl= \min\{m=0,\dotsc,N-1 : \Bpl_m> y_{m+1}\},
\end{equation}
i.e., $\taupl$ is the hitting time of the strict epigraph of the ``curve'' $(y_{m+1})_{m=0,\dotsc,n-1}$ by the random walk $(\Bpl_m)_{m\geq0}$ (we set $\taupl=\infty$ if the walk does not go above the curve by time $N-1$). 

Next for integer $n \geq 1$ and $0 \leq m < n$ and for a real $t \geq 0$ we define the kernels
  \begin{align}
\SM_{-n}(x, y) & = \frac{1}{2\pi\I}\oint_{\gamma_\rrin}\d w\,\frac{\theta^{y-x}}{w^{y-x + n + 1}} \varphi(w)^t \frac{\prod_{i = 1}^n(v_i-w)}{\prod_{i = 1}^{n} \alphapl_i\prod_{i = 1}^{n-1}\varphi(w)^{L_i - 1}}, \nonumber \\
\SN_{(m, n]}(x, y)  &= -\frac{1}{2\pi\I}\oint_{\Gamma_{\vec v}} \d w\,\frac{\theta^{x-y}}{w^{x-y - n + m + 1}} \varphi(w)^{-t} \frac{\prod_{i = m+1}^{n}\alphapl_i\prod_{i = m+1}^{n-1} \varphi(w)^{L_i - 1}}{\prod_{i=m+1}^n(v_i-w)},
\end{align}
and 
\begin{equation}
\SN^{\epi(\vec y)}_{n}(x, y) = \ee_{\Bpl_{0} = x} \bigl[\SN_{(\taupl, n]}(\Bpl_{\taupl}, y) \uno{\taupl < n}\bigr].
\end{equation}

We can now state our formula for the kernel $\K_t$ appearing in the Fredholm determinant formula for the multipoint distribution of the caterpillar heads \eqref{eq:caterpillars}.
Recall that we are considering a fixed initial condition $\vec y\in \Omega_{N}(\vec L)$ and we have $\pp\bigl(\xx^{\head}_t(i) > a_i,\; i = 1, \ldots, m\bigr) = \det \bigl(\Id- \bP_{a}  \K_t \bP_{a}  \bigr)_{\ell^2(\{n_1, \ldots, n_m\} \times \zz)}$ for $t\geq 0$, $1 \leq n_1 < \cdots < n_m \leq N$, and $a_1, \ldots, a_m\in\rr$.
Our result, which is valid for all the systems of caterpillars considered in this paper, is that the kernel $\K_t$ is given by
\begin{equation}\label{eq:caterpillars-Kt}
\K_t(n_i,x_i;n_j, x_j) = -Q_{(n_i, n_j]} (x_i, x_j) \uno{n_i<n_j} + (\SM_{-n_i})^* \SN^{\epi(\vec y)}_{n_j} (x_i, x_j).
\end{equation}
\end{subequations}
This formula will appear as particular cases of Thm.~\ref{thm:kernel-explicit}, which computes the correlation kernel for a general class of determinantal measures (see the comments at the beginning of Sec.~\ref{sec:application}).

Next we provide examples of particle systems and caterpillars for which the formula \eqref{eqs:caterpillars} holds. 
The proof of the propositions stated below is explained at the beginning of Sec.~\ref{sec:application}.

\subsection{Continuous time TASEP}

In \emph{continuous time TASEP} with inhomogeneous speeds one has $N$ particles $\xx_t(1)>\xx_t(2)>\dotsm>\xx_t(N)$ evolving as follows: the $i$-th particle tries to make unit jumps to the right at rate $v_i>0$, but attempted jumps are permitted only if the destination site is empty.
Except for the exclusion restriction, jumps by different particles occur independently.

\begin{prop}\label{prop:TASEP}
The distribution function of $X_t=\xx^{\head}_t$ for continuous time TASEP is given by \eqref{eq:caterpillars}/\eqref{eq:caterpillars-Kt} with $\varphi(w) = e^w$ and $L_i = 1$ for all $i \in \set{N}$.
\end{prop}

\subsection{Discrete time TASEPs with right Bernoulli jumps}
\label{sec:right-Bernoulli}

Next we introduce \emph{discrete time TASEP with right Bernoulli jumps} and with inhomogeneous speeds.
There are two natural variants of this model: \emph{sequential} and \emph{parallel update}.
Fix \emph{speed} parameters $p_i \in (0,1)$, $i\in\set{N}$.
Again we have particles occupying $\zz$ at locations $\xx_t(1)>\xx_t(2)>\dotsm>\xx_t(N)$.
Now to go from time $t$ to time $t+1$, particles are updated one by one, from right to left in the sequential case and from left to right in the parallel case, as follows: the $i$-th particle jumps to the right with probability $p_i$ and stays put with probability $q_i=1-p_i$, but if a particle tries to jump on top of an occupied site, the transition is blocked.
Note that in the case of sequential update, a particle trying to jump at time $t$ is blocked by the position of its right neighbor at time $t+1$, while in the case of parallel update the particle is blocked by its neighbor at time $t$. 

\begin{prop}
The distribution function of $X_t=\xx^{\head}_t$ for discrete time TASEP
with right Bernoulli jumps is given by \eqref{eq:caterpillars}/\eqref{eq:caterpillars-Kt} with $\varphi(w) = 1+ w$ and $v_i = p_i / q_i$, and with $L_i = 1$ for all $i \in \set{N}$ in the case of sequential update and $L_i = 2$ for all $i \in \set{N}$ in the parallel case.
\end{prop}

\subsection{Caterpillars with right Bernoulli jumps}
\label{sec:caterpillars-rB}

Now for fixed parameters $p_i \in (0,1)$, $i\in\set{N}$, we define a dynamics for caterpillars $\xx_t \in \Omega^{\catp}_{N, \vec L}$ in discrete time $t\in\nn_0$.
The transition from time $t$ to time $t+1$ occurs in the following way, with the positions of the caterpillars being updated consecutively for $i\in\set{N}$ (i.e., from right to left):
\begin{itemize}[leftmargin=0.5cm]
  \item The head of the $i$-th caterpillar makes a unit step to the right with probability $p_i \in (0,1)$ (i.e., $\xx^1_{t+1}(i) = \xx^1_t(i) + 1$), provided that the destination site is empty. Otherwise it stays put (i.e., $\xx^1_{t+1}(i) = \xx^1_t(i)$).
  \item The remaining sections of the $i$-th caterpillar move according to $\xx^j_{t+1}(i) = \xx^{j-1}_t(i)$, $j = 2, \dotsc, L_i$.
\end{itemize}
In words, the heads jump as in TASEP with right Bernoulli jumps, but are blocked by the whole caterpillar to its right, while each of the remaining sections of each caterpillar follows the movement of the section to its right in the previous time step.
One sees directly that the new configuration $\xx_{t+1}$ is again in $\Omega^{\catp}_{N, \vec L}$ and that this choice of dynamics defines a Markov chain on $\Omega^{\catp}_{N, \vec L}$.

It is easy to see from the definition of its dynamics that the heads in this system of caterpillars evolve as a version of discrete time TASEP, with right to left update, where particle $i$ at time $t$ is blocked by particle $i-1$ according to its location at time $t-L_{i-1}$, which provides the interpretation of caterpillars as encoding the memory lengths of the system.

Based on the last observation, it is natural to couple the model with a version of sequential TASEP with different starting times. 
In this extension of TASEP we fix starting times $0 \geq T_1 \geq T_2 \geq \dotsm \geq T_N$ and an initial configuration of particles $\vec{y} \in \Omega_N$,
and run the process with particle $i$ starting its evolution at $\xx^{\rB}_{T_i}(i)=y_i$ at time $T_i$.
In other words, from time $T_N$ to time $T_{N-1}$ only the $N$-th particle moves with the other particles staying put, then at time $T_{N-1}$ particle $N-1$ starts moving, and the two move together up to time $T_{N-2}$, when particle $N-2$ joins them, and so on.
Throughout its evolution, particle $i$ jumps to the right with probability $p_i$, provided that the target site is empty.
The coupling between the models is given in the following result (which for constant $L_j$ appeared as Lem.~2.1 in \cite{TASEPgeneral}), and follows directly from the definitions of the two models: %

\begin{lem}\label{lem:TASEP_and_caterpillars}
Let the process $\xx^{\rB}_t$ start at initial times $\vT = (-\sum_{1 \leq j < k}(L_j-1))_{k \in \set{N}}$ and at a configuration $\vec y \in \Omega_N(\vec L)$. Define for each $k\in\set{N}$ and $i \in \set{L_k}$
\begin{equation}
\xx^i_t(k) = \xx^{\rB}_{t - \sum_{1 \leq j < k}(L_j-1)-i+1}(k).
\end{equation}
Then $\xx_t \in \Omega^{\catp}_{N, \vec L}$ is distributed as the system of interacting caterpillars of lengths $\vec L$ described above, with initial condition $\vec y$.
\end{lem}

The restriction that the initial data $\vec{y}$ for the system of caterpillars be in $\Omega_N(\vec L)$ means, at the level of the initial data of the coupled TASEP particle system, that its starting times and locations have to satisfy $\xx^{\rB}_{T_{i-1}}(i-1)-\xx^{\rB}_0(i)\geq(T_{i-1}-T_{i})\vee1$, which resolves any
ambiguity in the evolution of the particles for small times (as each particle can interact with its right neighbor only after this neighbor has started its evolution).

Lem. \ref{lem:TASEP_and_caterpillars} puts systems of caterpillars with right Bernoulli jumps in the setting of Sec. \ref{sec:biorth-det}, and thus allows us to use the results of Sec. \ref{sec:biorth}.

\begin{prop}
The distribution of the heads $\xx^{\head}_t$ of right Bernoulli caterpillars is given by \eqref{eq:caterpillars}/\eqref{eq:caterpillars-Kt} with $\varphi(w) = 1+ w$, $v_i = p_i / q_i$, and the chosen values of the length parameters $L_i$.
\end{prop}

\subsection{Other types of caterpillars}

There are four basic variants of TASEP whose transition probabilities have determinantal formulas (of the form \eqref{eq:G} below), corresponding to combinations of Bernoulli and geometric jumps, and blocking and pushing dynamics.
These four variants were described in \cite{MR2469339} in relation to each of the four known variants of the Robinson-Schensted-Knuth (RSK) correspondence: the RSK and Burge algorithms, as well as their dual variants.
The TASEP dynamics considered in the previous two subsections correspond to Bernoulli jumps and blocking dynamics.
In the case of pushing dynamics, particles now jump to the left instead of to the right, updating from right to left, and when a jumping particle lands in an occupied site, the occupying particle is pushed to the left, being forced to jump.
The case of geometric jumps is similar, with particles now jumping according to a geometric distribution, with parameter $q_i \in (0,1)$ for particle $i$; in the case of pushing dynamics particles still update from right to left, but in the blocking case the update is from left to right (i.e., in parallel).

In the two cases with pushing dynamics one can construct corresponding systems of caterpillars (with inhomogeneous speeds and lengths) through a construction which is completely analogous to the one in Sec. \ref{sec:caterpillars-rB}.
The resulting caterpillar dynamics are described in Secs. 2.2 and 2.3 of \cite{TASEPgeneral} in the case of equal speeds and lengths, and can be adapted to the inhomogeneous case straightforwardly.

In the remaining case, right geometric jumps with blocking dynamics, the construction is slightly different, and is  restricted to considering mixtures of particles updating sequentially and in parallel; the construction and resulting dynamics are described in Sec. 2.4 of \cite{TASEPgeneral} for the case of all particles updating sequentially (an analog in this case of all $L_i$'s being equal to $1$), and can be adapted similarly to the inhomogeneous case.

\begin{prop}
The distribution of the heads of the caterpillars $\xx^{\head}_t$ is given in the above cases by \eqref{eq:caterpillars}/\eqref{eq:caterpillars-Kt} with (here $p_i=1-q_i$)
\begin{itemize}
\item $\varphi(w) = 1+ 1 / w$ and $v_i = p_i / q_i$ for left Bernoulli jumps with pushing dynamics,
\item $\varphi(w) = 1 / (1-1/w)$ and $v_i = 1 / q_i$ for left geometric jumps with pushing dynamics,
\item $\varphi(w) = 1 / (1-w)$ and $v_i = q_i$ for right geometric jumps with blocking dynamics,
\end{itemize}
and the chosen values of the length parameters $L_i$.
\end{prop}

\subsection{PushASEP}

As a last example we consider the case of PushASEP \cite{bp-push}, which is a version of TASEP where blocking and pushing dynamics occur together.
We will only discuss the model in continuous time and in a setting corresponding to all caterpillars having length $1$, although similar constructions can be made in some other cases. 
In this model there are two global parameters $r, \ell \geq0$, and the evolution is as follows. 
Particles jump independently to the right and to the left, with particle $i$ jumping to the right at rate $r v_i$ and to the left at rate $\ell / v_i$. 
When a particle jumps to the right onto an occupied site, the jump is forbidden (blocking dynamics).
When a particle jumps to the left onto an occupied site, it pushes the particle to the left, forcing it to jump (pushing dynamics).

\begin{prop}\label{prop:PushASEP}
The distribution of the particles $X_t=\xx^{\head}_t$ is given again by \eqref{eq:caterpillars}/\eqref{eq:caterpillars-Kt}, in this case with  $\varphi(w) = e^{r w+\ell/w}$, $L_i=1$, and the chosen values of the speed parameters $v_i$.
\end{prop}

\TOCstart
\section{Biorthogonal ensemble formula for determinantal measures}
\label{sec:biorth-det}

In \cite[Sec. 4]{TASEPgeneral} a Fredholm determinant formula, involving kernels given implicitly in a biorthogonal form, was given for certain marginals of a class of (in general, signed) determinantal measures in an abstract setting.
That result is a generalization of the results obtained for specific particle systems in earlier work such as \cite{sasamoto,borFerPrahSasam,bp-push,borodFerSas}, which covers all the examples considered in that paper, as well as the case of different speeds and different memory lengths considered here.
For clarity, and because it involves some definitions and notation which we will need later on anyway, we include the full result here (for proofs we refer to \cite{TASEPgeneral}).
At the end of this section we will explain how it applies to the particle systems discussed in Sec.~\ref{sec:examples}.

Throughout the section, $t$ denotes a time variable taking values in $\T$, which is either $\rr$ or $\zz$.
We fix $N\in\nn$ and a vector of speeds $\vec v = (v_i)_{i \in \set{N}}$ with $v_i > 0$ for each $i$.

The following result, whose proof can be found in \cite{TASEPgeneral} (Lem.~5.6), will be used often in this section and the next one to compute compositions of the kernels of a certain form:

\begin{lem}\label{lem:conv}
Consider two kernels $S_1$ and $S_2$ given by
\begin{equation}
S_i(x,y)=\frac1{2\pi\I}\oint_{\gamma}\d w\,\frac{\theta^{x-y}}{w^{x-y+1}}\phi_i(w),\label{eq:Si}
\end{equation}
where $\phi_1,\phi_2$ are complex functions which are both analytic on an annulus $A_{r_1,r_2}$ for some $r_1<r_2$ and $\gamma$ is any simple, positively oriented closed contour contained in $A_{r_1,r_2}$.
Then the series $S_1S_2(x,y)=\sum_{z\in\zz}S_1(x,z)S_2(z,y)$ is absolutely convergent and
\begin{equation}\label{eq:conv}
S_1S_2(x,y)=\frac1{2\pi\I}\oint_{\gamma}\d w\,\frac{\theta^{x-y}}{w^{x-y+1}}\phi_1(w)\phi_2(w).
\end{equation}
\end{lem}

Define the kernel
\begin{equation*}%
\Q_{i}(x_1, x_2) = \frac{1}{2\pi\I}\oint_{\gamma_{\rhoout}}\!\d w\, \frac{(w - v_i)^{-1}}{w^{x_2 - x_1}} = v_i^{x_1 - x_2} \uno{x_1 \geq x_2}
\end{equation*}
for $i\in\set{N}$ and $x_1, x_2 \in \zz$, where $\rhoout>\max_iv_i$.
The inverse of $\Q_{i}$ is
\begin{equation*}%
\Q^{-1}_{i}(x_1, x_2) = \frac{1}{2\pi\I}\oint_{\gamma_{\rhoin}}\d w\, \frac{w - v_i}{w^{x_2 - x_1 + 2}} = \uno{x_1 = x_2} -  v_i \uno{x_1 = x_2 + 1},
\end{equation*}
where $\rhoin>0$. 
For $k \in \set{N}$ we set
\begin{equation*}%
\Q^{[k]} = \Q_{1} \Q_{2} \dotsm \Q_{k}, \qquad \Q^{[-k]} = \Q_{k}^{-1} \dotsm \Q_{2}^{-1} \Q_{1}^{-1},
\end{equation*}
with the convention $\Q^{[0]} = I$.
The kernels of these operators can be written explicitly (using Lem.~\ref{lem:conv}) as 
\begin{equation*}%
\Q^{[k]}(x_1, x_2) = \frac{1}{2\pi\I}\oint_{\gamma_{\rhoout}}\!\d w\, \frac{\prod_{i = 1}^k(w - v_i)^{-1}}{w^{x_2 - x_1 - k + 1}}, \qquad \Q^{[-k]}(x_1, x_2) = \frac{1}{2\pi\I}\oint_{\gamma_\rhoin}\!\d w\, \frac{\prod_{i = 1}^k(w - v_i)}{w^{x_2 - x_1 + k + 1}}.
\end{equation*}
We also introduce the (multiplication) kernels
\begin{equation*}%
\E_{k} (x_1, x_2) = v_k^{-x_1} \uno{x_1 = x_2}, \qquad \E_{-k} (x_1, x_2) = v_k^{x_2} \uno{x_1 = x_2}.
\end{equation*}

Next we introduce a kernel
\begin{equation*}%
\R_t(x_1,x_2) = \frac{1}{2\pi\I}\oint_{\gamma_\rhoin}\d w\, \frac{\varphi(w)^t}{w^{x_2 - x_1 +1}},
\end{equation*}
which depends on a given complex function $\varphi$.
We will assume that $\varphi$ and the radii $\rhoin$ and $\rhoout$ satisfy the following:

\begin{assumption}\label{a:phi}
\leavevmode
\begin{enumerate}[label=\uptext{(\roman*)}]
\item $\varphi\!:U\longrightarrow\cc$, where the domain $U\subseteq\cc$ contains $0$ and all values $v_i$, and $\varphi$ has at most a finite number of singularities in $U$.
\item $\varphi$ is analytic on an annulus $A_{\rhoin,\rhoout} \subseteq U$ with radii $0<\rhoin < \min_iv_i$ and $\rhoout > \max_iv_i$.
\item $\varphi(w) \neq 0$ for all $w\in A_{\rhoin,\rhoout}$.
\end{enumerate}
\end{assumption}

For $k, \ell \in \set{N}$ and $t \in \T$ we define the function
\begin{align}
F_{k, \ell}(x_1,x_2; t) &= \bigl(\E_k \Q^{[k]} \R_t \Q^{[-\ell]} \E_{-\ell}\bigr) (x_1,x_2)\label{eq:F_def}\\
&= \frac{1}{2\pi\I}\oint_{\gamma_{\rhoout}}\!\!\d w\, \frac{(w/v_k)^{x_1}}{(w/v_\ell)^{x_2}} \frac{\prod_{i = 1}^{\ell} (w - v_i)}{\prod_{i = 1}^{k} (w - v_i)} \frac{\varphi(w)^t}{w^{\ell - k +1}}.\label{eq:F_formula}
\end{align}
Finally, for $\vec y, \vec x \in \Omega_N$ and $s \leq t$, we define
\begin{align}\label{eq:G}
\G^{(N)}_{s, t} (\vec y, \vec x) = \left( \prod_{i = 1}^N \varphi(v_i)^{s-t} \right) \det \bigl[F_{k, \ell}(y_{k}, x_{\ell}; t-s)\bigr]_{k, \ell \in \set{N}}.
\end{align}

The function \eqref{eq:G} defines, by convolution, an (in general, signed) measure on particle configurations in a space-time domain. 
We are interested in the projections of this measure to special sets known as space-like paths, which we introduce now. For $(n_1, t_1), (n_2, t_2) \in \set{N} \times \T$ we write $(n_1, t_1) \prec (n_2, t_2)$ if $n_1 \leq n_2$, $t_1 \geq t_2$ and $(n_1, t_1) \neq (n_2, t_2)$. We write $\fn=(n,t)$ to denote elements of $\set{N} \times \T$.
Then we define the set of \emph{space-like paths} as
\begin{equation}\label{eq:space-like-paths}
\SLP_N = \bigcup_{m \geq 1} \bigl\{(\fn_i)_{i \in \set{m}}\!: \fn_i \in \set{N} \times \T, \fn_i \prec \fn_{i+1}\bigr\}.
\end{equation}
For a space-like path $\cS = \{(n_1, t_1), \dotsc, (n_m, t_m)\} \in \SLP_N$ and for $\vec y \in \Omega_N$ and $\vec x \in \Omega_m$, we set
\begin{equation}\label{eq:G+}
G^{+}_{\cS} (\vec y, \vec x) = \sum_{\substack{\vec x(t_i) \in \Omega_{n_i} : \\ x_{n_i}(t_i) = x_i, i \in \set{m}}} \G^{(n_m)}_{0, t_m}(\vec y_{\leq n_m}, \vec x(t_m)) \prod_{i=1}^{m-1} \G^{(n_i)}_{t_{i + 1}, t_{i}}(\vec x_{\leq n_{i}}(t_{i+1}), \vec x(t_{i})),
\end{equation}
where we write $\vec x_{\leq n_{i}}(t_{i+1})$ for the vector in $\Omega_{n_i}$ obtained by removing the entries $n_{i}+1, \dotsm, n_{i+1}$ from $\vec x(t_{i+1})$, and similarly for $\vec y_{\leq n_m}$.
We use $\vec x(t_i)$ to parametrize vectors by time points. In particular, we postulate that $\vec x(t_i)$ and $\vec x(t_{i+1})$ are different vectors even if $t_{i} = t_{i+1}$ (this slight abuse of notation, which makes clear the correspondence between vectors and the associated time points, will simplify the presentation later on).
For $T_N \leq \dotsm \leq T_1$ and for $\vec x \in \Omega_N$ and $\vec y \in \zz^N$, we set  
\begin{equation}\label{eq:G-}
G^{-}_{\vT} (\vec y, \vec x)  = \left( \prod_{i = 1}^N \varphi(v_i)^{T_i} \right)  \det \bigl[F_{k, \ell}(y_k, x_\ell; - T_{k})\bigr]_{k, \ell \in \set{N}},
\end{equation}
In applications to particle systems such as TASEP, $G^{+}_{\cS} (\vec y, \vec x)$ corresponds to starting the system in configuration $\vec y$ at time $t=0$ and asking for the probability that $X_{t_i}(n_i)=x_i$ for each $i\in\set{m}$, while $G^{-}_{\vT} (\vec y, \vec x)$ corresponds to having particle $i$ start evolving at time $T_i$ from $x_i$ and asking for the probability that $X_0=\vec y$.

Convolving \eqref{eq:G+} and \eqref{eq:G-} in the case $T_1 \leq t_m$, we define
\begin{equation}\label{eq:G_TS}
G_{\vT, \cS}(\vec y, \vec x) = \sum_{\vec z \in \Omega_N} \G^{-}_{\vT} (\vec y, \vec z) \G^{+}_{\cS} (\vec z, \vec x).
\end{equation}
Our goal is to obtain a formula for the following integrated version of $\G_{\vT,\cS}$: for $\vec y\in\zz^N$, $\vec a\in\zz^m$,
\begin{equation}\label{eq:mu_fixed}
\CM_{\vT,\cS}(\vec y, \vec a) = \sum_{\substack{\vec x \in \Omega_m : \\ x_i > a_i,  i \in \set{m}}} \G_{\vT,\cS}(\vec y, \vec x).
\end{equation}
In words, one should think of a collection of $N$ particles evolving in time, such that the $i$-th particle starts at location $y_i$ at time $T_i$. Then for a fixed space-like path $\cS$, containing pairs $(n_i, t_i)$, $G_{\vT, \cS}(\vec y, \vec x)$ defines a measure on $\vec x \in \Omega_m$, with the $i$-th element of $\vec x$ intepreted as the position of the $n_i$-th particle at time $t_i$. $\CM_{\vT,\cS}(\vec y, \vec a)$ is then the measure of the set of all particle configurations so that the $n_i$-th particle is located strictly to the right from $a_i$ at time $t_i$.

Before stating the result we need to introduce a certain space of functions $\V{n}(\vec v, \theta)$.
For fixed $n\in\nn$, $\theta > 0$ and a vector $\vec v$ as above, let $\nu(n)$ be the number of distinct values among the first $n$ entries $v_1,\dotsc,v_n$ of $\vec v$, let $u_1(n) < u_2(n) < \dotsm < u_{\nu(n)}(n)$ denote these distinct values, and let $\beta_k(n)$ be the multiplicity of $u_k(n)$ among these entries (in particular, $\sum_{k = 1}^{\nu(n)} \beta_k(n) = n$).
Then we define 
\begin{equation}\label{eq:space}
\V{n}(\vec v, \theta) = \spanning \bigl\{x \in \zz \longmapsto x^\ell (u_k(n) / \theta)^{x} : 1 \leq k \leq \nu(n),\; 0 \leq \ell < \beta_k(n)\bigr\}.
\end{equation}

Finally, given a space like path $\cS=\{\fn_1,\dotsc,\fn_m\}$ as above and a fixed vector $a\in\rr^m$ we extend the notation introduced in \eqref{eq:defChis} to $\P_a(\fn_j,x)=1-\bP_a(\fn_j,x)=\uno{x>a_j}$. The following result can be found in \cite[Thm.~4.3]{TASEPgeneral}.

\begin{thm}\label{thm:biorth_general}
Let the function $\varphi$ and the values $v_i$ satisfy Assum.~\ref{a:phi}, and fix $T_N \leq \dotsm \leq T_1$ and a space-like path $\cS$, the time points of which are all greater than $T_1$.
Then the function \eqref{eq:mu_fixed} can be written as
\begin{equation}\label{eq:M_formula}
\CM_{\vT,\cS}(\vec y, \vec a) = \det \bigl(\Id- \bP_{a}  \K \bP_{a}  \bigr)_{\ell^2(\cS \times \zz)},
\end{equation}
where $\det$ is the Fredholm determinant and:
\begin{enumerate}[label=\uptext{(\arabic*)}]
\item The kernel $\K\!: (\cS \times \zz)^2\longrightarrow\rr$ depends on $\vv T$ and $\vec y$, and is given by
\begin{equation}\label{eq:KernelK}
\K(\fn_i, x_i; \fn_j, x_j)=-\phi^{(\fn_i, \fn_j)}(x_i,x_j) \uno{\fn_i \prec \fn_j} + \sum_{k = 1}^{n_j}\Psi^{\fn_i}_{n_i - k}(x_i)\Phi^{\fn_j}_{n_j - k}(x_j),
\end{equation}
for $\fn_i = (n_i, t_i)$ and $\fn_j = (n_j, t_j)$ in $\cS$.
\item For $\fn_i$ and $\fn_j$ as before, such that $\fn_i \prec \fn_j$, the function $\phi^{(\fn_i, \fn_j)}$ is defined as
\begin{equation}\label{eq:phi}
\phi^{(\fn_i, \fn_j)}(x_i, x_j) = \frac{1}{2\pi\I}\oint_{\gamma_\rhoin}\d w\, \frac{\theta^{x_i - x_j} \varphi(w)^{t_i - t_j}}{w^{x_i - x_j - n_j + n_i + 1}} \prod_{k = n_i + 1}^{n_j} (v_k - w)^{-1}.
\end{equation}
\item For $\fn = (n, t) \in\cS$ and $k \in \set N$, the function $\Psi^{\fn}_{n - k}$ is given by
\begin{equation}\label{eq:Psi}
\Psi^{\fn}_{n - k}(x) = \frac{1}{2\pi\I}\oint_{\gamma_\rhoin}\d w\,\frac{\theta^{x-y_{k}} \varphi(w)^{t - T_k}}{w^{x - y_k + n - k + 1}} \frac{\prod_{i = 1}^n (v_i - w)}{\prod_{i = 1}^k (v_i - w)}.
\end{equation}
\item The functions $\Phi^{\fn}_{n - k}$, for $ k \in \set n$ and $\fn = (n,t)$, are uniquely characterized by:
\begin{enumerate}[label=\uptext{(\alph*)}]
\item The biorthogonality relation $\sum_{x\in\zz}\Psi_\ell^{\fn}(x)\Phi_k^{\fn}(x)=\uno{k=\ell}$, for each $k,\ell=0,\dotsc,n-1$.\label{it:biorth}
\item $\spanning\{x \in \zz \longmapsto \Phi^{\fn}_{k}(x) : 0 \leq k < n \} = \V{n}(\vec v, \theta)$.\label{it:poly}
\end{enumerate}
\end{enumerate}
\end{thm}

In applications to particle systems we are usually interested in the case $\cS = \{(i, t + T_i) : i \in \set{N}\}$ for some $T_1\geq\dotsm\geq T_N$, corresponding to starting particle $i$ at time $t+T_i$. In this case each point $\fn = (n, t)$ in $\cS$ is determined by its first component $n$ and the the kernel in \eqref{eq:KernelK} can be reexpressed as a kernel $\K\!: (\set{N} \times \zz)^2\longrightarrow\rr$ given by 
\begin{equation}\label{eq:KernelK-particular}
\K(n_i, x_i; n_j, x_j)=-\phi^{(n_i, n_j)}(x_i,x_j) \uno{n_i < n_j} + \sum_{k = 1}^{n_j}\Psi^{n_i}_{n_i - k}(x_i)\Phi^{n_j}_{n_j - k}(x_j),
\end{equation}
with $\phi^{(n_i,n_j)}(x_i,x_j)=\frac{1}{2\pi\I}\oint_{\gamma_\rhoin}\d w\, \frac{\theta^{x_i - x_j} \varphi(w)^{T_{n_i} - T_{n_j}}}{w^{x_i - x_j - n_j + n_i + 1}} \prod_{k = n_i + 1}^{n_j} (v_k - w)^{-1}$ for $n_i<n_j$, $\Psi^{n}_{n - k}(x) = \frac{1}{2\pi\I}\oint_{\gamma_\rhoin}\d w\,\frac{\theta^{x-y_{k}} \varphi(w)^{t + T_n - T_k}}{w^{x - y_k + n - k + 1}} \frac{\prod_{i = 1}^n (v_i - w)}{\prod_{i = 1}^k (v_i - w)}$, and the $\Phi^{n}_{n - k}$ are uniquely characterized by the conditions in (4) of the last theorem (with $\fn$ replaced by $n$).

Consider now the concrete setting of Sec.~\ref{sec:caterpillars-rB}, where $T_k = - \sum_{1 \leq j < k} (L_j-1)$ (see Lem. \ref{lem:TASEP_and_caterpillars}) with the special choice $\varphi(w)=1+w$ and $v_i = p_i/q_i$. As explained in \cite[Sec.~1.2]{TASEPgeneral}, the representation of the distribution functions of caterpillars as measures of the type \eqref{eq:mu_fixed} requires separation of initial states $y_{j} - y_{j+1}\geq L_{j} - 1$ for all $j$ (this condition on the initial state appears also in Thm.~\ref{thm:kernel-explicit} where we prove a formula for the kernel corresponding to this measure).
In that case the functions $\phi^{(n_i,n_j)}$ and $\Psi^n_{n-k}$ appearing in \eqref{eq:KernelK-particular} are given by
\begin{equation}\label{eq:phi-caterpillars}
\phi^{(n_i, n_j)}(x_i, x_j) = \frac{1}{2\pi\I}\oint_{\gamma_\rhoin}\d w\, \frac{\theta^{x_i - x_j}}{w^{x_i - x_j - n_j + n_i + 1}} \prod_{k = n_i + 1}^{n_j} \frac{\varphi(w)^{L_{k-1}-1}}{v_k - w}
\end{equation}
and 
\begin{equation}\label{eq:Psi-caterpillars}
\Psi^{n}_{n - k}(x) = \frac{1}{2\pi\I}\oint_{\gamma_\rhoin}\d w\,\frac{\theta^{x-y_{k}} \varphi(w)^{t}}{w^{x - y_k + n - k + 1}} \frac{\prod_{i = 1}^n (v_i - w) / \varphi(w)^{L_{i-1}-1}}{\prod_{i = 1}^k (v_i - w) / \varphi(w)^{L_{i-1}-1}},
\end{equation}
where $L_{0} = 1$.
Analogous formulas can be obtained from \eqref{eq:KernelK-particular} for the other models described in Sec.~\ref{sec:examples}, using the choices of function $\varphi$ and speeds $v_i$ detailed in that section.
The fact that the multipoint distributions of these models can be expressed through \eqref{eq:mu_fixed} is explained in \cite[Secs.~2, 3]{TASEPgeneral} (the argument was given there in the case of equal speeds, but it extends to general case without changes).

\section{Explicit biorthogonalization scheme}
\label{sec:biorth}

Throughout this section we fix $N\in\nn$, which in applications to particle systems corresponds to the number of particles in the system under consideration.
We also fix vectors $\vec v\in(\rr_{>0})^N$ and $\vec y\in\zz^N$, which play the role of the particle speeds and initial positions\footnote{In applications we usually consider systems with infinitely many particles but where the evolution of the first $N$ particles is not affected by the remaining ones; since our formulas will yield the finite-dimensional distributions of the system, this restriction to $\vec v$ and $\vec y$ of size $N$ is not consequential.}.

\subsection{Setting}

We consider a family of strictly positive measures $(q_{\ell}(i))_{i\in\zz}$ on $\zz$, $\ell\in\set{N-1}$, which satisfies:

\begin{assumption}\label{assum:q}
\leavevmode
\begin{enumerate}[label=\uptext{(\roman*)}]
\item For each $\ell\in\set{N-1}$ there is a $\kappa_\ell \in\nn_0$ such that $q_{\ell}(i)=1$ for all $i>\kappa_\ell$,
\item\label{item:q2} There is a $\theta\in(0,\min_{j\in\set{N}}v_j)$ such that or each $\ell\in\set{N-1}$, $\sum_{i\in\zz}q_{\ell}(i)(\theta/v_\ell)^i<\infty$ and $\sum_{i\in\zz}q_{\ell}(i)(\theta/v_{\ell+1})^i<\infty$.
\end{enumerate}
\end{assumption}

Next we introduce a function $a_\ell(w)$, $\ell\in\set{N-1}$, which is constructed out of the measures $q_\ell$ through the following Laurent series:
\begin{equation}\label{eq:def-a}
a_\ell(w)=\sum_{i\leq\kappa_\ell}(q_{\ell}(i+1)-q_{\ell}(i))(w / v_\ell)^i.
\end{equation}
For convenience we also set
\begin{equation}\label{eq:convention0}
q_0(i)=\uno{i>0},\qquad\kappa_0=0\qqand a_0(w)=1.
\end{equation}
We also consider a fixed complex function $\psi$.
We assume that $\psi$ and the $a_\ell$'s satisfy:

\begin{assumption}\label{assum:apsi}
There are radii $\rin$ and $\rout$ satisfying $0<\rrin<\theta<\min v_i$, and $\rout > \max v_i$ (with $\theta$ given in Assum.~\ref{assum:q}) such that $a_\ell(w)$ is analytic on $\{w\in\cc\!:|w|\geq\rrin\}$ while $1/a_\ell(w)$, $\psi(w)$ and $1/\psi(w)$ are analytic and non-zero on the annulus $A_{\rin,\rout}$.
\end{assumption}

Using the functions $a_\ell$ we introduce the Markov kernels
\begin{equation}\label{eq:Q}
Q_\ell(x,y) = \frac{\alpha_{\ell}}{2\pi\I}\oint_{\gamma_\rrin}\d w\,\frac{\theta^{x-y}}{w^{x-y}}\frac{a_{\ell-1}(w)}{v_\ell - w}
\end{equation}
for $x, y \in \zz$ and $\ell\in\set{N}$, with 
\begin{equation}
\alpha_{\ell}=\frac{v_\ell-\theta}{a_{\ell-1}(\theta)\theta} = \frac{1}{\sum_{i \in \zz} (\theta / v_\ell)^i q_{\ell-1}(i)}.\label{eq:alphaell}
\end{equation}
Due to Assump.~\ref{assum:q}\ref{item:q2}, the sum in this expression is finite. 
Note that since $\rrin<v_\ell$, the contour $\gamma_\rrin$ in the integral in \eqref{eq:Q} includes only the pole at $w = 0$. 
Using \eqref{eq:def-a} we can write explicitly, for $\ell\in\set{N}$, 
\begin{equation}\label{eq:Q-alt}
Q_\ell(x,y) = \alpha_{\ell} (\theta / v_\ell)^{x-y} q_{\ell-1}(x-y).
\end{equation}
which shows that $Q_\ell$ is indeed a Markov kernel (by the definition of $\alpha_\ell$; recall also that $q_{\ell-1}$ is a positive measure).
In particular, since for $x-y>\kappa_{\ell}$ we have $q_\ell(x-y) = 1$,
\begin{equation}\label{eq:Q-simple}
Q_\ell(x,y) = \alpha_{\ell} (\theta / v_\ell)^{x-y} \qquad\forall~x - y > \kappa_{\ell-1}.
\end{equation}
Note in particular that $Q_1$ is simply the transition kernel of a geometric random walk:
\[Q_1(x,y)=\frac{v_1-\theta}{\theta} (\theta / v_1)^{x-y}\uno{x>y}.\]

\begin{rem}
Note that we have defined $Q_\ell$ using the function $a_{\ell-1}$. It might seem more natural to use $a_\ell$ in the definition, but in our setting this is not the case: thinking about the systems of caterpillars from Sec.~\ref{sec:caterpillars-rB}, the dynamics of the head of the $\ell$-th caterpillar depends on its ``speed'' $p_\ell$ and on the length $L_{\ell-1}$ of the caterpillar to its right, but not on its own length. This is also why we do not need to introduce the measures $q_\ell$ and the functions $a_\ell$ for $\ell=N$.
\end{rem}

The inverse kernel of $Q_\ell$ is
\begin{equation}\label{eq:Q-inverse}
Q^{-1}_\ell(x,y) = \frac{\alpha^{-1}_{\ell}}{2\pi\I}\oint_{\gamma_\rrin}\d w\,\frac{\theta^{x-y}}{w^{x-y + 2}}\frac{v_\ell - w}{a_{\ell-1}(w)}.
\end{equation}
Given integers $0 \leq \ell < n$ we denote  
\begin{equation}\label{eq:Q-many}
Q_{(\ell, n]}(x,y) = Q_{\ell+1} \cdots Q_n(x,y) = \frac{1}{2\pi\I}\oint_{\gamma_\rrin}\d w\,\frac{\theta^{x-y}}{w^{x-y - n + \ell + 1}} \prod_{i = \ell+1}^n \frac{\alpha_{i} a_{i-1}(w)}{v_i-w},
\end{equation}
whose inverse is
\begin{equation}\label{eq:Q-many-inv}
Q_{(\ell,n]}^{-1}(x,y) = Q^{-1}_{n} \cdots Q^{-1}_{\ell+1}(x,y) = \frac{1}{2\pi\I}\oint_{\gamma_\rrin}\d w\,\frac{\theta^{x-y}}{w^{x-y + n - \ell + 1}} \prod_{i = \ell+1}^n \frac{v_i-w}{\alpha_{i} a_{i-1}(w)},
\end{equation}
where we used Lem.~\ref{lem:conv}. We note that these formulas make sense also for $\ell =n$ if we postulate that the (empty) products in this case are equal to $1$: $Q_{(n, n]}=Q^{-1}_{(n, n]}=I$.
We also set $Q_{[\ell, n]} = Q_{(\ell - 1, n]}$ for $1 \leq \ell \leq n$ and $Q_{[1, k)} = Q_{[1, k-1]}$ for $k \geq 2$.

Similarly, using the function $\psi$ we define a kernel $\R$ and its inverse $\R^{-1}$ as
\begin{equation}\label{eq:defR}
\R(x,y)=\frac1{2\pi\I}\oint_{\gamma_\rin}\d w\,\frac{\theta^{x-y}}{w^{x-y+1}}\psi(w), \qquad \R^{-1}(x,y)=\frac1{2\pi\I}\oint_{\gamma_\rin}\d w\,\frac{\theta^{x-y}}{w^{x-y+1}} \frac{1}{\psi(w)}.
\end{equation}

\subsection{The biorthogonalization problem}\label{sec:biorthproblem}

For $k \in \set n$, we define
\begin{equation}\label{eq:defPsink}
\Psi^n_{n - k}(x) = \R Q^{-1}_{(k, n]}(x,y_{k}) = \frac{\theta^{x-y}}{2\pi\I}\oint_{\gamma_{\rin}}\d w\,\frac{\psi(w)}{w^{x-y_{k}+n - k+1}} \prod_{i = k+1}^n \frac{v_i-w}{\alpha_{i} a_{i-1}(w)}.
\end{equation}
We extend this definition to $k>n$ by setting
\begin{equation}\label{eq:defPsinkext}
\Psi^n_{n - k}(x) = \R Q_{(n, k]}(x,y_{k}) = \frac{\theta^{x-y}}{2\pi\I}\oint_{\gamma_{\rin}}\d w\,\frac{\psi(w)}{w^{x-y_{k}+n - k+1}} \prod_{i = n+1}^k \frac{\alpha_{i} a_{i-1}(w)}{v_i-w}.
\end{equation}
We consider a family of functions $(\Phi^n_k)_{k=0,\dotsc,n-1}$ characterized by:
\begin{itemize}
\item[($\star$)] \hypertarget{it:biorth4}{The} biorthogonality relation $\sum_{x\in\zz}\Psi_\ell^{n}(x)\Phi_k^{n}(x)=\uno{k=\ell}$ for each $k,\ell=0,\dotsc,n-1$.
\item[($\star\star$)] \hypertarget{it:poly4} $\spanning\{x \in \zz \longmapsto \Phi^{n}_{k}(x) : 0 \leq k < n \} = \V{n}(\vec v, \theta)$, where the set $\V{n}(\vec v, \theta)$ is defined in \eqref{eq:space}. %
\end{itemize}
When all values $v_i$ are equal to $1$, this biorthogonalization problem simplifies to the one considered in \cite[Sec.~5.2]{TASEPgeneral}.

Existence and uniqueness of the solution to this biorthogonalization problem is proved in Lem.~\ref{lem:biorth-unique} below, while an exact solution is provided in Thm.~\ref{thm:h_heat_Q}.

It will be convenient in the following computations to employ a different basis of the space \eqref{eq:space}: 
\begin{equation}\label{eq:basis}
\B{n}(\vec v, \theta) = \bigl\{ e^n_{k, \ell}(x) : 1 \leq k \leq \nu(n),\; 0 \leq \ell < \beta_k(n)\bigr\},
\end{equation}
where the basis functions are 
\begin{equation}\label{eq:basis-functions}
e^n_{k, \ell}(x) = (x)_\ell (u_k(n) / \theta)^{x},
\end{equation}
$(x)_\ell = x(x-1) \cdots (x - \ell+1)$ is the falling factorial and, we recall, $\nu(n)$ and $\beta_k(n)$ where defined in the paragraph preceding \eqref{eq:basis}.
Then the space \eqref{eq:space} can be expressed as follows:
\begin{equation}\label{eq:newbasis}
\V{n}(\vec v, \theta) = \spanning \bigl\{x \in \zz \longmapsto f(x) : f \in \B{n}(\vec v, \theta)\bigr\}.
\end{equation}

In the following two lemmas we demonstrate how convolutions with the kernels $(Q_{n}^*)^{-1}$, $\R^*$ and $(\R^{-1})^*$ act on the functions \eqref{eq:basis-functions}.

\begin{lem}\label{lem:Q-spaces}
Fix $n \geq 1$.
For each $1 \leq k \leq \nu(n)$ and $0 \leq \ell < \beta_k(n)$, there exist real values $c^n_{k}(\ell, m)$, $0 \leq m < \beta_k(n)$, such that 
\begin{equation}\label{eq:basis-and-Q}
(Q_{n}^*)^{-1} e^{n}_{k, \ell}(x) = \sum_{m = 0}^{\ell} c^n_{k}(\ell, m) e^{n}_{k, m}(x).
\end{equation}
Moreover, $c^n_{k}(\ell, \ell) \neq 0$ if $v_n \neq u_k(n)$. In the case $v_n = u_k(n)$ we have $c^n_{k}(\ell, \ell) = 0$ and $c^n_{k}(\ell, \ell-1) \neq 0$, where the latter holds if $\ell \geq 1$. 
In particular, the operator $(Q_{n}^*)^{-1}$ maps $\V{n}(\vec v, \theta)$ to $\V{n - 1}(\vec v, \theta)$, with the convention $\V{0}(\vec v, \theta) = \{0\}$. 
\end{lem}

\begin{proof}
We need to prove only the expansion \eqref{eq:basis-and-Q} and the stated properties of the coefficients $c^n_{k}(\ell, m)$, since the last statement in the lemma follows from those.

From \eqref{eq:Q-inverse} and \eqref{eq:basis} we have
\begin{equation}
(Q_{n}^*)^{-1} e^{n}_{k, \ell}(z) = \sum_{x \in \zz} e^{n}_{k, \ell}(x) Q_{n}^{-1}(x,z) 
= \sum_{x \in \zz} (x)_\ell\Big(\frac{u_k(n)}{\theta}\Big)^{x}\ts\frac{\alpha^{-1}_{n}}{2\pi\I}\oint_{\gamma_\rrin}\d w\,\frac{\theta^{x-z}}{w^{x-z + 1}}\frac{v_{n} - w}{w a_{n-1}(w)}.
\end{equation}
Changing the summation variable $x \longmapsto x+z$ and using the binomial identity for falling factorials $(x+z)_\ell=\sum_{m=0}^\ell\binom{\ell}{m}(x)_m(z)_{\ell-m}$, we write the preceding expression as
\begin{equation}\label{eq:sum}
\sum_{m = 0}^\ell {\ell \choose m} (z)_{\ell-m} \Big(\frac{u_k(n)}{\theta}\Big)^{z} \sum_{x \in \zz} (x)_{m} \frac{\alpha^{-1}_{n}}{2\pi\I}\oint_{\gamma_\rrin}\d w\,\frac{u_k(n)^{x}}{w^{x + 1}}\frac{v_{n} - w}{w a_{n-1}(w)}.
\end{equation}
Now for any $m\in\zz_{\geq 0}$ and any complex $\xi$ satisfying $|\xi| < 1$ we have 
\[\textstyle\sum_{x \geq 0} (x)_m \xi^x = \xi^m \frac{\d^m}{\d \xi^m} \sum_{x \geq 0} \xi^x = \xi^m \frac{\d^m}{\d \xi^m} \frac{1}{1 - \xi} = \xi^m \frac{m!}{(1 - \xi)^{m+1}},\]
and, similarly, in the case $|\xi| > 1$ we have 
\[\textstyle\sum_{x < 0} (x)_m \xi^x = \xi^m \frac{\d^m}{\d \xi^m} \sum_{x < 0} \xi^x = - \xi^m \frac{\d^m}{\d \xi^m} \frac{1}{1-\xi} = -\xi^m \frac{m!}{(1 - \xi)^{m+1}}.\]
Hence for the sum over $x \geq 0$ in \eqref{eq:sum} we can deform the integration contour to $\gamma_{\rout}$ (thanks to Assum.~\ref{assum:apsi}) so that $|w| > u_k(n)$ to get
\begin{equation*}
\sum_{x \geq 0} (x)_{m} \frac{\alpha^{-1}_{n}}{2\pi\I}\oint_{\gamma_{\rout}}\d w\,\frac{u_k(n)^{x}}{w^{x + 1}}\frac{v_{n} - w}{w a_{n-1}(w)} = m! \frac{\alpha^{-1}_{n}}{2\pi\I}\oint_{\gamma_{\rout}}\d w\,\frac{1}{(w - u_k(n))^{m+1}}\frac{v_{n} - w}{w a_{n-1}(w)},
\end{equation*}
while for the sum over $x < 0$ the contour satisfies $|w| < u_k(n)$, so
\begin{equation*}
\sum_{x < 0} (x)_{m} \frac{\alpha^{-1}_{n}}{2\pi\I}\oint_{\gamma_{\rin}}\d w\,\frac{u_k(n)^{x}}{w^{x + 1}}\frac{v_{n} - w}{w a_{n-1}(w)} = - m! \frac{\alpha^{-1}_{n}}{2\pi\I}\oint_{\gamma_{\rin}}\d w\,\frac{1}{(w - u_k(n))^{m+1}}\frac{v_{n} - w}{w a_{n-1}(w)}.
\end{equation*}
In these computations we used Fubini's theorem to swap summation and integration. Since $\rrin<u_k(n)<\rout$, adding the two expressions we conclude that the sum over $x$ in \eqref{eq:sum} equals 
\begin{equation*}
m! \frac{\alpha^{-1}_{n}}{2\pi\I}\oint_{\Gamma_{u_k(n)}}\d w\,\frac{1}{(w - u_k(n))^{m+1}}\frac{v_{n} - w}{w a_{n-1}(w)}
\end{equation*}
where the contour $\Gamma_{u_k(n)}$ includes only the pole at $u_k(n)$.
Using this in \eqref{eq:sum} together with Cauchy's integral formula we get
\begin{equation*}
(Q_{n}^*)^{-1} e^{n}_{k, \ell}(z) = \alpha^{-1}_{n} \sum_{m = 0}^\ell {\ell \choose m} (z)_{\ell-m} \left(\frac{u_k(n)}{\theta}\right)^{z} \frac{\d^m}{\d w^m} \left(\frac{v_{n} - w}{w a_{n-1}(w)}\right) \bigg|_{w = u_k(n)}.
\end{equation*}
The right-hand side is in the span of the functions $e^{n}_{k, \ell-m}(z)$ for $0 \leq m \leq \ell$, and it can be written as \eqref{eq:basis-and-Q} with the constants 
\begin{equation*}
c^n_{k}(\ell, m) = \alpha^{-1}_{n}  {\ell \choose \ell - m} \frac{\d^{\ell - m}}{\d w^{\ell - m}} \left(\frac{v_{n} - w}{w a_{n-1}(w)}\right) \bigg|_{w = u_k(n)}.
\end{equation*}
If $v_{n} \neq u_k(n)$, then the preceding formula yields $c^n_{k}(\ell, \ell) = \alpha^{-1}_{n} \frac{v_{n} - u_k(n)}{u_k(n) a_{n-1}(u_k(n))} \neq 0$. On the other hand, if $v_{n} = u_k(n)$, then we have $c^n_{k}(\ell, \ell) = 0$ and additionally if $\ell\geq1$,
\begin{equation*}
c^n_{k}(\ell, \ell-1) = \alpha^{-1}_{n}  \ell \frac{\d}{\d w} \left(\frac{v_{n} - w}{w a_{n-1}(w)}\right) \bigg|_{w = v_n} = - \frac{\alpha^{-1}_{n} \ell}{v_n a_{n-1}(v_n)} \neq 0.\qedhere
\end{equation*}
\end{proof}

\begin{lem}\label{lem:R-image}
The operators $\R^*$ and $(\R^{-1})^*$ map $\V{n}(\vec v, \theta)$ onto itself. 
\end{lem}

\begin{proof}
It is enough to prove the statement for $\R^*$.
The argument is in fact essentially the same as the one in the proof of Lem.~\ref{lem:Q-spaces}.
Using the definition \eqref{eq:defR} and repeating that argument,
\[\R^* e^{n}_{k, \ell}(z) = \sum_{x \in \zz} \left(\frac{u_k(n)}{\theta}\right)^{x} \frac{ (x)_\ell}{2\pi\I}\oint_{\gamma_\rrin}\d w\,\frac{\theta^{x-z}}{w^{x-z + 1}}\psi(w) = \sum_{m = 0}^\ell {\ell \choose m}\frac{\d^m}{\d w^m} \psi(w)\Big|_{w=u_k(n)} e^{n}_{k, \ell-m}(z).\]
As a consequence, $\R^*$ maps $\V{n}(\vec v, \theta)$ into itself, and the matrix of this map with respect to the basis $\B{n}(\vec v, \theta)$ is block diagonal, with blocks indexed by the index $k$ in \eqref{eq:basis}, and these blocks are triangular.
Moreover, the diagonal entries in the $k$-th block are given by the coefficients with $m=0$ in the above sum for each $\ell$; this coefficient equals $\psi(u_k(n))$, and Assum.~\ref{assum:apsi} guarantees that $\psi(u_k(n)) \neq 0$ (since by the assumption the function $\psi$ is analytic and non-zero in an annulus containing all speeds $v_i$, and hence all values $u_k(n)$).
This implies that the matrix of $\R^*$ with respect to the basis $\B{n}(\vec v, \theta)$ is non-singular, and hence that the map is onto.
\end{proof}

\begin{lem}\label{lem:biorth-unique}
There is a unique family of functions $(\Phi^n_k)_{k=0,\dotsc,n-1}$ satisfying the properties \bioneref--\bitworef.
\end{lem}

\begin{proof}
Lem.~\ref{lem:R-image} suggests that solving the problem \bioneref--\bitworef is equivalent to solving the following one: find functions $\big(\bar\Phi_k^{n}\big)_{k=0,\dots,n-1}$ such that
\begin{itemize}
\item[($\bar\star$)] $(Q^*_{(n-\ell, n]})^{-1} \bar\Phi_k^{n}(y_{n-k})=\uno{k=\ell}$ for each $k,\ell=0,\dotsc,n-1$.
\item[($\bar\star\bar\star$)] $\spanning\{x \in \zz \longmapsto \bar\Phi^{n}_{k}(x) : 0 \leq k < n \} = \V{n}(\vec v, \theta)$.
\end{itemize}
Indeed, we have $Q_{(k, n]}^{-1}(x,y_{k}) = \R^{-1}\Psi^n_{n - k}(x)$ (see \eqref{eq:defPsink}), so $(Q^*_{(n-\ell, n]})^{-1} \bar\Phi_k^{n}(y_{n-k})$ is equal to  $\sum_{x\in\zz}\R^{-1}\Psi^n_{n-k}(x)\bar\Phi_k^n(x)$ and thus the solutions to these two problems are related by the one-to-one correspondence $\bar{\Phi}^n_k = \R^* \Phi^n_k$. 
Then the lemma will follow if we prove that this new problem has a unique solution.

Property \bitworefbar means that the solution $(\bar \Phi^n_k)_{k=0,\dotsc,n-1}$ which we are looking for has to be given as
\begin{equation*}
\textstyle\bar\Phi^{n}_{k}(x) = \sum_{f \in \B{n}} \bar W(k, f) f(x).
\end{equation*}
for some square matrix $\bar W = (\bar W(k, f) : 0 \leq k < n, f \in \B{n})$, where we write $\B{n}$ for $ \B{n}(\vec v, \theta)$.
With this, showing that $\big(\bar\Phi_k^{n}\big)_{k=0,\dots,n-1}$ satisfies  \bionerefbar--\bitworefbar reduces to proving that the matrix $\bar W$ can be chosen so that property \bionerefbar is satisfied and, moreover, that the matrix is uniquely characterized by that property.

From the above formula for $\bar\Phi^{n}_{k}$ we have for each $k,\ell=0,\dotsc,n-1$ that
\[\textstyle(Q^*_{(n-\ell, n]})^{-1} \bar\Phi_k^{n} (y_{n-k}) = \sum_{f \in \B{n}} \bar W (k, f) (Q^*_{(n-\ell, n]})^{-1} f(y_{n-k})\]
For a fixed $k$, consider the square matrix $\big(F_{f, \ell}\big)_{f \in \B{n}, 0 \leq \ell < n}$ with entries $F_{f,\ell} = (Q^*_{(n-\ell, n]})^{-1} f(y_{n-k})$, so that the identity $(Q^*_{(n-\ell, n]})^{-1} \bar\Phi_k^{n} (y_{n-k})=\uno{k=\ell}$ can be written as $\bar W F=I$ ($I$ being the identity matrix of size $n$).
Lem.~\ref{lem:Q-spaces} implies that $F$ is a non-singular matrix, so the matrix $\bar W$ satisfying this identity is unique.
\end{proof}

By analogy with \eqref{eq:KernelK-particular} we use now the solution $(\Phi^n_k)_{k=0,\dotsc,n-1}$ of \bioneref--\bitworef to define the (extended) kernel
\begin{equation}\label{eq:K-schutz}
K(n_i,x_i;n_j,x_j)=-Q_{(n_i, n_j]}(x_i,x_j)\uno{n_i<n_j}+\sum_{k=1}^{n_j}\Psi^{n_i}_{n_i-k}(x_i)\Phi^{n_j}_{n_j-k}(x_j)
\end{equation}
for $n_i,n_j\in\set{N}$ and $x_i,x_j\in\zz$, which is our main object of interest.
In what follows section we will obtain exact formulas for the functions $\Phi^n_k$, which will yield explicit expressions for this kernel.

\begin{rem}\label{rem:caterpillars}
If we take $\psi(w) = \varphi(w)^t$ and $a_\ell(w) = \varphi(w)^{L_{\ell}-1}$, $\ell\in\set{N-1}$, then $Q_{(m, n]}=\alpha_{m+1}\dotsm\alpha_{n}\tts\phi^{(m,n)}$ for $\phi^{(m,n)}$ as given in \eqref{eq:phi-caterpillars}, while the functions $\Psi^n_{n-k}$ and $\Phi^n_{n-k}$ coincide with those defined in \eqref{eq:Psi-caterpillars} and the corresponding biorthogonalization problem, after dividing the first one by $\alpha_{n-k+1}\dotsm\alpha_{n}$ and multiplying the second one by the same factor.
This implies that, with these choices, the kernel $K(n_i,x_i;n_j,x_j)$ defined in \eqref{eq:K-schutz} equals the one appearing in \eqref{eq:KernelK-particular} with an additional conjugation $\prod_{\ell=1}^{n_j}\alpha_\ell/\prod_{\ell=1}^{n_i}\alpha_\ell$, and such a conjugation does not change the value of the Fredholm determinant in \eqref{eq:M_formula}.
In view of this and Lem. \ref{lem:TASEP_and_caterpillars}, the abstract setting of this section will allow us to study models such as the general systems of interacting caterpillars with right Bernoulli jumps introduced in Sec.~\ref{sec:caterpillars-rB}.
\end{rem}

\subsection{The boundary value problem}\label{sec:bdvpb}

Our goal now is to derive an explicit formula for the functions $(\Phi^n_k)_{k=0,\dotsc,n-1}$ satisfying the properties \bioneref--\bitworef. 
We will prove in Thm.~\ref{thm:h_heat_Q} that these functions are uniquely defined via the solutions of the boundary value problem
\ifthenelse{\boolean{final}}{%
  \mathtoolsset{showonlyrefs=false}
}{%
}
\begin{subnumcases}{\label{bhe}}
(Q_{n - \ell}^*)^{-1} h^{n}_{k}(\ell,z) = h^{n}_{k}(\ell+1,z), &  $\ell<k,\,z \in \zz$,\label{bhe1}\\ 
h^{n}_{k}(k,z) = (\theta / v_{n - k})^{y_{n-k}-z}, & $z \in \zz$,\label{bhe2}\\ 
h^n_k(\ell,y_{n-\ell}) = 0, & $\ell<k$,\label{bhe3} \\
\spanning\{x \in \zz \longmapsto h^n_k(\ell,x) : \ell \leq k < n \} \subseteq \V{n-\ell}(\vec v, \theta),\quad & $0 \leq \ell < n$, \label{bhe4}
\end{subnumcases}
for fixed $0\leq k<n$. 
\ifthenelse{\boolean{final}}{%
  \mathtoolsset{showonlyrefs=true}
}{%
}
Note that we are looking for solutions in the particular spaces appearing in \eqref{bhe4} (recall their definition in \eqref{eq:space}). 
We have opted here for a slightly weaker version of the boundary value problem compared to \cite{fixedpt, TASEPgeneral}, as we consider inclusions rather than equalities of the sets in \eqref{bhe4}. This will simplify the proof of Lem.~\ref{lem:Phi-formula}, where we show that particular functions satisfy all the conditions in \eqref{bhe}.
On the other hand, we will prove in the next lemma that if the problem \eqref{bhe} has a solution, then the inclusions are necessarily equalities.
In this lemma we also show that the problem \eqref{bhe} has at most a unique solution; existence will be proved later in Sec.~\ref{sec:rw}.
 
\begin{lem}\label{lem:unique-solution}
The boundary value problem \eqref{bhe} has at most one solution. 
Moreover, if $\big(h^n_k(\ell,\cdot)\big)_{0\leq\ell\leq k}$ solves \eqref{bhe}, then the inclusion in \eqref{bhe4} is for this solution an equality. 
\end{lem}

\begin{proof}
Assume that there exist two solutions of \eqref{bhe}, which we denote by $h^n_k(\ell, z)$ and $\bar{h}^n_k(\ell, z)$, $0\leq\ell\leq k$. We set $g^n_k(\ell, z) = h^n_k(\ell, z) - \bar h^n_k(\ell, z)$, which satisfies \eqref{bhe1} and \eqref{bhe4}.
We are going to show by induction, backwards in $\ell$, that $g^{n}_{k}(\ell,z) = 0$ for all $0\leq\ell\leq k$ and $z \in \zz$. 
The case $\ell = k$ follows directly from \eqref{bhe2}.
Now assume that $g^{n}_{k}(\ell+1,z) = 0$ for some $0 \leq \ell \leq k-1$, so that we need to prove that $g^{n}_{k}(\ell,z) = 0$.

The functions $h^n_k(\ell, z)$ and $\bar h^n_k(\ell, z)$ are in $\V{n-\ell}(\vec v, \theta)$, so the same is true for $g^n_k(\ell, z)$. Thus, by \eqref{eq:newbasis}, we can write 
\begin{equation}\label{eq:g-basis}
\textstyle g^n_k(\ell, z) = \sum_{p = 1}^{\nu(n - \ell)} \sum_{q = 0}^{\beta_p(n - \ell)-1} b_{p, q} e^{n - \ell}_{p, q}(z),
\end{equation}
for some real constants $b_{p,q}$. From \eqref{bhe1} and the induction hypothesis we have $(Q_{n - \ell}^*)^{-1} g^{n}_{k}(\ell,z) = g^n_k(\ell+1, z) = 0$, which can be written explicitly as 
\begin{equation}\label{eq:g-expansion}
\textstyle\sum_{p = 1}^{\nu(n - \ell)} \sum_{q = 0}^{\beta_p(n - \ell)-1} b_{p, q} (Q_{n - \ell}^*)^{-1} e^{n - \ell}_{p, q}(z) = 0.
\end{equation}
From Lem.~\ref{lem:Q-spaces} we conclude that $(Q_{n - \ell}^*)^{-1} e^{n - \ell}_{p, q}(z) = \sum_{m = 0}^{q} c^{n-\ell}_{p}(q, m) e^{n - \ell}_{p, m}(z)$. 
Hence, after an interchange of summations, \eqref{eq:g-expansion} can be rewritten as
\begin{equation*}
\textstyle\sum_{p = 1}^{\nu(n - \ell)} \sum_{m = 0}^{\beta_p(n - \ell)-1} \sum_{q = m}^{\beta_p(n - \ell)-1} b_{p, q} c^{n-\ell}_{p}(q, m) e^{n - \ell}_{p, m}(z) = 0,
\end{equation*}
which means that 
\begin{equation}\label{eq:system}
\textstyle\sum_{q = m}^{\beta_p(n - \ell)-1} b_{p, q} c^{n-\ell}_{p}(q, m) =0
\end{equation}
for all $p$ and $m$. 
Now we fix $p$ and consider \eqref{eq:system} as a system of equations in the unknowns $b_{p, q}$, $0 \leq q < \beta_p(n - \ell)$, for the given matrix of coefficients $c^{n-\ell}_{p}(q, m)$, $0 \leq m < \beta_p(n - \ell)$, $m \leq q < \beta_p(n - \ell)$.
 
If $v_{n-\ell} \neq u_p(n - \ell)$, then Lem.~\ref{lem:Q-spaces} implies that the matrix $c^{n-\ell}_{p}(q, m)$ in the system \eqref{eq:system} is triangular and non-singular, and we conclude that $b_{p, q} = 0$ for all $q$. If $v_{n-\ell} = u_p(n - \ell)$, then Lem.~\ref{lem:Q-spaces} yields $c^{n-\ell}_{p}(q, q) = 0$, and the system of equations \eqref{eq:system} turns to $\sum_{q = m + 1}^{\beta_p(n - \ell)-1} b_{p, q} c^{n-\ell}_{p}(q, m) =0$. The matrix of coefficients restricted to $0 \leq m < \beta_p(n - \ell)-1$, $m+1 \leq q < \beta_p(n - \ell)$ is triangular and non-singular, which implies again that $b_{p, q} = 0$ for all $1 \leq q < \beta_p(n - \ell)$.

Now consider the case of $p=p^*$ where $p^*$ is such that $v_{n-\ell} = u_{p_*}(n-\ell)$.
From the argument in the preceding paragraph we conclude that the expansion \eqref{eq:g-basis} simplifies to 
\begin{equation*}
g^n_k(\ell, z) = b_{p_*, 0} \bar e_{p_*, 0}(z).
\end{equation*} 
From \eqref{bhe3} we have $g^n_k(\ell, y_{n-\ell}) = 0$, and since $\bar e_{p_*, 0}(y_{n-\ell}) \neq 0$, we conclude that $b_{p_*, 0} = 0$.
This finishes the proof of $g^n_k(\ell, z) = 0$ for all $z \in \zz$.

We have proved that there is a unique solution $h^n_k(\ell, z)$ of \eqref{bhe}, so what is left to prove that the inclusion in \eqref{bhe4} is an equality for any solution. 
We will proceed by contradiction, so let us assume that there is $\ell_* < n-1$ such that the inclusion in \eqref{bhe4} for $\ell_*$ is strict.
We can take $\ell_*$ to be maximal, i.e., we can assume that \eqref{bhe4} holds with equality for $\ell_* < \ell < n$. 
The functions $h^n_k(\ell_*,x)$, $\ell_* \leq k < n$, are necessarily linearly dependent, so
\begin{equation}\label{eq:h-are-dependent}
\textstyle\sum_{k=\ell_*}^{n-1} a_k h^n_k(\ell_*, z) = 0 \qquad \forall ~z \in \zz
\end{equation}
for some real values $a_k$, not all equal $0$. Using the integral formula \eqref{eq:Q-inverse} and the exact value \eqref{bhe2} we can compute $(Q_{n - \ell_*}^*)^{-1} h^n_{\ell_*}(\ell_*, z) = 0$. Applying then $(Q_{n - \ell_*}^*)^{-1}$ and using \eqref{bhe1} yields 
\begin{equation}
\textstyle\sum_{k=\ell_*+1}^{n-1} a_k h^n_k(\ell_*+1, z) = 0 \qquad \forall ~z \in \zz.
\end{equation}
Since $\ell_*$ is maximal, the span of the functions $h^n_k(\ell_*+1, z)$ is $\V{n-\ell_* - 1}(\vec v, \theta)$ and the preceding identity may hold only if $a_k = 0$, $\ell_* < k < n$. Then \eqref{eq:h-are-dependent} simplifies to $ h^n_{\ell_*}(\ell_*, z) = 0$, which is a contradiction due to \eqref{bhe2}.
 \end{proof}
 
Now we are ready to solve the biorthogonalization problem \bioneref--\bitworef. 

\begin{thm}\label{thm:h_heat_Q}
Let $h^n_k(\ell, z)$, $0\leq\ell\leq k$, $z\in\zz$, be the unique functions satisfying \eqref{bhe}. Then the unique solution of the biorthogonalization problem \bioneref--\bitworef with respect to $(\Psi^n_k)_{k=0,\dotsc,n-1}$ is given by $(\Phi^n_k)_{k=0,\dotsc,n-1}$ with 
\begin{equation}\label{eq:h_heat_Q}
\Phi^n_k(x)= (\R^*)^{-1}h^n_k(0,x).
\end{equation}
\end{thm}

\begin{proof}
The argument is similar to the proof of biorthogonality in \cite[Thm.~2.2]{fixedpt}. To prove \bioneref we write
\begin{align*}
  \textstyle\sum_{x\in\zz}\Psi^n_\ell(x)\Phi^n_k(x)&\textstyle= \sum_{x\in\zz}\R Q_{(n-\ell, n]}^{-1}(x,y_{n-\ell})(\R^*)^{-1}h^{n}_{k}(0,x) \\
  &\textstyle=\R^* (Q_{(n-\ell, n]}^{-1})^*(\R^*)^{-1}h^n_k(0,y_{n-\ell})
  =(Q_{(n-\ell, n]}^{-1})^* h^n_k(0,y_{n-\ell}).
\end{align*}
If $\ell \leq k$, then \eqref{bhe1} allows to write this expression as $h^n_k(\ell,y_{n-\ell})$, which according to \eqref{bhe2}-\eqref{bhe3} equals $\uno{k = \ell}$.
If $\ell > k$ we write $(Q_{(n-\ell, n]}^{-1})^* = (Q_{(n-\ell, n-k]}^{-1})^* (Q_{(n-k, n]}^{-1})^*$ and \eqref{bhe1} yields $(Q_{(n-\ell, n]}^{-1})^* h^n_k(0,y_{n-\ell}) = (Q_{(n-\ell, n-k]}^{-1})^* h^n_k(k,y_{n-\ell})=(Q_{(n-\ell, n-k)}^{-1})^*(Q_{n-k}^{-1})^* h^n_k(k,y_{n-\ell})$.
Using \eqref{bhe2} we have $(Q_{n-k}^{-1})^* h^n_k(k,y_{n-\ell})=\sum_{z \in \zz} (Q_{(n-\ell, n-k]}^{-1})^*(y_{n-\ell}, z) (\theta / v_{n - k})^{y_{n-k}-z}$, and this vanishes using \eqref{eq:Q-inverse} after a simple computation, so $(Q_{(n-\ell, n]}^{-1})^* h^n_k(0,y_{n-\ell}) = 0$.

Now we turn to \bitworef.
From \eqref{bhe4} we have $\spanning\{x \in \zz \longmapsto h^n_k(0,x) : \ell \leq k < n \} = \V{n}(\vec v, \theta)$ and Lem.~\ref{lem:R-image} implies that the same holds if we convolve the functions with $(\R^*)^{-1}$. 
\end{proof}

\subsection{Main result: representation in terms of random walk hitting times}\label{sec:rw}

\subsubsection{Preliminaries}

Let, for $\ell\geq1$,
\begin{equation}\label{eq:Q-ring}
\Qo_\ell(x,y) = \frac{1}{2\pi\I}\oint_{\gamma_\rrin}\d w\,\frac{\theta^{x-y-1}}{w^{x-y}} \frac{v_\ell - \theta}{v_\ell - w} = \frac{v_\ell - \theta}{\theta} (\theta / v_\ell)^{x-y} \uno{x > y},
\end{equation}
and for $0<\ell\leq n$
\begin{equation}\label{eq:Q-ring-many}
\Qo_{(\ell, n]}(x,y) = \Qo_{\ell+1} \cdots \Qo_n(x,y) = \frac{1}{2\pi\I}\oint_{\gamma_\rrin}\d w\,\frac{\theta^{x-y- n + \ell}}{w^{x-y - n + \ell+1}} \prod_{i = \ell+1}^n \frac{v_\ell - \theta}{v_i-w},
\end{equation}
which coincide with the functions \eqref{eq:Q-many} if we set $a_i \equiv 1$ and $\alpha_i = 1$ for all $i$.
Then the kernels \eqref{eq:Q} and \eqref{eq:Q-many} can be written as
\begin{equation}\label{eq:different-Qs}
Q_m = \Qo_m A_{m-1}, \qquad\qquad Q_{(k, m]} = \Qo_{(k, m]} A_{k} \cdots A_{m-1}
\end{equation}
for $m\in\set{N}$ and $0\leq k<m$, with
\begin{equation}\label{eq:A-def}
A_\ell(x,y) = \frac{1}{2\pi\I}\oint_{\gamma_\rrin}\d w\,\frac{\theta^{x-y}}{w^{x-y+1}} \frac{a_{\ell}(w)}{a_{\ell}(\theta)},
\end{equation}
$\ell=0,\dotsc,N-1$.
Note that, in view of \eqref{eq:convention0},
\[A_0=I\]
(and thus $Q_1=\Qo_1$).
We also have
\begin{equation}\label{eq:A-inverse}
A_\ell^{-1}(x,y) = \frac{1}{2\pi\I}\oint_{\gamma_\rrin}\d w\,\frac{\theta^{x-y}}{w^{x-y+1}} \frac{a_{\ell}(\theta)}{a_{\ell}(w)}.
\end{equation}
Note that, since the contour $\gamma_r$ does not include any of the $v_i$'s, $\Qo_{(\ell, n]}(x,y)$ vanishes whenever $y>x+\ell-n$.
Now let
\begin{equation}\label{eq:Q-ring-many-extension}
\bar \Qo_{(\ell, n]}(x,y) = -\frac{1}{2\pi\I}\oint_{\Gamma_{\vec v}}\d w\,\frac{\theta^{x-y- n + \ell}}{w^{x-y - n + \ell+1}} \prod_{i = \ell+1}^n \frac{v_i - \theta}{v_i-w},
\end{equation}
where $\Gamma_{\vec v}$ is a simple, positively oriented contour enclosing all the $v_i$'s but not the origin. We denote for $1 \leq m \leq n$
\begin{equation*}
(\vec v)_{m}^n = (v_m, \ldots, v_n),
\end{equation*}
and we claim that, for fixed $x\in\zz$,
\begin{equation}
\bar\Qo_{(\ell,n]}(x,\cdot)\in\V{n-\ell}((\vec v)_{\ell+1}^n, \theta)\qqand\Qo_{(\ell, n]}(x,y) = \bar \Qo_{(\ell, n]}(x,y) \quad\forall\, y<x;
\end{equation}
in this sense, we think of $\bar\Qo(x,\cdot)$ as an extension of $\Qo(x,\cdot)$ to $\V{n-\ell}((\vec v)_{\ell+1}^n, \theta)$.
To see the identity simply note that if $x > y$ then the residue at infinity of the integrand in \eqref{eq:Q-ring-many} vanishes and thus by Cauchy's formula, \eqref{eq:Q-ring-many} equals \eqref{eq:Q-ring-many-extension}.
That $\bar\Qo_{(\ell,n]}(x,\cdot)\in\V{n-\ell}((\vec v)_{\ell+1}^n,\theta)$ also follows from Cauchy's formula, since the integral in \eqref{eq:Q-ring-many-extension} is a sum of residues computed at the different values among $v_{\ell+1},\dotsc,v_n$.
Moreover, one can readily compute
\begin{equation}\label{eq:QobarQoinv}
\Qo^{-1}_{(k, n]} \bar\Qo_{(\ell, n]}=\bar\Qo_{(\ell, n]}\Qo^{-1}_{(k,n]}=\bar\Qo_{(\ell,k]}\;\;\text{for}\;\;\ell<k,
\qquad\Qo^{-1}_{(k, n]} \bar\Qo_{(\ell, n]}=\bar\Qo_{(\ell, n]}\Qo^{-1}_{(k,n]}=0\;\;\text{for}\;\;\ell\geq k.
\end{equation}

Next we introduce a new kernel
\begin{equation}\label{eq:Q-underline}
\Qpl_{\ell}(x,y) = \Qo_{\ell} A_{\ell}(x,y) = A_{\ell-1}^{-1}Q_\ell A_\ell(x,y)
=\frac{\alphapl_\ell}{2\pi\I}\oint_{\gamma_\rrin}\d w\,\frac{\theta^{x-y}}{w^{x-y}}\frac{a_{\ell}(w)}{v_\ell - w},
\end{equation}
$\ell\in\set{N-1}$, where 
\begin{equation}\label{eq:alphapl}
\alphapl_\ell = \frac{v_\ell-\theta}{a_{\ell}(\theta)\theta} = \frac{1}{\sum_{i \in \zz} (\theta / v_\ell)^i q_{\ell}(i)},
\end{equation}
which is well-defined thanks to Assump.~\ref{assum:q}\ref{item:q2}. 
As above we also write $\Qpl_{(\ell,n]}= \Qpl_{\ell+1}\dotsm\Qpl_n=\Qo_{(\ell,n]} A_{\ell+1} \dotsm A_{n}$, for $\ell\leq n$ in $\set{N-1}$, and we have
\begin{align}%
\Qpl_{(\ell, n]}(x,y) & = \frac{1}{2\pi\I}\oint_{\gamma_\rrin}\d w\,\frac{\theta^{x-y}}{w^{x-y - n + \ell + 1}} \prod_{i = \ell+1}^n \frac{\alphapl_i a_{i}(w)}{v_i-w},\\
(\Qpl_{(\ell,n]})^{-1}(x,y) &= \frac{1}{2\pi\I}\oint_{\gamma_\rrin}\d w\,\frac{\theta^{x-y}}{w^{x-y + n - \ell + 1}} \prod_{i = \ell+1}^n \frac{v_i-w}{\alphapl_i a_{i}(w)}.
\end{align}
We define an extension of $Q_{(\ell, n]}$ to $\V{n-\ell}((\vec v)_{\ell+1}^n, \theta)$ as follows:
\begin{equation}
\barQpl_{(\ell, n]}(x, y) = \bar\Qo_{(\ell, n]} A_{\ell+1} \dotsm A_{n}(x, y)
= -\frac{1}{2\pi\I}\oint_{\Gamma_{\vec v}} \d w\,\frac{\theta^{x-y}}{w^{x-y - n + \ell + 1}} \prod_{i = \ell+1}^n \frac{\alphapl_i a_i(w)}{v_i-w}.\label{eq:bar-Qpl-many}
\end{equation}
Here we choose $\Gamma_{\vec v}$ as above, with the additional restriction that the contour is contained in $\{w\in\cc\!:|w|>r\}$.
Each $a_i(w)$ is analytic in this region, so exactly as for $\bar\Qo_{(\ell,n]}$, we have $\barQpl_{(\ell,n]}(x,\cdot)\in\V{n-\ell}((\vec v)_{\ell+1}^n)$.
Moreover, the definition \eqref{eq:def-a} of $a_\ell$ implies that the coefficient of $w^k$ in the Laurent series for $\prod_{i=\ell+1}^na_i(w)$ vanishes for all $k\geq\sum_{i=\ell+1}^n\kappa_i$, and hence arguing again as for $\bar\Qo_{(\ell,n]}$ we have
\begin{equation}\label{eq:Qbar-equal}
\barQpl_{m}(x,y)=\Qpl_{m}(x,y)\;\; \forall x-y>\textstyle\kappa_m,
\qquad\barQpl_{(\ell,n]}(x,y)=\Qpl_{(\ell,n]}(x,y)\;\; \forall x-y>\textstyle\sum_{i=\ell+1}^n\kappa_i.
\end{equation}
We also use the notation $\Qpl_{[\ell,n]} = \Qpl_{(\ell-1,n]}$, and respectively for the other kernels. 

The kernel $\Qpl_{m}$ is Markov; we let $\Bpl_m$ be the time-inhomogeneous random walk which has transitions from time $m-1$ to time $m$, $m\in\set{N-1}$, with step distribution $\Qpl_m$.
We also define the stopping time
\begin{equation}
\taupl= \min\{m=0,\dotsc,N-1 : \Bpl_m> y_{m+1}\}.\label{eq:deftau-ul}
\end{equation}
Next for $n \geq 1$ and $0 \leq m < n$ define the kernels 
\begin{align}
\SM_{-n}(z_1, z_2) &= a_{n}(\theta) (\R(\Qpl_{[1, n]})^{-1}A_n)^* (z_1, z_2) = a_{n}(\theta) (\R Q^{-1}_{[1,n]})^*(z_1,z_2) \label{eq:SM} \\
& = \frac{1}{2\pi\I}\oint_{\gamma_\rrin}\d w\,\frac{\theta^{z_2-z_1}}{w^{z_2-z_1 + n + 1}} \psi(w)\frac{\prod_{i = 1}^n (v_i-w)}{\prod_{i = 1}^{n}\alphapl_i \prod_{i = 1}^{n-1}a_i(w)}, \nonumber \\
\SN_{(m, n]}(z_1, z_2)  &= a_{n}(\theta)^{-1} \barQpl_{(m, n]} A_{n}^{-1} \R^{-1}(z_1, z_2), \label{eq:SN}\\
& = - \frac{1}{2\pi\I}\oint_{\Gamma_{\vec v}} \d w\,\frac{\theta^{z_1-z_2}}{w^{z_1-z_2 - n + m + 1}}\psi(w)^{-1}\frac{\prod_{i = m+1}^{n} \alphapl_i \prod_{i = m+1}^{n-1} a_i(w)}{\prod_{i=m+1}^n(v_i-w)},
\end{align}
and 
\begin{equation}\label{eq:SN-epi}
\SN^{\epi(\vec y)}_{n}(z_1, z_2) = \ee_{\Bpl_{0} = z_1} \bigl[\SN_{(\taupl, n]}(\Bpl_{\taupl}, z_2) \uno{\taupl < n}\bigr].
\end{equation}

\subsubsection{Main result for the kernel}

The following is our main result:

\begin{thm}\label{thm:kernel-explicit}
Assume $y_{j} - y_{j+1}\geq\kappa_{j}$ for each $j\in\set{N-1}$. Then the kernel $K$ defined in \eqref{eq:K-schutz} can be expressed as
\begin{equation}\label{eq:kernel-explicit}
K(n_i,x_i;n_j, x_j) = -Q_{(n_i, n_j]} (x_i, x_j) \uno{n_i<n_j} + (\SM_{-n_i})^* \SN^{\epi(\vec y)}_{n_j} (x_i, x_j).
\end{equation}
\end{thm}

\begin{rem}\label{rem:main}
\leavevmode
\begin{enumerate}[label=(\alph*)]
\item As explained in \cite[Rem.~5.16(b)]{TASEPgeneral}, the choice of parameter $\theta$ enters simply as a conjugation in the kernel \eqref{eq:kernel-explicit}. Indeed, if $\hat K$ is defined by the same formula but for a different value $\hat \theta$ satisfying Assum.~\ref{assum:q}, then the two kernels are related as $\hat K(n_i,x_i;n_j;x_j) = \left(\prod_{\ell = 0}^{n_j - 1}(\frac{\hat \alpha^+_\ell}{\alphapl_\ell}) / \prod_{\ell = 0}^{n_i - 1}(\frac{\hat\alpha^+_\ell}{\alphapl_\ell})\right) (\frac{\hat \theta}{\theta})^{x_i - x_j} K(n_i,x_i;n_j;x_j)$, where $\hat\alpha^+_\ell$ is defined in the same way as $\alphapl_\ell$ but using the value $\hat \theta$. This conjugation does not change the value of the Fredholm determinant of the kernel, which in our applications is all we are interested in (see Thm.~\ref{thm:biorth_general}).
In applications to scaling limits, the value of $\theta$ needs to be adjusted according to the average density of particles in the system, see Sec.~\ref{sec:caterpillars-mix} for an example.
\item In the case when all values $v_i$ are equal to $1$ and all functions $a_\ell$ are equal to $a$, we recover from \eqref{eq:kernel-explicit} the formula for the kernel obtained in \cite[Thm.~5.15]{TASEPgeneral}. 
\item In the case of discrete time TASEP with right Bernoulli jumps, sequential update and inhomogeneous rates (corresponding to letting the $v_i$'s be different, setting all $a_\ell$'s to be identically $1$, and choosing $\psi(w)=(1+w)^t$), formula \eqref{eq:kernel-explicit} was obtained recently and independently in \cite[Thm.~1.1]{bisiLiaoSaenzZigouras}.
In that article, the authors consider a more general version of the model where the jump rates are also allowed to depend on time in the following way: the $k$-th particle makes a right jump at time $t$ with probability $p_t q_k / (1+p_t q_k)$, for positive parameters $p_t$ and $q_k$ (here we are using the notation of \cite[Thm.~1.1]{bisiLiaoSaenzZigouras}; the values $p_t$ and $q_k$ should not be confused with the jump rates in Sec.~\ref{sec:right-Bernoulli}).
This more general version of the result for sequential TASEP can be recovered in our setting by choosing $v_k = q_k$ and $\psi(w)= \prod_{s = 1}^t (1+p_s w)$, as one can see by comparing \cite[Eqs.~1.1-1.2]{bisiLiaoSaenzZigouras} and \eqref{eq:SM}-\eqref{eq:SN};
the fact that a representation of the form \eqref{eq:mu_fixed} holds in this case (which ensures that our results are applicable) can be proved by composing the dynamics at successive time steps using the convolution result appearing in Appdx. A of \cite{TASEPgeneral}.
See also Rem.~\ref{rem:pfcat}(b).
\end{enumerate}
\end{rem}

The first step in the proof of the Thm.~\ref{thm:kernel-explicit} consists in deriving an explicit formula for the functions $\Phi^n_k$ in terms of the hitting time problem for the random walk $\Bpl_m$.
We turn to this task next; the proof of the theorem appears in Sec.~\ref{sec:formula-proof}.

In the proof it will be convenient to write \eqref{eq:kernel-explicit} in terms of the one-point kernel 
\begin{equation}\label{eq:kernel-explicit-one-point}
K^{(n)} = (\SM_{-n})^* \SN^{\epi(\vec y)}_{n}
\end{equation}
as
\begin{equation}\label{eq:kernel-explicit-2}
K(n_i,x_i;n_j, x_j) =  - Q_{(n_i, n_j]}(x_i, x_j) \uno{n_i<n_j} + Q^{-1}_{[1, n_i]} Q_{[1, n_j]} K^{(n_j)} (x_i, x_j).
\end{equation}

\subsubsection{Explicit formula for $\Phi^n_k$}

For $0 \leq \ell \leq k < n$ and $z\leq y_{n-\ell}- \kappa_{n-\ell}$ we define
\begin{equation}\label{eq:p-as-probability}
p^n_k(\ell, z) = \sum_{\eta > y_{n-k}} \pp_{\Bpl_{n-k-1} = \eta} (\Bpl_m \leq y_{m+1} \text{ for } n-k \leq m < n-\ell, \Bpl_{n-\ell} = z),
\end{equation}
which can also be thought of as a hitting time distribution for the walk $\Bpl_m$ moving backwards in time: more precisely, it corresponds to starting with the walk at $z$ at time $n-\ell$, and moving backwards in time hitting the strict epigraph of $(y_{m+1})_{\geq0}$ exactly at time $m=n-k-1$.
In the next result we will find a $\V{k-\ell + 1}((\vec v)_{n-k}^{n-\ell},\theta)$ extension of this function to all $z \in \zz$, which we denote by $\bar p^n_k(\ell,z)$:

\begin{lem}\label{lem:barp-formula}
Assume $y_{j} - y_{j+1}\geq\kappa_{j}$ for each $j\in\set{N-1}$ and let
\begin{multline}\label{eq:bar-p-formula}
\bar p^n_k(\ell, z) = \sum_{\eta > y_{n-k}} \barQpl_{[n-k, n-\ell]}(\eta, z) \\
 - \uno{\ell < k} \sum_{\eta > y_{n-k}} \sum_{\eta' \in \zz} \Qpl_{n-k} (\eta, \eta') \ee_{\Bpl_{n-k} = \eta'} \bigl[\barQpl_{(\taupl, n-\ell]}(\Bpl_{\taupl}, z) \uno{\taupl < n - \ell}\bigr]
\end{multline}
for $\ell \leq k$ and $z\in\zz$.
Then
\[\bar p^n_k(\ell,\cdot)\in\V{k-\ell + 1}((\vec v)_{n-k}^{n-\ell},\theta)\qqand\bar p^n_k(\ell, z)=p^n_k(\ell, z)\quad\forall~z\leq y_{n-\ell}- \kappa_{n-\ell}.\]
\end{lem}

\begin{proof}
For $\ell = k$ and $z\leq y_{n-k}-\kappa_{n-k}$ we have 
\begin{equation*}
\textstyle p^n_k(k, z) = \sum_{\eta > y_{n-k}} \pp_{B_{n-k-1} = \eta} (\Bpl_{n-k} = z) = \sum_{\eta > y_{n-k}} \Qpl_{n-k} (\eta, z)=\sum_{\eta > y_{n-k}} \barQpl_{n-k} (\eta, z),
\end{equation*}
where the last equality follows from \eqref{eq:Qbar-equal}, since inside the sum we have $\eta-z>\kappa_{n-k}$, showing that $\bar p^n_k(k, z)=p^n_k(k, z)$ for such $z$.
On the other hand, we know already that $\barQpl_{n-k}(\eta,\cdot)\in\V{1}(v_{n-k},\theta)$, i.e., that $\barQpl_{n-k}(\eta,z)=c\tts(v_1/\theta)^{z-\eta}$ for some $c\in\rr$, from which it is straightforward to deduce that $\bar p^n_k(k, z) = \sum_{\eta > y_{n-k}} \barQpl_{n-k}(\eta, z)$ is in $\V{1}(v_{n-k},\theta)$.

Next we turn to the case $\ell < k$.
For $z\leq y_{n-\ell}- \kappa_{n-\ell}$, $p^n_k(\ell, z)$ equals
\begin{equation}
\textstyle \sum_{\eta > y_{n-k}} \sum_{\eta' \leq y_{n-k +1}} \Qpl_{n-k} (\eta, \eta') \pp_{\Bpl_{n-k} = \eta'} (\Bpl_m \leq y_{m+1} \text{ for } n-k < m < n-\ell, \Bpl_{n-\ell} = z). \label{eq:p-k-some-derivation}
\end{equation}
The last probability can be written as (using the stopping time $\taupl$ defined in \eqref{eq:deftau-ul})
\begin{align}
&\textstyle \Qpl_{(n-k, n-\ell]}(\eta', z) - \pp_{\Bpl_{n-k} = \eta'} (\text{hit on } (n - k, n - \ell), \Bpl_{n-\ell} = z) \\
&\;\;\textstyle = \Qpl_{(n-k, n-\ell]}(\eta', z) - \sum_{m = n - k+1}^{n-\ell-1} \pp_{\Bpl_{n-k} = \eta'} (\taupl =m, \Bpl_{n-\ell} = z), \\
&\;\;\textstyle = \Qpl_{(n-k, n-\ell]}(\eta', z) - \sum_{m = n - k+1}^{n-\ell-1} \sum_{\eta'' > y_{m+1}} \pp_{\Bpl_{n-k} = \eta_{k-1}} (\taupl =m, \Bpl_{m} = \eta'') \Qpl_{(m, n-\ell]}(\eta'', z) \\
&\;\;\textstyle = \Qpl_{(n-k, n-\ell]}(\eta', z) - \ee_{\Bpl_{n-k} = \eta'} \bigl[\Qpl_{(\taupl, n-\ell]}(\Bpl_{\taupl}, z) \uno{\taupl < n - \ell}\bigr].
\end{align}
Plugging this into the above expression for $p^n_k(\ell, z)$ gives
\begin{multline}\label{eq:alt-p-formula}
\textstyle p^n_k(\ell, z) = \sum_{\eta > y_{n-k}} \Qpl_{n-k} \bP_{y_{n-k +1}} \Qpl_{(n-k, n-\ell]}(\eta, z) \\
\textstyle - \sum_{\eta > y_{n-k}} \sum_{\eta' \leq y_{n-k +1}} \Qpl_{n-k} (\eta, \eta') \ee_{\Bpl_{n-k} = \eta'} \bigl[\Qpl_{(\taupl, n-\ell]}(\Bpl_{\taupl}, z) \uno{\taupl < n - \ell}\bigr].
\end{multline}
Observe now that for $\eta'>y_{n-k+1}$ we have $\tau^+=n-k$ in the last expectation, which then equals $\Qpl_{(n-k,n-\ell]}(\eta',z)$.
Thus
\begin{align}
p^n_k(\ell, z) &= \textstyle\sum_{\eta > y_{n-k}} \Qpl_{n-k} \bP_{y_{n-k +1}} \Qpl_{(n-k, n-\ell]}(\eta, z) + \sum_{\eta > y_{n-k}} \Qpl_{n-k}\P_{y_{n-k+1}}\Qpl_{(n-k, n-\ell]}(\eta, z)\\
&\hspace{0.4in} - \textstyle\sum_{\eta > y_{n-k}} \sum_{\eta' \in \zz} \Qpl_{n-k} (\eta, \eta') \ee_{\Bpl_{n-k} = \eta'} \bigl[\Qpl_{(\taupl, n-\ell]}(\Bpl_{\taupl}, z) \uno{\taupl < n - \ell}\bigr]\\
&=\textstyle\sum_{\eta > y_{n-k}} \Qpl_{[n-k, n-\ell]}(\eta, z)  \\
&\hspace{0.4in}- \textstyle\sum_{\eta > y_{n-k}} \sum_{\eta' \in \zz} \Qpl_{n-k} (\eta, \eta') \ee_{\Bpl_{n-k} = \eta'} \bigl[\Qpl_{(\taupl, n-\ell]}(\Bpl_{\taupl}, z) \uno{\taupl < n - \ell}\bigr]\\
&=\textstyle\sum_{\eta > y_{n-k}} \barQpl_{[n-k, n-\ell]}(\eta, z)\label{eq:p-formula-pre}  \\
&\hspace{0.4in}- \textstyle\sum_{\eta > y_{n-k}} \sum_{\eta' \in \zz} \Qpl_{n-k} (\eta, \eta') \ee_{\Bpl_{n-k} = \eta'} \bigl[\barQpl_{(\taupl, n-\ell]}(\Bpl_{\taupl}, z) \uno{\taupl < n - \ell}\bigr],\label{eq:p-formula}
\end{align}
where the last equality follows as before from \eqref{eq:Qbar-equal}, because $\eta-z>y_{n-k}-y_{n-\ell}+\kappa_{n-\ell}\geq\sum_{i=n-k}^{n-\ell}\kappa_{i}$ for the first sum while inside the expectation in the second one we have $\Bpl_{\taupl}>y_{n-\taupl+1}$ so $\Bpl_{\taupl}-z\geq y_{n-\taupl+1}-y_{n-\ell}+\kappa_{n-\ell}\geq\sum_{i=n-\taupl+1}^{n-\ell}\kappa_{i}$.
This shows that $\bar p_k^n(\ell,z)=p_k^n(\ell,z)$ for $z\leq y_{n-\ell}-\kappa_{n-\ell}$.
To see that $\bar p_k^n(\ell,\cdot)\in\V{k-\ell+1}((\vec v)_{n-\ell}^{n-k},\theta)$ we proceed as in the case $\ell=k$, using for the first sum on the right hand side of \eqref{eq:bar-p-formula} that $\barQpl_{[n-k,n-\ell]}(\eta,\cdot)\in\V{k-\ell+1}((\vec v)^{n-\ell}_{n-k},\theta)$ while, for the second term, using that $\barQpl_{(\taupl,n-\ell]}(\eta,\cdot)\in\V{n-\ell-\taupl}((\vec v)^{n-\ell}_{\taupl},\theta)$, which is a subspace of $\V{k-\ell+1}((\vec v)^{n-\ell}_{n-k},\theta)$ for $n-k\leq\taupl<n-\ell$.
\end{proof}

Now we can show that the functions $\bar p^n_k(\ell, z)$ yield a solution to the system \eqref{bhe}.

\begin{lem}\label{lem:Phi-formula}
$y_{j} - y_{j+1}\geq\kappa_{j}$ for each $j\in\set{N-1}$ and let, for $0\leq\ell\leq k\leq n\leq N$ and $z\in\zz$,
\begin{align}
h^n_k(\ell, z) &= (A^{-1}_{n-\ell})^* \bar p^n_k(\ell, z) \\
&= \sum_{\eta > y_{n-k}} \barQpl_{[n-k, n-\ell]} A_{n-\ell}^{-1} (\eta, z) \label{eq:h-formula}\\
&\qquad - \uno{\ell < k} \sum_{\eta > y_{n-k}} \sum_{\eta' \in \zz} \Qpl_{n-k} (\eta, \eta') \ee_{\Bpl_{n-k} = \eta'} \bigl[\barQpl_{(\tau, n-\ell]}A_{n-\ell}^{-1} (\Bpl_{\taupl}, z) \uno{\taupl < n - \ell}\bigr].
\end{align}
Then $h^n_k(\ell,z)$ solves \eqref{bhe}.
In particular, the functions \eqref{eq:h_heat_Q} are given by
\begin{align}\label{eq:Phi-formula}
\Phi^n_{k}(x) &= \sum_{\eta > y_{n-k}} \barQpl_{[n-k, n]} A_{n}^{-1} \R^{-1} (\eta, x) \\
&\qquad - \uno{k > 0} \sum_{\eta > y_{n-k}} \sum_{\eta' \in \zz} \Qpl_{n-k} (\eta, \eta') \ee_{\Bpl_{n-k} = \eta'} \bigl[\barQpl_{(\taupl, n]} A_{n}^{-1}\R^{-1}(\Bpl_{\taupl}, x) \uno{\taupl < n}\bigr]. \nonumber
\end{align}
\end{lem}

\begin{proof}
The second equality in \eqref{eq:h-formula} follows from \eqref{eq:bar-p-formula}. Now we show that $h^n_k(\ell,z)$ satisfies \eqref{bhe1}.
We have
\begin{align}\label{eq:p-bar-proof}
&\textstyle(Q_{n - \ell}^*)^{-1} h^n_k(\ell,z) = \sum_{\eta > y_{n-k}}  \barQpl_{[n-k, n-\ell]} A_{n-\ell}^{-1} Q^{-1}_{n - \ell}(\eta, z) \\
&\textstyle\qquad - \sum_{\eta > y_{n-k}} \sum_{\eta' \in \zz} \Qpl_{n-k} (\eta, \eta') \ee_{\Bpl_{n-k} = \eta'} \bigl[\barQpl_{(\tau, n-\ell]} A_{n-\ell}^{-1} Q^{-1}_{n - \ell}(\Bpl_{\taupl}, z) \uno{\taupl < n - \ell}\bigr]. \nonumber
\end{align}
By \eqref{eq:different-Qs}, \eqref{eq:bar-Qpl-many} and \eqref{eq:QobarQoinv} we have that, for $m < n-\ell$, 
\begin{align*}
\barQpl_{[m, n-\ell]} A_{n-\ell}^{-1} Q^{-1}_{n - \ell} &= (\bar \Qo_{[m, n-\ell]} A_{m} \cdots A_{n-\ell}) A_{n-\ell}^{-1}(\Qo^{-1}_{n - \ell} A^{-1}_{n - \ell - 1}) \\
& = \bar \Qo_{[m, n-\ell-1]} A_{m} \cdots A_{n-\ell-2} = \barQpl_{[m, n-\ell-1]} A^{-1}_{n - \ell - 1},
\end{align*}
while for $m = n - \ell$ this expression vanishes. Hence, 
\begin{multline*}
\textstyle(Q_{n - \ell}^*)^{-1} h^n_k(\ell,z) = \sum_{\eta > y_{n-k}} \bigl( \barQpl_{[n-k, n-\ell-1]} A_{n-\ell-1}^{-1} \bigr)(\eta, z) \\
\textstyle- \uno{\ell < k-1} \sum_{\eta > y_{n-k}} \sum_{\eta' \in \zz} \Qpl_{n-k} (\eta, \eta') \ee_{\Bpl_{n-k} = \eta'} \bigl[\barQpl_{(\tau, n-\ell-1]} A_{n-\ell-1}^{-1}(\Bpl_{\taupl}, z) \uno{\taupl < n - \ell-1}\bigr],
\end{multline*}
which is exactly $h^n_k(\ell+1,z)$, so \eqref{bhe1} holds.

To show \eqref{bhe2} we start by using \eqref{eq:bar-Qpl-many} and \eqref{eq:A-inverse} to write
\begin{equation}\label{eq:barQ-and-A}
\textstyle\barQpl_{[m, n-\ell]} A_{n-\ell}^{-1}(x, y) = -\frac{1}{2\pi\I}\oint_{\Gamma_{\vec v}} \d w\,\frac{\theta^{x-y}}{w^{x-y - n + \ell + m}} \prod_{i = m}^{n-\ell} \frac{\alphapl_i a_i(w)}{v_i-w} \frac{a_{n-\ell}(\theta)}{a_{n-\ell}(w)}.
\end{equation}
Since we assumed that $\theta$ is smaller than the $v_i's$, the above contour can be chosen such that $\theta < |w|$ and then using this formula in the case $m=n-k$, $\ell=k$ we get
\begin{align*}
h^n_k(k, z) &= \textstyle \sum_{\eta > y_{n-k}} \barQpl_{n-k} A_{n-k}^{-1} (\eta, z) = - \frac{1}{2\pi\I}\oint_{\Gamma_{v_{n-k}}}\d w\,\frac{\theta^{y_{n-k}-z}}{w^{y_{n-k}-z}} \frac{v_{n-k} - \theta}{(v_{n-k}-w) (w - \theta)}.
\end{align*}
The contour encloses only the simple pole at $v_{n-k}$, and evaluating the residue we get $h^n_k(k, z) = (\theta / v_{n-k})^{y_{n-k}-z}$, which is what we want.

Next we check \eqref{bhe3}.
From \eqref{eq:def-a} we have that that the coefficient of $w^k$ in the Laurent series for $\prod_{i=m}^{n-\ell-1}a_i(w)$ vanishes for all $k\geq\sum_{i=m}^{n-\ell-1}\kappa_i$, and as a consequence that the integrand in \eqref{eq:barQ-and-A} has a vanishing residue at infinity for $x - y > \sum_{i=m}^{n-\ell-1} \kappa_i$.
Cauchy's formula, \eqref{eq:Q-underline}, \eqref{eq:Q-many}, and \eqref{eq:barQ-and-A} then imply that for such $x,y$,
\begin{equation}\label{eq:QpQpQ}
\begin{split}
\textstyle\barQpl_{[m, n-\ell]} A_{n-\ell}^{-1}(x, y) &\textstyle= \frac{1}{2\pi\I}\oint_{\gamma_r}\,\frac{\theta^{x-y}}{w^{x-y - n + \ell + m}} \prod_{i = m}^{n-\ell} \frac{\alphapl_i a_i(w)}{v_i-w} \frac{a_{n-\ell}(\theta)}{a_{n-\ell}(w)} = \Qpl_{[m, n-\ell]} A_{n-\ell}^{-1}(x, y)\\
&\textstyle=\Qpl_{[m, n-\ell)} \Qo_{n-\ell}(x, y).
\end{split}
\end{equation}
Then for $z \leq y_{n-\ell}$ the first term in \eqref{eq:h-formula} equals 
$\sum_{\eta > y_{n-k}} \Qpl_{[n-k, n-\ell)} \Qo_{n-\ell}(\eta, z)$,
because $\eta - z > y_{n-k} - y_{n-\ell} \geq \sum_{i=n-k}^{n-\ell-1} \kappa_i$ by our assumption on the $y_i$'s.
Arguing similarly, since $\Bpl_{\taupl} > y_{\taupl + 1}$, for $z \leq y_{n-\ell}$ the expectation in \eqref{eq:h-formula} equals
$\ee_{\Bpl_{n-k} = \eta'} \bigl[\Qpl_{(\tau, n-\ell)}\Qo_{n-\ell}(\Bpl_{\taupl}, z) \uno{\taupl < n - \ell}\bigr]$.
Hence for such $z$ we have
\begin{multline}
\textstyle h^n_k(\ell, z) = \sum_{\eta > y_{n-k}} \Qpl_{[n-k, n-\ell)}\Qo_{n-\ell}(\eta, z) \\
\textstyle - \sum_{\eta > y_{n-k}} \sum_{\eta' \in \zz} \Qpl_{n-k} (\eta, \eta') \ee_{\Bpl_{n-k} = \eta'} \bigl[\Qpl_{(\tau, n-\ell)}\Qo_{n-\ell} (\Bpl_{\taupl}, z) \uno{\taupl < n - \ell}\bigr].
\end{multline}
Let now $(\Bpl)^{(n-\ell)}$ be the random walk defined like $\Bpl$ except that the step from time $n-\ell-1$ to time $n-\ell$ has distribution $\Qo_{n-\ell}$.
Then in the same way as we rewrote \eqref{eq:p-as-probability} in the form \eqref{eq:p-formula}, we can write the preceding expression as
\[\textstyle h^n_k(\ell, z) = \sum_{\eta > y_{n-k}} \pp_{(\Bpl)^{(n-\ell)}_{n-k-1} = \eta} ((\Bpl)^{(n-\ell)}_m \leq y_{m+1} \text{ for } n-k < m < n-\ell, (\Bpl)^{(n-\ell)}_{n-\ell} = z).\]
The last probability can be written as 
\[\textstyle\sum_{\eta' \leq y_{n-\ell}} \pp_{(\Bpl)^{(n-\ell)}_{n-k-1} = \eta} ((\Bpl)^{(n-\ell)}_m \leq y_{m+1} \text{ for } n-k < m < n-\ell, (\Bpl)^{(n-\ell)}_{n-\ell-1} = \eta')\Qo_{n-\ell}(\eta', z),\]
and setting $z=y_{n-\ell}$ we get the required identity $h^n_k(\ell, y_{n-\ell}) = 0$ because, from \eqref{eq:Q-ring}, $\Qo_{n-\ell}(\eta', y_{n-\ell}) = 0$ for $\eta'-y_{n-\ell}\leq0$.

It remains to prove \eqref{bhe4}, i.e., that for each $0 \leq \ell < n$ the functions $h^n_k(\ell,\cdot)$, $\ell \leq k < n$, span $\V{n-\ell}(\vec v,\theta)$. For this, we use the definitions \eqref{eq:A-def}, \eqref{eq:Q-ring-many-extension} and \eqref{eq:bar-Qpl-many} to compute
\begin{align}
\textstyle\sum_{\eta > y_{m}} \barQpl_{[m, n-\ell]} A_{n-\ell}^{-1} (\eta, z) &\textstyle= \sum_{\eta > y_{m}} \bar\Qo_{[m, n-\ell]} A_{m} \dotsm A_{n-\ell-1}(\eta, z) \label{eq:QA-sum}\\
&\textstyle= - \frac{1}{2\pi\I}\oint_{\Gamma_{\vec v}} \d w\,\frac{\theta^{y_{m}-z} (v_{n-\ell}-\theta)}{w^{y_{m}-z - n+\ell + m} (w - \theta)} \frac{\prod_{i = m}^{n-\ell-1} \alphapl_i a_i(w)}{\prod_{i = m}^{n-\ell} (v_i-w)},
\end{align}
where the contour $\Gamma_{\vec v}$ is such that $|w| > \theta$ (it can be taken such due to Assump.~\ref{assum:q}). Cauchy's formula implies that this is an element of $\V{n - \ell - m + 1}((\vec v)_m^{n-\ell}, \theta)$ as a function of $z$. The functions \eqref{eq:h-formula} can be written as linear combinations of \eqref{eq:QA-sum}, for $n-k \leq m \leq n - \ell$, and hence each function $h^n_k(\ell, \cdot)$ belongs to $\V{k - \ell + 1}((\vec v)_{n-k}^{n-\ell}, \theta)$. Since $\V{k - \ell + 1}((\vec v)_{n-k}^{n-\ell}, \theta) \subset \V{n-\ell}(\vec v, \theta)$, we get the required inclusion \eqref{bhe4}.
\end{proof}

\subsubsection{Proof of Thm.~\ref{thm:kernel-explicit}}
\label{sec:formula-proof}

Consider the one-point kernel $K^{(n)}(z_1,z_2)$
\begin{equation}\label{eq:Kn}
K^{(n)}(z_1,z_2) = K(n,z_1;n,z_2) = \sum_{k=1}^{n}\Psi^{n}_{n-k}(z_1)\Phi^{n}_{n-k}(z_2).
\end{equation}
From the definition \eqref{eq:defPsink}/\eqref{eq:defPsinkext} of the functions $\Psi^{n}_{n-k}$ we readily get 
\begin{equation}\label{eq:K-from-one-to-many}
K(n_i,x_i;n_j,x_j)=-Q_{(n_i, n_j]}(x_i,x_j)\uno{n_i<n_j}+ Q_{(n_i, n_j]} K^{(n_j)}(x_i, x_j),
\end{equation}
where we take $Q_{(m,n]}=Q^{-1}_{[n,m)}$ if $m>n$.
In view of this and the definition \eqref{eq:SM} of $\SM_{-n}$, \eqref{eq:kernel-explicit} will follow if we show that for any $n\in\set{N}$,
\begin{equation}\label{eq:toshow}
K^{(n)}=(\SM_{-n})^*\SN^{\epi(\vec y)}_n.
\end{equation}

Using \eqref{eq:defPsink} and \eqref{eq:h_heat_Q}, we rewrite the right hand side of \eqref{eq:Kn} as
\begin{equation}
\textstyle K^{(n)}(z_1,z_2)=\sum_{k=1}^{n}\Psi^{n}_{n - k}(z_1)\Phi^{n}_{n - k}(z_2)
= \sum_{k=1}^{n} \R Q^{-1}_{[1, n]} G^{(k)}_{0, n} \R^{-1} (z_1,z_2) \label{eq:kernel-second-part}
\end{equation}
with
\begin{equation}\label{eq:defGkn}
\textstyle G^{(k)}_{0, n} (z_1, z_2) = Q^{[1, k]} (z_1, y_k) h^n_{n-k}(0, z_2).
\end{equation}
Let also
\begin{equation}\label{eq:defhatGkn}
\hat{G}^{(k)}_{0, n}(z_1, z_2) = A_{k-1}^{-1} G^{(k)}_{0, n} A_{n}(z_1, z_2)= A_{k-1}^{-1}Q^{[1, k]} (z_1, y_k)\bar p^n_{n-k}(0, z_2),
\end{equation}
where we used \eqref{eq:h-formula}.
Using Lem.~\ref{lem:barp-formula} and \eqref{eq:alt-p-formula} 
together with \eqref{eq:different-Qs} we get, for $z_2\leq y_{n}- \kappa_{n}$, 
\begin{multline}
\textstyle\hat G^{(k)}_{0, n} (z_1, z_2) = Q_{[1, k)} \Qo_k (z_1, y_k) \sum_{\eta > y_{k}} \sum_{\eta' \leq y_{k+1}} \Qpl_k(\eta, \eta') \Qpl_{(k, n]} (\eta', z_2)\\
\textstyle- Q_{[1, k)} \Qo_k (z_1, y_k) (z_1, y_k) \sum_{\eta > y_{k}} \sum_{\eta' \leq y_{k + 1}} \Qpl_{k} (\eta, \eta') \ee_{\Bpl_{k} = \eta'} \bigl[ \Qpl_{(\taupl, n]} (\Bpl_{\taupl}, z_2) \uno{\taupl < n}\bigr]. \label{eq:G-k}
\end{multline}
Recall that $\Qo_{k}(z_1,z_2) = \frac{v_k - \theta}{\theta} (\theta / v_{k})^{z_1 - z_2} \uno{z_1 > z_2}$.
On the other hand, as in \eqref{eq:Q-simple} we have $\Qpl_{k}(z_1,z_2) = \alphapl_k(\theta / v_{k})^{z_1-z_2}$ for $z_1-z_2>\kappa_k$.
Thus, since $y_k-y_{k+1}\geq\kappa_k$ we have, for $\eta'\leq y_{k+1}$,
\begin{equation*}
\textstyle\Qo_{k}(z,y_{k}) \sum_{\eta > y_{k}} \Qpl_{k}(\eta, \eta') = \frac{v_k - \theta}{\theta} (\theta / v_{k})^{z - y_{k}}\tfrac{\theta}{v_{k} - \theta} \alphapl_{k} (\theta / v_{k})^{y_k - \eta'} \uno{z > y_{k}} = \Qpl_{k}(z, \eta') \uno{z > y_{k}}.
\end{equation*}
Using this identity in our last expresion for $\hat G^{(k)}_{0,n}$ we get, still for $z_2\leq y_n-\kappa_n$
\begin{align*}
&\textstyle\hat G^{(k)}_{0,n} (z_1,z_2) = \textstyle Q_{[1, k)} \P_{y_k} \Qpl_{k} \bP_{y_{k + 1}} \Qpl_{(k, n]} (z_1, z_2) \\
&\textstyle\hspace{1.4in}- \sum_{\eta' \leq y_{k + 1}} Q_{[1, k)} \P_{y_k} \Qpl_{k} (z_1, \eta') \ee_{\Bpl_{k} = \eta'} \bigl[ \Qpl_{(\taupl, n]} (\Bpl_{\taupl}, z_2) \uno{\taupl < n}\bigr] \\
&\qquad\textstyle= Q_{[1, k)} \P_{y_k} \Qpl_{[k, n]} (z_1, z_2) 
 - \sum_{\eta' \in \zz} Q_{[1, k)} \P_{y_k} \Qpl_{k} (z_1, \eta') \ee_{\Bpl_{k} = \eta'} \bigl[ \Qpl_{(\taupl, n]} (\Bpl_{\taupl}, z_2) \uno{\taupl < n}\bigr].
\end{align*}
By its definition \eqref{eq:defhatGkn} and using Lem.~\ref{lem:barp-formula}, the left hand side is in $\V{n-k+1}((\vec v)_k^n,\theta)$ as a function of $z_2$.
On the other hand, if we replace $\Qpl_{[k, n]}$ by $\barQpl_{[k, n]}$ and $\Qpl_{(\taupl, n]}$ by $\barQpl_{(\taupl, n]}$ in the last line then the result extends that expression to $\V{n-k+1}((\vec v)_k^n,\theta)$ as a function of $z_2$ (which follows from arguing as in the proof of Lem.~\ref{lem:barp-formula}).
$\V{n-k+1}((\vec v)_k^n,\theta)$ is a finite-dimensional vector space, whence it is easy to see that such an extension is unique, so we conclude that
\[\textstyle\hat G^{(k)}_{0,n} (z_1,z_2) =Q_{[1, k)} \P_{y_k} \barQpl_{[k, n]} (z_1, z_2) 
 - \sum_{\eta' \in \zz} Q_{[1, k)} \P_{y_k} \Qpl_{k} (z_1, \eta') \ee_{\Bpl_{k} = \eta'} \bigl[ \barQpl_{(\taupl, n]} (\Bpl_{\taupl}, z_2) \uno{\taupl < n}\bigr]\]
for all $z_1,z_2\in\zz$.
Using this and \eqref{eq:defhatGkn} in \eqref{eq:kernel-second-part} we get 
\begin{align*}
&K^{(n)}(z_1,z_2)=\textstyle\sum_{k=1}^{n} \R Q^{-1}_{[1, n]}A_{k-1} \hat G^{(k)}_{0, n} A_{n}^{-1}\R^{-1} (z_1,z_2) \\
&\quad\textstyle = \sum_{k=1}^{n} \R Q^{-1}_{[1, n]}A_{k-1} Q_{[1, k)} \P_{y_k} \barQpl_{[k, n]} A_{n}^{-1}\R^{-1} (z_1,z_2) \\
&\qquad\textstyle - \sum_{k=1}^{n}\sum_{\eta' \in \zz} \R Q^{-1}_{[1, n]} A_{k-1} Q_{[1, k)} \P_{y_k} \Qpl_{k} (z_1, \eta') \ee_{\Bpl_{k} = \eta'} \bigl[ \barQpl_{(\taupl, n]} A_{n}^{-1}\R^{-1} (\Bpl_{\taupl}, z_2) \uno{\taupl < n}\bigr].
\end{align*}
From \eqref{eq:Q-underline} we have $A_{k-1} Q_{[1, k)} = \Qpl_{[1, k)}$ (recall $A_0=I$), so using \eqref{eq:SM} and \eqref{eq:SN} we can write
\begin{multline}
K^{(n)}(z_1,z_2)=\textstyle\sum_{k=1}^{n} \SM_{-n} \Qpl_{[1, k)} \P_{y_k} \SN_{[k, n]}(z_1,z_2) \\
\textstyle- \sum_{k=1}^{n}\sum_{\eta' \in \zz} \SM_{-n} \Qpl_{[1, k)} \P_{y_k} \Qpl_{k} (z_1, \eta') \ee_{\Bpl_{k} = \eta'} \bigl[ \SN_{(\taupl, n]} (\Bpl_{\taupl}, z_2) \uno{\taupl < n}\bigr].
\end{multline}
Thus all that is left to prove \eqref{eq:toshow} is to show that $\SN^{\epi(\vec y)}_{n}(\eta,z)$ equals
\[\textstyle\sum_{k=1}^{n}\Qpl_{[1, k)} \P_{y_k} \SN_{[k, n]}(\eta,z)\\
\textstyle- \sum_{k=1}^{n}\sum_{\eta' \in \zz}\Qpl_{[1, k)} \P_{y_k} \Qpl_{k} (\eta, \eta') \ee_{\Bpl_{k} = \eta'} \bigl[ \SN_{(\taupl, n]} (\Bpl_{\taupl}, z) \uno{\taupl < n}\bigr]\]
or, after multiplying by $\R A_n$ on the right, that
\begin{multline}\label{eq:toshow2}
\ee_{\Bpl_0=\eta}\big[\barQpl_{(\taupl,n]}(\Bpl_{\taupl},z)\uno{\taupl<n}\big]=\textstyle\sum_{k=1}^{n}\Qpl_{[1, k)} \P_{y_k} \barQpl_{[k, n]}(\eta,z)\\
\textstyle- \sum_{k=1}^{n}\sum_{\eta'\in\zz}\Qpl_{[1, k)}\P_{y_k}\Qpl_{k} (\eta, \eta') \ee_{\Bpl_{k} = \eta'} \bigl[ \barQpl_{(\taupl, n]} (\Bpl_{\taupl}, z) \uno{\taupl < n}\bigr]
\end{multline}
for all $\eta,z\in\zz$.

Arguing as above, both sides of \eqref{eq:toshow2} are in $\V{n}(\vec v,\theta)$ as functions of $z$ so it is enough to prove the identity for all $\eta\in\zz$ and all $z\leq y_n-\kappa_n$ for which $\barQpl_{[k, n]}$ and $\barQpl_{(\taupl, n]}$ in \eqref{eq:toshow2} can be replaced by $\Qpl_{[k, n]}$ and $\Qpl_{(\taupl, n]}$.
The right hand side becomes, using the strong Markov property,
\begin{multline}
\textstyle\sum_{k=1}^{n}\left(\Qpl_{[1, k)}\P_{y_k} \Qpl_{[k, n]}(\eta,z)-\sum_{\eta'\in\zz}\Qpl_{[1, k)}\P_{y_k}\Qpl_{k} (\eta, \eta') \ee_{\Bpl_{k} = \eta'} \bigl[ \Qpl_{(\taupl, n]} (\Bpl_{\taupl}, z) \uno{\taupl < n}\bigr]\right)\\
\textstyle=\sum_{k=1}^{n}\left(\Qpl_{[1, k)}\P_{y_k} \Qpl_{[k, n]}(\eta,z)-\Qpl_{[1, k)}\P_{y_k}\Qpl_{k} (\eta, \eta') \pp_{\Bpl_{k} = \eta'} \bigl(\Bpl_{n}=z,\,\taupl < n\bigr)\right).
\end{multline}
The first term inside the parenthesis is the probability that the walk $\Bpl_m$ goes from $\eta$ at time $0$ to $z$ at time $n$ and that it is above $y_{k}$ at time $k-1$, while the second term is the probability that the same happens and that the walk hits goes above $\vec y$ again after time $k-1$.
The difference is thus the probability that the walk goes from $\eta$ to $z$, goes above $y_{k}$ at time $k-1$ and stays below $\vec y$ after that, so the sum in $k$ is nothing but last hitting time decomposition of $\pp_{B_0=\eta}(B_n=z,\,\taupl<n)$, which is exactly the left hand side of \eqref{eq:toshow2} with $\barQpl_{(\taupl,n]}$ replaced by $\Qpl_{(\taupl,n]}$.
This yields the desired identity.

\begin{rem}\label{rem:pfcat}
\leavevmode
\begin{enumerate}[label=(\alph*)]
\item The scheme which we have used to prove the theorem is more complicated than the one used for \cite[Thm.~5.15]{TASEPgeneral}, due to the inhomogeneous speeds $v_i$ and, particularly, the inhomogeneous values of $\kappa_i$.
Besides the necessary additional care needed when performing manipulations with the hitting probabilities due to the fact that the random walk is now time-inhomogeneous and the additional technical difficulties coming from the fact that the space $\V{n}(\vec v, \theta)$ is no longer essentially made out of polynomials, an important difficulty here is that the operators $A_k$ cannot be removed from the kernel \eqref{eq:kernel-second-part} by conjugation, which can be done, and simplifies the argument, in the case where they do not depend on $k$.
\item Our proof of Thm.~\ref{thm:kernel-explicit} is based on manipulating the kernel $K^{(n)}(z_1,z_2)$ under the  restriction $z_2\leq y_{n}- \kappa_{n}$ and then extending the resulting formula to all $z_2 \in \zz$.
An alternative derivation of the formula \eqref{eq:kernel-explicit} from \eqref{eq:K-schutz}, which avoids the need to work with a restricted variable and then extending, was developed in the proof of \cite[Prop.~4.6]{bisiLiaoSaenzZigouras} in the setting of right Bernoulli TASEP with inhomogeneous speeds.
The method, which should be applicable in our setting too, uses a clever double induction argument based on proving formulas of the type \eqref{eq:kernel-explicit} for some intermediate kernels. 
\end{enumerate}
\end{rem}

\section{Application to particle systems}
\label{sec:application}

As shown in Secs. 2 and 3 of \cite{TASEPgeneral}, the finite dimensional-distributions of the heads of systems of caterpillars for several TASEP-like dynamics can be written in the form \eqref{eq:mu_fixed} in the case of equal speeds and lengths.
The argument provided there extends without any difference to systems of caterpillars with general speeds and lengths.
A version of the Fredholm determinant formula \eqref{eq:caterpillars}/\eqref{eq:caterpillars-Kt} for these systems can then be obtained from the arguments of this article, by applying Thm.~\ref{thm:kernel-explicit} to the kernel \eqref{eq:KernelK-particular} defined via the functions \eqref{eq:phi-caterpillars}, \eqref{eq:Psi-caterpillars} and the corresponding biorthogonal functions, i.e., by taking $\psi(w) = \varphi(w)^{t}$, $a_i(w) = \varphi(w)^{L_i - 1}$ and $\kappa_i = L_i - 1$ in the setting of Thm.~\ref{thm:kernel-explicit} with the specific choice of $\varphi$ corresponding to each example.
This provides the proof of Props.~\ref{prop:TASEP}-\ref{prop:PushASEP}. 

In this section we study the explicit kernels $\K_t$ appearing in \eqref{eq:caterpillars} for a few particular examples.
In Sec.~\ref{sec:two-speeds} we consider two-speed TASEP with half-periodic initial condition, in continuous and discrete times and with both sequential and parallel updates.
In this case, there are two particle blocks within which the particles have equal speeds.
In particular, for two-speed TASEP in continuous time we recover the formula from \cite{TwoSpeed}.
In Sec.~\ref{sec:caterpillars-mix} we consider a version of discrete time TASEP where particles with sequential and parallel update are mixed, and show that it converges to the KPZ fixed point for general one-sided initial data (with arbitrary particle density).
Finally, in Sec.~\ref{sec:caterpillar-block} we study a variant of TASEP with sequential update where we modify the memory lengths of a macroscopic block of particles located in the bulk of the system.

\subsection{Two-speed variants of TASEP}\label{sec:two-speeds}

In \cite{TwoSpeed} a version of continuous time TASEP is considered where the first $M$ particles have jump rate $\alpha>0$ and the remaining ones have jump rate $1$.
The focus of that paper was the case of $2$-periodic initial state $X_0(i) = 2(M - i)$, $i \geq 1$ (which is chosen so that rate $1$ particles start on the negative even integers and the additional rate $\alpha$ particles are placed to the right of those), for which the associated biorthogonal functions $\Psi^{n}_{k}$ and $\Phi^{n}_{k}$ were computed explicitly, leading to a Fredholm determinant formula for the multipoint distribution of the process.
This formula was further used in that paper to compute the long time scaling limit of the process, which leads to an explicit ``process diagram'' for the model describing how its asymptotic fluctuations depend on the value of $\alpha$ and the characteristic direction used in the scaling.

In this section we will consider the two-speed setting for discrete and continuous time TASEP.
More precisely, we consider the kernel $K$ studied in Sec. \ref{sec:biorth} %
with $a_\ell(w) = (1+w)^\kappa$, $\kappa\in\{0,1\}$ and speeds given by
\begin{equation}\label{eq:TwoSpeeds_speeds-general}
v_i = 
\begin{cases}
\alpha, &1 \leq i \leq M,\\
\beta, &i > M,
\end{cases}
\end{equation}
for some $M\geq1$ and two real parameters $\alpha, \beta > 0$.
The function $\psi$ will initially remain general (subject to the assumptions of Sec. \ref{sec:biorth}).
In the case $\psi(w) = e^{t w}$ and $\kappa=0$, the kernel corresponds to the two-speed version of continuous time TASEP, with the values $\alpha$ and $\beta$ corresponding to the jump rates of the particles in the respective blocks.
In the case $\psi(w) = (1+w)^{t}$ the kernel describes two-speed right Bernoulli TASEP in discrete time, with either sequential ($\kappa =0$) or parallel ($\kappa=1$) update; in this last case, and recalling that the values $v_i$ are equal to $p_i/q_i$, where $p_i$ is the probability of the $i$-th particle making a jump, the choice \eqref{eq:TwoSpeeds_speeds-general} can be written as
\begin{equation}
p_i = 
\begin{cases}
\frac{\alpha}{1 + \alpha}, &1 \leq i \leq M,\\
\frac{\beta}{1 + \beta}, &i > M.
\end{cases}
\end{equation}
By employing other choices of $\psi$ and $a_\ell$ one recovers, in the same way, two-speed versions of the other TASEP variants, but the formulas which we will derive in what follows depend on the specific choice of $a_\ell$, so for simplicity we restrict to this setting.

Our main goal here will be to show how versions of the formulas from \cite{TwoSpeed} for the $2$-periodic initial state $y_i = 2(M - i)$, $i \geq 1$, can be derived in the current setting using our results.
With those formulas in hand, a similar analysis can be performed to recover their process diagrams for general two-speed TASEP variants.
More generally, one could attempt to use these formulas to study the process diagram in the case of general (right-finite) initial data.
We leave this for future work.

To simplify notation in this section, we are going to derive a formula only for the one-point kernel \eqref{eq:kernel-explicit-one-point}; the formula for the multipoint kernel then follows from \eqref{eq:kernel-explicit-2}.
The kernel can be written as
\begin{equation}\label{eq:Kn-RW-comp1}
K^{(n)}(x, x')=(\SM_{-n})^* \P_{y_1} \SN_{[1, n]}(x, x') + (\SM_{-n})^* \bP_{y_1} \SN^{\epi(\vec y)}_{n}(x, x'),
\end{equation}
where the functions \eqref{eq:SM} and \eqref{eq:SN} are given for $n > M$ by
\begin{align}
\textstyle (\SM_{-n})^*(z_1,z_2) &\textstyle = \bigl(\prod_{i = 1}^{n-1}\alphapl_i \bigr)^{-1} \frac{\theta^{z_1-z_2}}{2\pi\I}\oint_{\Gamma_0}\d u\,\frac{[\frac{u(\alpha-u)}{1 + \kappa u}]^M [\frac{u(\beta - u)}{1 + \kappa u}]^{n - M}}{u^{z_1-z_2+2n+1}} (1 + \kappa u) \psi(u), \label{eq:SM-two-speeds}\\
\textstyle \SN_{[1, n]}(z_1,z_2) &\textstyle = - \bigl(\prod_{i = 1}^{n-1}\alphapl_i \bigr) \frac{\theta^{z_1 - z_2}}{2\pi\I}\oint_{\Gamma_{\alpha, \beta}}\d w\,\frac{w^{z_2-z_1 + 2n - 1}}{[\frac{w(\alpha - w)}{1 + \kappa w}]^{M} [\frac{w(\beta - w)}{1 + \kappa w}]^{n-M}} (1 + \kappa w)^{-1} \psi(w)^{-1},\label{eq:SN-two-speeds}
\end{align}
where we used the trivial identity $(1 + u)^\kappa = 1 + \kappa u$ for $\kappa = 0$ or $1$.
Here and throughout the section we use the subscripts in the integration contours to indicate which poles they include.
The first part of the kernel \eqref{eq:Kn-RW-comp1} can be written as 
\begin{equation}\label{eq:K1-new}
\textstyle (\SM_{-n})^* \P_{y_1} \SN_{[1, n]}(x, x')=\frac{\theta^{x-x'}}{(2\pi\I)^2} \oint_{\Gamma_0}\d u \oint_{\Gamma_{\alpha, \beta}}\d w\, \frac{[\frac{u(\alpha-u)}{1 + \kappa u}]^M [\frac{u(\beta - u)}{1 + \kappa u}]^{n - M}}{[\frac{w(\alpha - w)}{1 + \kappa w}]^{M} [\frac{w(\beta - w)}{1 + \kappa w}]^{n-M}} \frac{w^{x'-y_1+2n}}{ u^{x-y_1+2n+1}} \frac{1}{u-w} \frac{(1 + \kappa u) \psi(u)}{(1 + \kappa w) \psi(w)}.
\end{equation}
For the other part it is convenient to decompose the product according to the two blocks:
\begin{equation}\label{eq:K-periodic-1}
\begin{split}
\textstyle(\SM_{-n})^* \bP_{y_1} \SN^{\epi(\vec y)}_{n}(x, x')&=\textstyle \sum_{z \leq y_1} (\SM_{-n})^*(x, z) \ee_{\Bpl_{0}=z}\bigl[\SN_{(\taupl, n]}(\Bpl_\taupl,x')\uno{\taupl<M}\bigr] \\
&\textstyle \qquad + \sum_{z \leq y_1} (\SM_{-n})^*(x, z) \ee_{\Bpl_{0}=z}\bigl[\SN_{(\taupl, n]}(\Bpl_\taupl,x')\uno{M \leq \taupl<n}\bigr],
\end{split}
\end{equation}
Next we will use this decomposition to compute the second term in \eqref{eq:Kn-RW-comp1} explicitly in the case of $2$-periodic initial data.

\subsubsection{$2$-periodic initial data}

Throughout the rest of this section, we fix the choice $y_i=2(M-i)$, $i\geq1$.
We will compute each of the two sums on the right hand side of \eqref{eq:K-periodic-1} separately.
We start by computing the hitting probabilities for the random walk $\Bpl$ introduced in Sec.~\ref{sec:rw} corresponding to our choices.
The distribution of the $\ell$-th step of this random walk is given by the kernel \eqref{eq:Q-underline}, which is equal in our case to
\begin{equation*}
\textstyle \Qpl_{\ell}(x,x') =\frac{\alphapl_\ell}{2\pi\I}\oint_{\Gamma_0}\d w\,\frac{\theta^{x-x'} (1+w)^\kappa}{w^{x-x'}(v_\ell - w)} = \alphapl_{\ell} (\theta / v_\ell)^{x-x'} (\uno{x - x' \geq 1} + \kappa v_\ell \uno{x - x' \geq 2}),
\end{equation*}
where $\alphapl_\ell=\frac{v_\ell-\theta}{(1+\theta)^\kappa\theta}$. For $\zeta > - \log (v_{i} / \theta)$ we define the functions 
\begin{equation*}
\textstyle r_{i}(\zeta) = \log \ee_{\Bpl_{i-1}=0} [e^{\zeta \Bpl_{i}}] = \log \Bigl(\frac{(v_{i} - \theta) (e^\zeta + \kappa \theta)}{(1+\theta)^\kappa (v_{i} e^{\zeta}-\theta) e^\zeta} \Bigr)
\end{equation*}
and $r_{k, \ell}(\zeta) = \sum_{i = k}^\ell r_i(\zeta)$. 
For $\Bpl_{0}=z \leq y_1$ the process $(e^{\zeta \Bpl_{m} - r_{1, m}(\zeta)})_{m \geq 0}$ is a martingale, where we postulate $r_{1, 0}(\zeta) = 0$.
Applying the optional stopping theorem, we get $\ee_{\Bpl_{0}=z} [e^{\zeta \Bpl_{\taupl}-r_{1, \taupl}(\zeta)}] = e^{\zeta z}$ for $\zeta > \max_i\{- \log (v_{i} / \theta)\}$, where the stopping time $\taupl$ is defined in \eqref{eq:deftau-ul}. 
The definition of the initial state $\vec y$ yields $\Bpl_{\taupl} = y_{\taupl + 1} + 1 = 2 (M - \taupl) - 1$, and hence $\ee_{\Bpl_{0}=z} [e^{-2 \zeta \taupl - r_{1, \taupl}(\zeta)}] = e^{\zeta (z - 2 M + 1)}$ for $\zeta > \max_i\{- \log (v_{i} / \theta)\}$. 
Introducing a new variable $u = e^{-\zeta}$, the preceding identity may be written as $\ee_{\Bpl_{0}=z} [u^{\taupl} \prod_{i = 1}^{\taupl} \frac{(1+\theta)^\kappa (v_{i} - \theta u)}{(v_{i} - \theta)(1 + \kappa \theta u)} ] = u^{2 M -z - 1}$.
Defining the functions $p(u) = \frac{(1+\theta)^\kappa u (\alpha - \theta u)}{(\alpha - \theta)(1 + \kappa\theta u)}$ and $q(u) = \frac{(1+\theta)^\kappa u (\beta - \theta u)}{(\beta - \theta)(1 + \kappa \theta u)}$, in the two-speed case \eqref{eq:TwoSpeeds_speeds-general} the preceding identity is equivalent to
\begin{align*} 
&\textstyle \ee_{\Bpl_{0}=z} \bigl[p(u)^{\taupl}\uno{\taupl < M} \bigr] + \ee_{\Bpl_{0}=z} \bigl[ p(u)^M q(u)^{\taupl - M} \uno{\taupl \geq M} \bigr] = u^{2 M -z - 1}.
\end{align*}
This formula can be analytically extended to all non-zero $u \in \cc$ in a neighborhood of the origin. From this identity we get 
\begin{align*}
&\ee_{\Bpl_{0}=z} \bigl[q(u)^{\taupl} \uno{\taupl \geq M} \bigr] = \bigl( \tfrac{q(u)}{p(u)} \bigr)^{M} \Bigl(u^{2 M - z - 1} - \ee_{\Bpl_{0}=z} \bigl[p(u)^\taupl \uno{\taupl < M} \bigr]\Bigr).
\end{align*}

The functions $p(u)$ and $q(u)$ are one-to-one in a neighborhood of $u = 0$, so 
\begin{align}
\label{eq:P1} \pp_{\Bpl_{0}=z} (\taupl = k) &= \textstyle \frac{1}{k!} \frac{\d^k}{\d p(u)^k} u^{2 M -z - 1} \Big|_{u = 0}, \\
\label{eq:P2} \pp_{\Bpl_{0}=z} (\taupl = \ell) &= \textstyle \frac{1}{\ell!} \frac{\d^{\ell}}{\d q(u)^{\ell}} \bigl( \tfrac{q(u)}{p(u)} \bigr)^{M} \Bigl(u^{2 M - z - 1} - \ee_{\Bpl_{0}=z} \bigl[p(u)^\taupl \uno{\taupl < M} \bigr] \Bigr) \Big|_{u = 0}.
\end{align}
From this and Cauchy's integral formula we get $\pp_{\Bpl_{0}=z}(\taupl = k) = \frac{1}{2\pi\I} \oint_{\Gamma_0} \d u \frac{p'(u)}{p(u)^{k+1}} u^{2 M -z - 1}$, so for $k < M$ and $z \leq y_1$ we have
\begin{align}
\pp_{\Bpl_{0}=z}(\taupl = k) &= \textstyle \Bigl(\frac{\alpha - \theta}{(1 + \theta)^\kappa}\Bigr)^{k} \frac{1}{2\pi\I} \oint_{\Gamma_0} \d u \frac{(1 + \kappa \theta u)^{k-1} u^{2M - z - 1}}{[u(\alpha - \theta u)]^{k+1}} ( \alpha - 2 \theta u - \kappa \theta^2 u^2 ) \nonumber \\
&= \textstyle \theta^{-2M + z + k + 1} \Bigl(\frac{\alpha - \theta}{(1 + \theta)^\kappa}\Bigr)^{k} \frac{1}{2\pi\I} \oint_{\Gamma_0} \d u \frac{(1 + \kappa u)^{k-1} u^{2M - z - 1}}{[u(\alpha - u)]^{k+1}} ( \alpha - 2 u - \kappa u^2 ),\qquad\label{eq:prob1}
\end{align}
where in the last identity we rescaled $u$ by $\theta^{-1}$. 
Similarly, for $k \geq M$ and $z \leq y_1$ we have
\begin{align*}
&\textstyle \pp_{\Bpl_{0}=z} (\taupl = k) = \frac{1}{2\pi\I} \oint_{\Gamma_0} \d u \frac{q'(u)}{q(u)^{k - M + 1}} u^{2 M -z - 1} p(u)^{-M}  \\
&\textstyle \hspace{2in} - \sum_{\ell = 0}^{M-1} \pp_{\Bpl_{0}=z} (\taupl = \ell) \frac{1}{2\pi\I} \oint_{\Gamma_0} \d u \frac{q'(u)}{q(u)^{k - M + 1}} p(u)^{\ell-M} \\
&\qquad\textstyle = \left(\frac{\alpha - \theta}{(1 + \theta)^\kappa}\right)^{M} \Bigl( \frac{\beta - \theta}{(1 + \theta)^\kappa} \Bigr)^{k-M} \frac{1}{2\pi\I} \oint_{\Gamma_0} \d u \frac{(1 + \kappa \theta u)^{k-1} u^{2M - z - 1}}{[u(\alpha - \theta u)]^{M} [u(\beta - \theta u)]^{k - M +1}} ( \beta - 2 \theta u - \kappa \theta^2 u^2 ) \\
&\textstyle\qquad\qquad\qquad - \sum_{\ell = 0}^{M-1} \pp_{\Bpl_{0}=z} (\taupl = \ell) \left(\frac{\alpha - \theta}{(1 + \theta)^\kappa}\right)^{M-\ell} \Bigl( \frac{\beta - \theta}{(1 + \theta)^\kappa} \Bigr)^{k-M} \\
&\textstyle\hspace{5cm} \times \frac{1}{2\pi\I} \oint_{\Gamma_0} \d u \frac{(1 + \kappa \theta u)^{k - \ell-1}}{[u(\alpha - \theta u)]^{M - \ell} [u(\beta - \theta u)]^{k - M +1}} ( \beta - 2 \theta u - \kappa \theta^2 u^2 ).
\end{align*}
Using \eqref{eq:prob1}, the sum in $\ell$ in the last term equals $\sum_{\ell = 0}^{M-1} \frac{1}{2\pi\I} \oint_{\Gamma_0} \d v \frac{(1 + \kappa \theta v)^{\ell-1} v^{2M - z - 1}}{[v(\alpha - \theta v)]^{\ell+1}} ( \alpha - 2 \theta v - \kappa \theta^2 v^2 ) \frac{1}{2\pi\I} \oint_{\Gamma_0} \d u \frac{(1 + \kappa \theta u)^{k - \ell-1}}{[u(\alpha - \theta u)]^{M - \ell} [u(\beta - \theta u)]^{k - M +1}} ( \beta - 2 \theta u - \kappa \theta^2 u^2 )$.
Since we are taking $z\leq y_1$, the integral with respect to $v$ vanishes for $\ell < 0$, and thus the sum in the last term can be extended to all $\ell < M$. 
Choosing the integration contour for $v$ so that $|\frac{v(\alpha - \theta v)}{1 + \kappa \theta v}| < |\frac{u(\alpha - \theta u)}{1 + \kappa \theta u}|$, the sum can be computed explicitly and yields
$ \frac{1}{(2\pi\I)^2} \oint_{\Gamma_0} \d u \oint_{\Gamma_{0}} \d v\, \frac{(1 + \kappa \theta v)^{M-1} v^{2M - z - 1}}{[v(\alpha - \theta v)]^{M}} \frac{(1 + \kappa \theta u)^{k - M}}{[u(\beta - \theta u)]^{k - M +1}} \frac{( \alpha - 2 \theta v - \kappa \theta^2 v^2 ) ( \beta - 2 \theta u - \kappa \theta^2 u^2)}{(u-v) (\alpha - \theta u - \theta v - \kappa \theta^2 uv)}$.
Enlarging the $v$ contour so that it now encloses $u$ we pick up a residue from the simple pole at that point, which cancels exactly the first integral on the right hand side above.
Thus we obtain for $k \geq M$ and $z \leq y_1$, and after 
rescaling the integration variables by $\theta^{-1}$,
\begin{align}\label{eq:prob2}
 \pp_{\Bpl_{0}=z} (\taupl = k) &= \textstyle \theta^{-2M + z + k + 1} \left(\frac{\alpha - \theta}{(1 + \theta)^\kappa}\right)^{M} \Bigl( \frac{\beta - \theta}{(1 + \theta)^\kappa} \Bigr)^{k-M}\\
&\textstyle\hspace{-1cm} \times \frac{1}{(2\pi\I)^2} \oint_{\Gamma_0} \d u \oint_{\Gamma_{0,u}} \d v\, \frac{(1 + \kappa v)^{M-1} v^{2M - z - 1}}{[v(\alpha - v)]^{M}} \frac{(1 + \kappa u)^{k - M}}{[u(\beta - u)]^{k - M +1}} \frac{( \alpha - 2 v - \kappa v^2 ) ( \beta - 2 u - \kappa u^2)}{(v - u) (\alpha - u - v - \kappa uv)}.
\end{align}

We can now use the formulas \eqref{eq:prob1} and \eqref{eq:prob2} for the hitting probabilities to compute the two sums in \eqref{eq:K-periodic-1}. 
We start with the first sum. If $\taupl = k$, then the definition of $\vec y$ implies $\Bpl_\taupl = 2(M - k) - 1$ and then for $k < M$,
\eqref{eq:SN-two-speeds} and \eqref{eq:prob1} imply that $\ee_{\Bpl_{0}=z}\bigl[\SN_{(\taupl, n]}(\Bpl_\taupl,x')\uno{\taupl<M}\bigr]$ equals
\[\textstyle-\sum_{k = 0}^{M-1} \bigl(\prod_{i = 1}^{n-1}\alphapl_i \bigr) \frac{\theta^{z -x'}}{(2\pi\I)^2} \oint_{\Gamma_0} \d u \oint_{\Gamma_{\alpha, \beta}}\d w\, \frac{u^{2M - z - 1}}{[\frac{u(\alpha - u)}{1 + \kappa u}]^{k+1}} \frac{\alpha - 2 u - \kappa u^2}{(1 + \kappa u)^{2}} \frac{w^{x' - 2 (M-n)}(1 + \kappa w)^{-1} \psi(w)^{-1}}{[\frac{w(\alpha-w)}{1 + \kappa w}]^{M-k} [\frac{w(\beta-w)}{(1 + \kappa w)}]^{n-M}}.\]
The sum can be extended to all $k < M$ (because for $k < 0$ the integrand does not have a pole at $u = 0$), and after choosing the contours so that $|\frac{u(\alpha - u)}{1 + \kappa u}| < |\frac{w(\alpha - w)}{1+ \kappa w}|$ it can be computed to give
\begin{align*}
&\textstyle \ee_{\Bpl_{0}=z}\bigl[\SN_{(\taupl, n]}(\Bpl_\taupl,x')\uno{\taupl<M}\bigr] \\
&\textstyle \quad = \bigl(\prod_{i = 1}^{n-1}\alphapl_i \bigr) \frac{\theta^{z -x'}}{(2\pi\I)^2} \oint_{\Gamma_{\alpha, \beta}}\d w \oint_{\Gamma_0} \d u \, \frac{u^{2M - z - 1}}{[\frac{u(\alpha - u)}{1 + \kappa u}]^{M}} \frac{w^{x' - 2 (M-n)}}{[\frac{w(\beta-w)}{1+ \kappa w}]^{n-M}} \frac{\alpha - 2 u - \kappa u^2}{(u - w) (\alpha - u - w - \kappa u w)} (1 + \kappa u)^{-1} \psi(w)^{-1}.
\end{align*}
Using this and \eqref{eq:SM-two-speeds}, the first kernel in \eqref{eq:K-periodic-1} is given by an explicit sum, which is computed to be
\begin{equation}
\textstyle\frac{\theta^{x -x'}}{(2\pi\I)^3} \oint_{\Gamma_0} \d u \oint_{\Gamma_{\alpha, \beta}}\d w \oint_{\Gamma_{0,u}} \d v\, \frac{[\frac{v(\alpha-v)}{1 + \kappa v}]^M [\frac{v(\beta - v)}{1 + \kappa v}]^{n - M}}{[\frac{u(\alpha - u)}{1 + \kappa u}]^{M} [\frac{w(\beta-w)}{1+ \kappa w}]^{n-M}} \frac{u w^{x' - 2 (M-n)}}{v^{x-2(M - n)+2}} \frac{\alpha - 2 u - \kappa u^2}{(v-u) (u - w) (\alpha - u - w - \kappa u w)} \frac{(1 + \kappa v) \psi(v)}{(1 + \kappa u) \psi(w)}.\label{eq:K21-new}
\end{equation}

The second term in \eqref{eq:K-periodic-1} can be computed similarly.
For $k \geq M$ one uses \eqref{eq:SN-two-speeds} and \eqref{eq:prob2}, and then extends the resulting finite sum into an infinite one and computes the resulting geometric series to get a formula for $\ee_{\Bpl_{0}=z}\bigl[\SN_{(\taupl, n]}(\Bpl_\taupl,x')\uno{M \leq \taupl <n}\bigr]$.
Then one uses this formula together with \eqref{eq:SM-two-speeds} to compute the sum in $z$ defining that term, which yields a quadruple contour integral which, in turn, after evaluation of the residue of a simple pole yields a triple contour integral,
\[\textstyle\frac{\theta^{x-x'}}{(2\pi\I)^3}\oint_{\Gamma_0}\!\d \eta \oint_{\Gamma_0}\! \d v \oint_{\Gamma_0}\! \d u\,\frac{[\frac{\eta(\alpha-\eta)}{1 + \kappa \eta}]^M [\frac{\eta(\beta - \eta)}{1 + \kappa \eta}]^{n - M}}{[\frac{v(\alpha - v)}{1 + \kappa v}]^{M} [\frac{u(\beta - u)}{1 + \kappa u}]^{n - M}} \frac{v [\frac{\beta-u}{1+\kappa u}]^{x' - 2 (M-n)}}{\eta^{x- 2(M -n) +2}} \frac{( \alpha - 2 v - \kappa v^2 )(1+\kappa u)(1 + \kappa \eta) \psi(\eta)}{(v - u) (\eta - v) (\alpha - u - v - \kappa uv)(1 + \kappa v) \psi(\frac{\beta-u}{1+\kappa u})}.\]
We omit the details of this computation.
The end result, after the change of variables $u\longmapsto\frac{\beta-u}{1+\kappa u}$, is that the second term in \eqref{eq:K-periodic-1} equals
\begin{align}
&\textstyle (1+\kappa \beta)^2 \frac{\theta^{x-x'}}{(2\pi\I)^3}\oint_{\Gamma_0}\d \eta \oint_{\Gamma_0} \d v \oint_{\Gamma_\beta} \d u\,\frac{[\frac{\eta(\alpha-\eta)}{1 + \kappa \eta}]^M [\frac{\eta(\beta - \eta)}{1 + \kappa \eta}]^{n - M}}{[\frac{v(\alpha - v)}{1 + \kappa v}]^{M} [\frac{u(\beta - u)}{1 + \kappa u}]^{n - M}} \frac{v u^{x' - 2 (M-n)}}{\eta^{x- 2(M -n) +2}} \\
& \textstyle \hspace{4cm} \times \frac{( \alpha - 2 v - \kappa v^2 )}{(v-\eta) (\beta - u- v - \kappa u v) (\beta - \alpha - \alpha \kappa u + \beta \kappa v - u + v)} \frac{(1 + \kappa \eta) \psi(\eta)}{(1+\kappa u)(1 + \kappa v) \psi(u)}. \label{eq:K22-new}
\end{align}

We have computed all parts of the kernel \eqref{eq:Kn-RW-comp1} in \eqref{eq:K1-new}, \eqref{eq:K21-new} and \eqref{eq:K22-new}.
The final result is
\begin{align}
&\textstyle K^{(n)}(x, x') = \frac{\theta^{x-x'}}{(2\pi\I)^2} \oint_{\Gamma_0}\d u \oint_{\Gamma_{\alpha, \beta}}\d w\, \frac{(\alpha-u)^M (\beta - u)^{n - M} (1 + \kappa w)^{n}}{(\alpha - w)^{M} (\beta - w)^{n-M} (1 + \kappa u)^n} \frac{w^{x'-2 M + n+1}}{ u^{x-2M + n+2}} \frac{1}{u-w} \frac{(1 + \kappa u) \psi(u)}{(1 + \kappa w) \psi(w)} \\
& \textstyle \hspace{1cm} + \frac{\theta^{x -x'}}{(2\pi\I)^3} \oint_{\Gamma_0} \d v \oint_{\Gamma_{\alpha, \beta}}\d w \oint_{\Gamma_0} \d u\, \frac{(\alpha-v)^M (\beta - v)^{n - M} (1 + \kappa u)^M (1 + \kappa w)^{n-M}}{(\alpha - u)^{M} (\beta - w)^{n-M} (1 + \kappa v)^n} \frac{u^{1-M} w^{x' - M + n}}{v^{x-2 M + n+2}} \\
&\textstyle \hspace{7cm} \times \frac{\alpha - 2 u - \kappa u^2}{(v-u) (u - w) (\alpha - u - w - \kappa u w)} \frac{(1 + \kappa v) \psi(v)}{(1 + \kappa u) \psi(w)} \\
&\textstyle \hspace{1cm} + (1+\kappa \beta)^2 \frac{\theta^{x-x'}}{(2\pi\I)^3}\oint_{\Gamma_0}\d v \oint_{\Gamma_0} \d u \oint_{\Gamma_\beta} \d w\,\frac{(\alpha-v)^M (\beta - v)^{n - M} (1 + \kappa u)^M (1 + \kappa w)^{n-M}}{(\alpha - u)^{M} (\beta - w)^{n-M} (1 + \kappa v)^n} \frac{u^{1-M} w^{x' - M + n}}{v^{x-2 M + n+2}} \\
& \textstyle \hspace{4cm} \times \frac{\alpha - 2 u - \kappa u^2}{(u-v) (\beta - u - w - \kappa u w) (\beta - \alpha + \beta \kappa u - \alpha \kappa w + u - w)} \frac{(1 + \kappa v) \psi(v)}{(1 + \kappa u) (1+\kappa w) \psi(w)} \label{eq:kernel-two-speed}
\end{align}
(where we changed the names of the variables in the last integral).

In the case where a single particle with a different speed is placed to the right of the system, i.e., the case $M=1$, the second term in the above expression vanishes because there is no pole at $u=0$, and the formula simplifies to 
\begin{align}
&\textstyle K^{(n)}(x, x') = \frac{\theta^{x-x'}}{(2\pi\I)^2} \oint_{\Gamma_0}\d u \oint_{\Gamma_{\alpha, \beta}}\d w\, \frac{(\alpha-u) (\beta - u)^{n - 1} (1 + \kappa w)^{n}}{(\alpha - w) (\beta - w)^{n-1} (1 + \kappa u)^n} \frac{w^{x' + n - 1}}{u^{x + n}} \frac{1}{u-w} \frac{(1 + \kappa u) \psi(u)}{(1 + \kappa w) \psi(w)} \\
&\textstyle \hspace{1cm} + (1+\kappa \beta)^2 \frac{\theta^{x-x'}}{(2\pi\I)^3}\oint_{\Gamma_0}\d v \oint_{\Gamma_0} \d u \oint_{\Gamma_\beta} \d w\,\frac{(\alpha-v) (\beta - v)^{n - 1} (1 + \kappa u) (1 + \kappa w)^{n-1}}{(\alpha - u) (\beta - w)^{n-1} (1 + \kappa v)^n} \frac{w^{x' + n - 1}}{v^{x + n}} \\
& \textstyle \hspace{4cm} \times \frac{\alpha - 2 u - \kappa u^2}{(u-v) (\beta - u - w - \kappa u w) (\beta - \alpha + \beta \kappa u - \alpha \kappa w + u - w)} \frac{(1 + \kappa v) \psi(v)}{(1 + \kappa u) (1+\kappa w) \psi(w)}. \label{eq:kernel-two-speed-M1}
\end{align}

\subsubsection{The continuous time case}

For two-speed continuous time TASEP, corresponding to the choice $\psi(w) = e^{t w}$ and $\kappa=0$, it was shown in \cite[Prop.~6]{TwoSpeed} that for $n \geq M+1$ and in the case $\beta=1$, one has $\pp(X_t(n)>a)=\det\bigl(\Id- \bP_{a}  \K_{\BFS}^{(n)}\bP_{a}  \bigr)_{\ell^2(\zz)}$ with the one point kernel $\K_{\BFS}^{(n)}$ given as 
\begin{align}
 K_{\BFS}^{(n)}(x,x') &= \textstyle \frac{1}{(2\pi\I)^2}\oint_{\Gamma_1} \d v \oint_{\Gamma_{0, 1-v}} \d w\; \frac{(w-1)^{n-M} v^{x'+ n- M}}{(v-1)^{n-M} w^{x+ n - M + 1}} \frac{(2v-1) e^{t (w - v)}}{(w-v)(w+v - 1)} \\
 &\qquad + \textstyle \frac{1}{(2\pi\I)^3}\oint_{\Gamma_{\alpha}} \d v \oint_{\Gamma_{1,v}} \d z \oint_{\Gamma_{0, \alpha - v}} \d w\; \frac{(w-1)^{n-M} (w-\alpha)^M}{(z-1)^{n-M} (v-\alpha)^M} \frac{z^{x'+ n - M}}{v^M w^{x+ n- 2M + 1}} \nonumber\\
 &\textstyle \hspace{8cm} \times \frac{(2z-1)(2v -\alpha) e^{t (w - z)}}{(z-v)(z+v-1)(w-v)(w + v -\alpha)} \nonumber 
\end{align}
(here we have shifted the variables $v$ and $z$ by $1$ compared to their formulas).
This formula was derived in that paper slightly differently, by finding explicitly the biorthogonal functions $\Phi^n_k$ and computing $K_{\BFS}^{(n)}(x,x')=\sum_{k=0}^{n-1}\Psi^n_k(x)\Phi^n_k(x')$.

We are going to show now how the formula for the kernel \eqref{eq:kernel-two-speed} which we derived in the preceding subsection can be written in the same way up to conjugation by $\theta^{x}$ (i.e., we will show that $K^{(n)}(x,x') = \theta^{x - x'} K_{\BFS}^{(n)}(x,x')$). 
For brevity, since in the general case the computations are more involved, we will only demonstrate this in the case $M=1$.
In this case computing the residue at the simple pole $z = v$ and changing the order of integration, the preceding formula turns into 
\begin{align}
 &K_{\BFS}^{(n)}(x,x') = \textstyle \frac{1}{(2\pi\I)^2} \oint_{\Gamma_{0}} \d w \oint_{\Gamma_1} \d v\; \frac{(w-1)^{n-1} v^{x'+ n - 1}}{(v-1)^{n-1} w^{x+ n}} \frac{(2v-1) e^{t (w - v)}}{(w-v)(w+v - 1)} \\
 &\hspace{2.2cm} + \textstyle \frac{1}{(2\pi\I)^2} \oint_{\Gamma_{0}} \d w \oint_{\Gamma_{\alpha}} \d v\; \frac{(w-1)^{n-1} (w-\alpha)}{(v-1)^{n-1} (v-\alpha)} \frac{v^{x'+n - 2}}{w^{x+n- 1}} \frac{(2v -\alpha) e^{t (w - v)}}{(w-v)(w + v -\alpha)} \nonumber \\
 &\hspace{2.2cm} + \textstyle \frac{1}{(2\pi\I)^3} \oint_{\Gamma_{0}} \d w \oint_{\Gamma_{1}} \d z \oint_{\Gamma_{\alpha}} \d v\; \frac{(w-1)^{n-1} (w-\alpha)}{(z-1)^{n-1} (v-\alpha)} \frac{z^{x' + n - 1}}{v w^{x+ n - 1}} \frac{(2z-1)(2v -\alpha) e^{t (w - z)}}{(z-v)(z+v-1)(w-v)(w + v -\alpha)}. \nonumber
\end{align}
Computing the residues at the simple poles $v = \alpha$, changing the variable $z$ to $v$ and summing the two double integrals we get
\begin{multline}\label{eq:TwoSpeedKernel}
 K_{\BFS}^{(n)}(x,x') = \textstyle \frac{1}{2\pi\I} \oint_{\Gamma_{0}} \d w\, \frac{(w-1)^{n-1}}{(\alpha-1)^{n-1}} \frac{\alpha^{x'+n - 1}}{w^{x+n}} e^{t (w - \alpha)} \\
  + \textstyle \frac{1}{(2\pi\I)^2} \oint_{\Gamma_{0}} \d w \oint_{\Gamma_{1}} \d v\; \frac{(w-1)^{n-1}}{(v-1)^{n-1}} \frac{v^{x'+n - 1}}{w^{x+ n}} \frac{(2v-1) (w - \alpha) (w + \alpha - 1) e^{t (w - v)}}{(v-\alpha)(v+\alpha-1) (w-v)(w+v - 1)}.
\end{multline}

Next consider our formula \eqref{eq:kernel-two-speed-M1} in this setting,
\begin{align}
&\textstyle K^{(n)}(x, x') = \frac{\theta^{x-x'}}{(2\pi\I)^2} \oint_{\Gamma_0}\d u \oint_{\Gamma_{\alpha, 1}}\d w\, \frac{(1 - u)^{n - 1}}{(1 - w)^{n-1}} \frac{w^{x' + n -1}}{ u^{x + n}} \frac{(u-\alpha) e^{t(u - w)}}{(w-\alpha) (u-w)} \\
&\textstyle \qquad + \frac{\theta^{x-x'}}{(2\pi\I)^3}\oint_{\Gamma_0}\d v \oint_{\Gamma_0} \d u \oint_{\Gamma_1} \d w\,\frac{(1 - v)^{n - 1}}{(1 - w)^{n - 1}} \frac{w^{x' +n - 1}}{v^{x + n}} \frac{(v-\alpha) (2 u-\alpha) e^{t(v - w)}}{(u-\alpha) (u - w + 1 - \alpha) (u-v) ( u + w-1)}.
\end{align}
In the first integral we evaluate the residue at the simple pole $w = \alpha$.
In the second term we swap the integrals with respect to $u$ and $w$ and then evaluate the residue at the simple pole $u = 1-w$. 
This yields a single contour integral plus two double integrals, and after adding those two we get
\begin{multline}
K^{(n)}(x, x') = \textstyle \frac{\theta^{x-x'}}{2\pi\I} \oint_{\Gamma_0}\d u\, \frac{(1 - u)^{n - 1}}{(1 - \alpha)^{n-1}} \frac{\alpha^{x' + n -1}}{ u^{x + n}} e^{t(u - \alpha)} \\
\textstyle \qquad + \frac{\theta^{x-x'}}{(2\pi\I)^2}\oint_{\Gamma_0}\d v \oint_{\Gamma_1} \d w\,\frac{(1 - v)^{n - 1}}{(1 - w)^{n - 1}} \frac{w^{x' + n - 1}}{v^{x + n}} \frac{(2 w-1) (v-\alpha) (v + \alpha - 1) e^{t(v - w)}}{(w-\alpha) (w + \alpha - 1) (v-w) (w+v-1)}.
\end{multline}
Up to conjugation by $\theta^x$, and changing the names of the variables, this is exactly \eqref{eq:TwoSpeedKernel}.

\subsection{\for{toc}{Mixed sequential and parallel TASEP with general density and KPZ fixed point limit}\except{toc}{Mixed sequential and parallel TASEP with general particle density and KPZ fixed point limit}}\label{sec:caterpillars-mix}

Next we consider a situation where all speeds are equal but the $\kappa_i$'s are not homogeneous.
For concreteness, we focus on the case of discrete time TASEP with right Bernoulli jumps with parameter $p\in(0,1)$ and consider a situation where particles are arranged periodically in blocks of length $a\geq0$ and $b\geq0$, with particles in the first type of block updating sequentially and particles in the second type of block updating in parallel.
In other words, we are considering the system of caterpillars from Sec.~\ref{sec:caterpillars-rB} with lengths 
\[L_i=\begin{dcases*}1 & if $(i-1)\mod (a+b)\in\{0,\dotsc,a-1\}$\\2 & if $(i-1)\mod (a+b)\in\{a,\dotsc,b-1\}$\end{dcases*}\]
for each $i\geq1$.
In the notation of Sec.~\ref{sec:examples} (see in particular Sec.~\ref{sec:caterpillars-rB}) this means we need to take $\varphi(w)=1+w$, $v_i = p/q$ and $L_i$ as above, and then the multipoint distribution of this particle system is given by \eqref{eq:caterpillars} and \eqref{eq:caterpillars-Kt}.
Taking $a=1$ and $b=0$ we recover right Bernoulli TASEP with sequential update, while taking $a=0$ and $b=1$ we recover right Bernoulli TASEP with parallel update.

\subsubsection{KPZ universality}

Both sequential and parallel TASEP are known to converge, after proper rescaling, to the KPZ fixed point,  the conjectured universal scaling limit of all models in the KPZ universality class.
The KPZ fixed point was first constructed in \cite{fixedpt} as the scaling limit of continuous time TASEP, and is by now known to be the scaling limit of several TASEP-like particle systems (see \cite{arai,nqr-RBM,TASEPgeneral}, although in several cases the necessary asymptotics have not been performed in full detail), last passage percolation models (in this case using a different method, see \cite{DOV,DV}), as well as the semi-discrete polymer and the KPZ equation \cite{virag}, and ASEP and the six-vertex model \cite{aggarwalCorwinHegde2}.

To be more precise about the convergence of sequential and parallel TASEP to the KPZ fixed point, let $\UC$ be the space of upper semi-continuous functions $\fh\!:\rr\longrightarrow\rr\cup\{-\infty\}$ satisfying $\fh(x)\leq c_1|x|+c_2$ for some $c_1,c_2>0$ and $\fh\not\equiv-\infty$, endowed with the local Hausdorff topology (see \cite[Sec.~3]{fixedpt} for more details) and fix a parameter $\rho\in(0,1)$ which plays the role of the average particle density.
Then one expects that there are explicit constants $\alpha,\beta,\gamma,\sigma>0$ (which are not universal, in particular they differ between the sequential and parallel cases, see the end of this subsection for their explicit values) so that following holds: if the TASEP initial data $(X^\ep_0(i))_{i\geq1}$ is such that for some $\fh_0\in\UC$ satisfying $\fh_0(\fx)=-\infty$ for $\fx>0$, 
\begin{equation}\label{eq:TASEPini}
-\gamma^{-1}\ep^{1/2}\Big(X^\ep_0\big(\ep^{-1}\fx\big)+\rho^{-1}\ep^{-1}\fx\Big)\xrightarrow[\ep\to0]{}\fh_0(-\fx)
\end{equation}
in $\UC$, then one has
\begin{equation}
-\gamma^{-1}\sigma^{-1}\ep^{1/2}\Big(X_{\ep^{-3/2}\ft}\big(\alpha\ep^{-3/2}\ft-\sigma^2\ep^{-1}\fx\big)-\beta\ep^{-3/2}\ft-\rho^{-1}\sigma^2\ep^{-1}\fx\Big)\longrightarrow\fh(\ft,\fx;\fh_0)\label{eq:fxpt-conv}
\end{equation}
as a process in $\ft>0$ and $\fx\in\rr$, in distribution in $\UC$;
the limiting process $\fh(\ft,\fx)$ is the KPZ fixed point, which is a $\UC$-valued Markov process with initial data $\fh_0$ (which is indicated in the notation $\fh(\ft,\fx;\fh_0)$). 
The KPZ fixed point has explicit transition probabilities, which we will introduce shortly.
A full proof of \eqref{eq:fxpt-conv} involves some heavy asymptotic analysis, and has actually not appeared in detail in the literature for these models (restricting to convergence of finite dimensional distributions, in the sequential case, \cite{bfp} proved it for fixed time marginals and periodic initial data, and \cite{arai} gave a proof at the level of critical point computations for the general result; while in the parallel case a proof for fixed time marginals and periodic initial data appears in \cite{borodFerSas}), but the result can be established by starting from the Fredholm determinant formulas derived in \cite{TASEPgeneral} and suitably adapting the arguments given in \cite{fixedpt} for continuous time TASEP.

By KPZ universality, one expects that \eqref{eq:fxpt-conv} should also hold for the mixed sequential/parallel version of TASEP (for different choices of $\alpha,\beta,\gamma,\sigma$).
In this section we will work out the correct scaling and provide a sketch of the proof of the convergence for fixed $\ft>0$, although for brevity, we restrict ourselves to establishing pointwise convergence of the kernels appearing inside the Fredholm determinants.
A full proof of convergence to the KPZ fixed point would require upgrading this to convergence of the kernel in trace norm, which can be achieved by adapting the argument in \cite[Appdx.~B]{fixedpt} to the scaling analysis we will present here (although the extension in principle requires a considerable amount of work).





It is worth stressing that this analysis could also be performed, in an analogous way, for mixed sequential/parallel versions of other TASEP variants, such as those described in Sec.~2 of \cite{TASEPgeneral}, as well as for more general mixtures of caterpillars with different lengths.

Additionally, we note that in the setting of \eqref{eq:fxpt-conv} we are working with a general choice of particle density $\rho\in(0,1)$.
In most works (including \cite{fixedpt}) the density is chosen to be $1/2$ for simplicity, but the same limit is expected to hold for any density $\rho$ under the right scaling.
Since we will already be giving a detailed derivation of the scaling and a careful verification of the asymptotics at the level of pointwise convergence of the kernels, and since the extension to general $\rho$ for models such as sequential and parallel TASEP is important but mostly missing from the literature (a very recent exception is \cite{ferrariNejjar}, where the authors provide an argument for continuous time TASEP with periodic initial condition at a general density), we will provide here the arguments for the scaling and asymptotics at this level of generality.

\subsubsection{KPZ fixed point}\label{eq:kpzfixedpt}

Before getting started with the derivation, let us introduce the explicit formula for the KPZ fixed point transition probabilities.
We restrict the discussion to one-sided initial data $\fh_0$, which are such that $\fh_0(x)=-\infty$ for all $x>0$ (which is the class arising naturally from the class of TASEP initial conditions being considered in this paper where there is a rightmost particle); for the extension to two-sided data see \cite{fixedpt}.
Introduce the kernels
\begin{equation}\label{eq:fT}
\fT_{\ft,\fx}(u,v)=\ft^{-1/3}e^{\frac{2\fx^3}{3\ft^2}-\frac{(u-v)\fx}{\ft}}\Ai(\ft^{-1/3}(v-u)+\ft^{-4/3}\fx^2)
\end{equation}
for $\ft\neq0$, where $\Ai$ is the Airy function.
Note that
\begin{equation}
\fT_{-\ft,\fx}=\fT_{\ft,\fx}^*.\label{eq:fTadjoint}
\end{equation}
For $\ft>0$ one has the contour integral representation 
\begin{equation}
\fT_{\ft,\fx}(u,v)=\frac1{2\pi\I}\int_{\langle}\d z\,e^{\frac{\ft}3z^3+\fx z^2+(u-v)z}\label{eq:FT2}
\end{equation}
with $\langle$ the (positively oriented) contour made of two rays departing the origin at angles $\pm\pi/3$.
Then for any $\ft>0$ and any $\fx_1,\dotsc,\fx_m$ one has \cite{fixedpt}
\begin{equation}\label{eq:fixedptdet}
\pp\big(\fh(\ft,\fx_i)\leq\fa_i,\,i\in\set{m}\big)
=\det\!\left(\fI-\P_{\fa}\fK^{\hypo(\fh_0)}_{\ft,\uptext{ext}}\P_{\fa}\right)_{L^2(\{\fx_1,\dotsc,\fx_m\}\times\rr)}
\end{equation}
where (here $e^{\fy\p^2}(u,v)$, $\fy>0$, is the heat kernel, corresponding to the transition density of a Brownian motion with diffusivity $2$)
\begin{equation}
\fK^{\hypo(\fh_0)}_{\ft,\uptext{ext}}(\fx_i,\cdot;\fx_j,\cdot)=-e^{(\fx_j-\fx_i)\partial^2}\uno{\fx_i<\fx_j}+(\fT^{\hypo(\fh_0^-)}_{\ft,-\fx_i})^*\fT_{\ft,\fx_j}\label{eq:fixedptker}
\end{equation}
with $\fh_0^-(\fy)=\fh_0(-\fy)$ and
\begin{equation}\label{eq:defShypo}
\fT^{\hypo(\fh)}_{\ft,\fx}(v,u)=\ee_{\fB(0)=v}\big[\fT_{\ft,\fx-\ftau}(\fB(\ftau),u)\uno{\ftau<\infty}\big],
\end{equation}
and where $\fB$ is a Brownian motion with diffusivity $2$ and $\ftau$ is the hitting time of $\fB$ to the epigraph of $\fh$.

\subsubsection{Scaling}\label{sec:scaling-mixed}

In order to derive the desired convergence \eqref{eq:fxpt-conv} we need to fix $\ft>0$ and $\fx,\fa\in\rr^m$ and study the quantity 
\begin{multline}\label{eq:prob-det-ep}
\pp\Big(X_{\ep^{-3/2}\ft}(\alpha\tts\ep^{-3/2}\ft-\sigma^2\ep^{-1}\fx_i)>\tts\beta\tts\ep^{-3/2}\ft+\rho^{-1}\sigma^2\ep^{-1}\fx_i-\gamma\sigma\ep^{-1/2}\fa_i,\,i\in\set{m}\Big)\\
=\det(I-\bar\P_{r}K_t\bar\P_{r})_{\ell^2(\{n_1,\dotsc,n_m\}\times\zz)},
\end{multline}
where the kernel $\K_t$ on the right hand side is the one given in \eqref{eq:caterpillars-Kt} with the choices specified above and with
\begin{equation}\label{eq:tn-scaling}
t=\ep^{-3/2}\ft,\quad n_i=\alpha\tts\ep^{-3/2}\ft-\sigma^2\ep^{-1}\fx_i, \quad r_i= \beta\tts\ep^{-3/2}\ft+\rho^{-1}\sigma^2\ep^{-1}\fx_i-\gamma\sigma\ep^{-1/2}\fa_i
\end{equation}
for some parameters $\alpha$, $\beta$, $\gamma$, $\rho$ and $\sigma$ which we will specify shortly.

We introduce the following change of variables in the kernel $K_t(n_i,x_i;n_j,x_j)$:
\begin{equation}\label{eq:xi-scaling}
x_i=\beta\tts\ep^{-3/2}\ft+\rho^{-1}\sigma^2\ep^{-1}\fx_i+\gamma \sigma\ep^{-1/2}u_i.
\end{equation}
More properly, $x_i$, $n_i$, and $r_i$ should be taken to be the integer part of the right hand side, but we ignore this here and below.
After the change of variables in the Fredholm determinant we are on a lattice of size $\ep^{1/2}$ and the kernel $K_t$ gets multiplied by $\gamma\sigma\ep^{-1/2}$, while the limiting Fredholm determinant will be computed on $L^2(\{x_1,\dotsc,x_m\}\times\rr)$ and the projection $\bar\P_{r}$ will become $\bar\P_{-\fa}$ .
Our goal then is to compute, under this scaling
\begin{equation}\label{eq:barKlim}
\bar\fK(\fx_i,u_i;\fx_j,u_j)\coloneqq\lim_{\ep\to0}\gamma\sigma\ep^{-1/2}K_t(n_i,x_i;n_j,x_j),
\end{equation}
which will identify the limit of the scaled multipoint distributions of the left hand side of \eqref{eq:fxpt-conv} as $\det(\fI-\bar\P_{-\fa}\bar\fK\bar\P_{-\fa})_{L^2(\{\fx_1,\dotsc,\fx_m\}\times\rr)}$.
Note that we are interpreting the right hand side of \eqref{eq:barKlim} as a kernel acting on $L^2(\{\fx_1,\dotsc,\fx_m\}\times\rr)$

The parameters appearing in the scaling have to be chosen as follows (recall that $q=1-p$):
\begin{equation}\label{eq:paramset}
\begin{gathered}
\alpha = \frac{(p-q\theta)^2}{pq(1+\theta)^2 + \lambda(p-q\theta)^2}, 
\qquad\beta = \frac{p (q (1+\theta)^2 - 1)}{pq(1 + \theta)^2 + \lambda (p-q\theta)^2}\qquad\;\\
\gamma = \bigg(\frac{pq\theta}{2(p-q\theta)^2}+\frac{\theta\lambda}{2(1+\theta)^2}\bigg)^{1/2},\qquad
\sigma = \left(\frac{2 p q \theta (1+\theta) (p - q \theta)}{\gamma (pq(1 + \theta)^2 + \lambda (p-q\theta)^2)^2}\right)^{1/3},
\end{gathered}
\end{equation}
with
\begin{equation}\label{eq:thetaset}
\lambda=\frac{b}{a+b},\qquad
\theta = \frac{p-q-p(1+\lambda)\rho+\sqrt{(p^2 (1-\lambda)^2 + 4p\lambda)\rho^2-2p(1+\lambda)\rho+1}}{2q(1 - \lambda\rho)}.
\end{equation}
This last choice sets the value of the free parameter $\theta$ in the definition of $K_t$, and can be checked to be in $(0,p/q)$ as needed\footnote{Different choices of $\theta$ in that definition would require an additional conjugation on the right hand side of \eqref{eq:barKlim}, which anyway would not change the value of the associated Fredholm determinants; see also \cite[Rem.~5.16(b)]{TASEPgeneral}.} (in fact, one actually has
\begin{equation}\label{eq:theta-bd}
0<\theta<\tfrac{p}{q}(1-\rho),
\end{equation}
which will be useful later on), while $\lambda$ is simply the macroscopic proportion of parallel particles.
We will explain shortly where the key choice of $\theta$, as well as that of $\gamma$, come from; $\alpha$, $\beta$ and $\sigma$ could in principle be derived from KPZ scaling theory by studying the invariant measure of the process \cite{spohn-KPZscaling}, but we will not attempt that here and instead choose these parameters based directly on the asymptotics of our formulas.

\subsubsection{Asymptotics}

With the scaling \eqref{eq:tn-scaling}--\eqref{eq:thetaset} introduced in the previous section, the kernel 
\begin{equation}\label{eq:Ktep-def}
  \fK^\ep_{\ft}(\fx_i,u_i;\fx_j,u_j)=\gamma\sigma\ep^{-1/2}K_t(n_i,x_i;n_j,x_j)
\end{equation} 
appearing on the right hand side of \eqref{eq:barKlim} can be written as
\begin{equation}\label{eq:Ktep}
  \fK^\ep_{\ft}(\fx_i,\cdot;\fx_j,\cdot)=-\fQ^\eps_{\fx_i-\fx_j}\uno{\fx_i>\fx_j}+(\fT^\ep_{-\ft,\fx_i})^*\bar \fT^{\ep,\epi(X_0)}_{-\ft, -\fx_j}
\end{equation}
with, writing $y=\gamma\sigma\ep^{-1/2}\zeta$,
\begin{align}
\fQ^\eps_{\fx_i,\fx_j}(u_i,u_j) &= \gamma\sigma\ep^{-1/2}Q_{(n_i, n_j]}(x_i,x_j),\label{eq:Qep}\\
\fT^\ep_{-\ft,\fx_i}(\zeta,u_i) &= \gamma \sigma \ep^{-1/2} (1+\theta)^{-t}\SM_{-n_i}(y,x_i),\label{eq:S-rescaled}\\
\bar \fT^\ep_{-\ft, -\fx_j}(\zeta,u_j) &= \gamma \sigma \ep^{-1/2} (1+\theta)^{t} \SN_{(0, n_j]}(y,x_j),\label{eq:barS-rescaled}\\
\shortintertext{and}
\bar \fT^{\ep,\epi(X_0)}_{-\ft, -\fx_j}(\zeta,u_j) &= \gamma \sigma \ep^{-1/2} (1+\theta)^{-t}\SN^{\epi(X_0)}_{n_j}(y,x_j)\label{eq:barS-epi-rescaled}
\end{align}
(the factors in front of the kernels have been introduced because they are needed in the asymptotic analysis, while choice of subscript with $-\ft$ instead of $\ft$ is to coordinate notation with \cite{fixedpt}).\noeqref{eq:Ktep-def}

In order to compute the limit of $\fK^\ep_\ft$ we need to study the limits of each of the four kernels.
The third one does not appear explicitly in \eqref{eq:Ktep}, but it is convenient to study it separately, as the limit for the fourth kernel will then follow from that analysis and an analysis of the random walk hitting problem.

In the rest of this section we will provide a preliminary analysis which will allow us to identify the limits of the kernels.
The proof of the pointwise convergence of the kernels to their limit is deferred to the appendix.

For notational simplicity, in what follows we will omit the subscripts in $n_i$, $n_j$, $x_i$, $x_j$, $\fx_i$, $\fx_j$, $u_i$ and $u_j$ by $u'$, and replace them in the obvious way by $n$, $n'$ and so on.

\vskip6pt
\paragraph{\emph{Asymptotics for $\fQ^\ep_{\fx,\fx'}(u,u')$}}
\enspace
We begin by explaining how $Q_{(n,n']}(x,x')$, with $n<n'$ and scaled according to \eqref{eq:tn-scaling}, \eqref{eq:xi-scaling} and \eqref{eq:Qep}, leads to a heat kernel.
The kernel corresponds to the transition matrix of the random walk $B_m$ so, using the scaled variables, $Q_{(n,n']}(x,x')$ is the probability that $B_m-B_0=\rho^{-1}\sigma^2\ep^{-1}(\fx'-\fx)+\gamma\sigma\ep^{-1/2}(u'-u)$ with $m=\sigma^2\ep^{-1}(\fx-\fx')$ (note $n<n'$ implies $\fx>\fx'$).
We want to use the central limit theorem to argue that this converges to a Gaussian density.
For this we need the mean of the random walk to be $-\rho^{-1}$ (this choice in our scaling comes from the choice of average density $\rho$, as implied in \eqref{eq:TASEPini}).
Given $\theta\in(0,1)$, let 
\[\hat\theta=\tfrac{q}{p}\theta.\]
Then the mean of the jump distribution $Q_\ell$ corresponding to sequential particles is $-\frac{1}{1-\hat\theta}$ while the one corresponding to parallel particles is $-\frac{1-p(1-\hat\theta)^2}{(q+p\hat\theta))(1-\hat\theta)}$, as can be computed directly from their definition \eqref{eq:Q} (or \eqref{eq:Q-alt}) with the current choices.
Recalling that $\lambda$ denotes the density of parallel particles, the average mean of the jump distribution of the (inhomogeneous) random walk for the mixed case is
\[\textstyle-(1-\lambda)\frac{1}{1-\hat\theta}-\lambda\frac{1-p(1-\hat\theta)^2}{(q+p\hat\theta)(1-\hat\theta)}=-\frac{1}{1-\hat\theta}-\frac{p\tts\hat\theta\lambda}{q+p\hat\theta}.\]
Hence we need to choose $\hat\theta$ to be the solution of $\frac{1}{1-\hat\theta}+\frac{p\tts\hat\theta\lambda}{q+p\hat\theta}=\rho^{-1}$, which is explicitly given by $\hat\theta=\tfrac{q}{p}\theta$ with $\theta$ as chosen in \eqref{eq:thetaset}.
With this choice of the parameter $\theta$ and the above scaling, the central limit theorem implies that $\fQ^\eps_{\fx,\fx'}(u,u') \xrightarrow[\ep\to0]{}e^{\upsilon(\fx-\fx')/(2\gamma^2)\p^2}(u,u')$,
where $\upsilon$ is the average variance of the random walk jump distribution. 
This variance can be computed similarly, and equals $(1-\lambda)\frac{\hat\theta}{(1-\hat\theta)^2}+\lambda\frac{q\hat\theta+p\hat\theta^3}{(q+p\hat\theta)(1-\hat\theta)^2}=\frac{pq\theta}{(p-q\theta)^2}+\frac{\theta\lambda}{(1+\theta)^2}$.
The above choice of $\gamma$ implies that $\upsilon/\gamma^2=2$ and hence that, pointwise,
\begin{equation}
\label{eq:Qtep-conv}
\fQ^\eps_{\fx,\fx'}(u,u') \xrightarrow[\ep\to0]{}e^{(\fx-\fx')\p^2}\!(u,u').
\end{equation}

\vskip6pt
\paragraph{\emph{Formula and asymptotic expansion for $\fT^\ep_{-\ft,\fx}(\zeta,u)$}}
\enspace Using \eqref{eq:SM} (and recalling also that $v_\ell=p/q$ for all $\ell$ in this case while $\kappa_\ell$ equals $0$ for sequential particles and $1$ parallel ones) we get
\begin{align}
\SM_{-n}(y, x) &\textstyle= \frac{1}{2\pi\I}\oint_{\gamma_\rrin}\d w\,\frac{\theta^{x-y}}{w^{x-y + n + 1}} (1+w)^t \prod_{\ell = 1}^{n} \frac{p/q-w}{(1 + w)^{\kappa_\ell}} \frac{(1 + \theta)^{\kappa_\ell}\theta}{p/q-\theta}\\
&\textstyle=\left(\frac{\theta}{p/q-\theta}\right)^{n} \frac{p/q}{2\pi\I}\oint_{\gamma_{q \rrin / p}}\d w\,\frac{\theta^{x-y} (p/q)^{n} (1-w)^{n}}{(pw/q)^{x-y + n + 1}}(1+pw/q)^t \prod_{\ell = 1}^{n} \left( \frac{1 + \theta}{1 + pw/q}\right)^{\kappa_\ell}\\
&\textstyle=\left(\frac{\hat\theta}{1- \hat\theta}\right)^{n} \frac{1}{2\pi\I}\oint_{\gamma_{\hat r}}\d w\,\frac{\hat\theta^{x - y} (1-w)^{n}}{w^{x - y + n + 1}} (1+pw/q)^t \left( \frac{q + p \hat\theta}{q + pw}\right)^{\kappa(n_i)},\label{eq:Snihats}
\end{align}
with $\hat\theta=q\theta/p$ and $\hat r=qr/p$, and where we write, with $\lambda$ as defined in \eqref{eq:thetaset},
\begin{equation}
\textstyle\kappa(n)=\sum_{\ell=1}^{n}\kappa_\ell=\lambda n+\mathcal{O}(1)\label{eq:kappa_ni}
\end{equation}
(in fact, the $\mathcal{O}(1)$ term is at most $a+b$).
From our assumptions, $r$ here is bounded above by $\theta$, but at this stage we may clearly replace it by any radius in $(0,1)$, which means that we may replace $\hat r$ by any radius $r_0\in(0,q/p)$.
Moreover, note that the integrand in \eqref{eq:Snihats} only has a pole at $w = 0$; there is no singularity at $-q/p$ because $t > \kappa(n)\approx\lambda n$ for sufficiently small $\eps$, as follows from \eqref{eq:tn-scaling} (note that $\alpha\lambda<1$, we will use this often below).

With this we get, after a bit of algebra,
\begin{equation}
\textstyle\fT^\ep_{-\ft,\fx}(\zeta, u) = \gamma\sigma\tts\ep^{-1/2} \frac{1}{2\pi\I}\oint_{\gamma_{\hat r}}\d w\, \frac1{\hat\theta}\tts e^{F_\eps(w)} = \gamma\sigma\ep^{-1/2}\frac{1}{2\pi\I}\oint_{\Gamma}\d v\, e^{F_\eps(\hat\theta(1 - v))},\label{eq:123-a}
\end{equation}
where $\Gamma$ is a circle of radius $\hat r/\hat\theta$ centered at $1$, which we orient negatively (thus absorbing the minus sign from the change of variables) and where
$F_\eps(w) = n \log(\tfrac{\hat\theta}{1-\hat\theta}) + (x-y + 1) \log\hat\theta + n \log (1-w)
- (x-y + n + 1) \log w + (t - \kappa(n)) \log(\tfrac{q + pw}{q + p\hat\theta})$,
which gives
\begin{equation}
F_\eps(\hat\theta(1 - v)) = - (x-y + n + 1) \log(1 - v) + n \log(1 + A v) + (t - \kappa(n)) \log( 1 - B v), \label{eq:F_eps-a}
\end{equation}
with 
\begin{equation}
\textstyle A = \frac{\hat\theta}{1 - \hat\theta}=\frac{q\theta}{p-q\theta}\qqand B = \frac{p \hat\theta}{q+p\hat\theta}=\frac{\theta}{1+\theta}.\label{eq:AB-def-a}
\end{equation}
After the change of variables, the integrand in \eqref{eq:123-a} only has a singularity at $v=1$ inside the contour (there is no singularity at $(q+p\hat\theta)/(p\hat\theta)$ because $t>\lambda n_i$).

Using this and the choice of scaled variables \eqref{eq:tn-scaling} and \eqref{eq:xi-scaling} under the replacement $u_i=u$, and letting $y=\gamma\sigma\ep^{-1/2}\zeta$, we get
\begin{equation}\label{eq:F-eps-expansion-a}
F_\eps(\hat\theta(1 - v)) = \ep^{-3/2}\ft f_3(v) + \ep^{-1}\fx f_2(v) + \ep^{-1/2}(u- \zeta) f_1(v) + f_0(v),
\end{equation}
where
\begin{equation}\label{eq:functions-f-a}
\begin{aligned}
&f_3(v) = - (\alpha + \beta)\log(1 - v) + \alpha \log(1 + A v) + (1 - \alpha \lambda) \log( 1 - B v), \\
&f_2(v) = \sigma^2 \Bigl( - \tfrac{1-\rho}{\rho}\log(1 - v) - \log(1 + A v) + \lambda \log( 1 - B v) \Bigr), \\
&f_1(v) = \gamma \sigma\log(1 - v), \qquad\qquad f_0(v) = -(1+\CO(1))\log(1 - v),
\end{aligned}
\end{equation}
where the $\CO(1)$ term comes from \eqref{eq:kappa_ni} (and only depends on $a$ and $b$).

In order to analyze the exponent in \eqref{eq:123-a} we will use the expansion 
\begin{equation}
  \textstyle \log (1 + c\tts v) = - \sum_{k = 1}^n \frac{1}{k}(-c\tts v)^k + \CO(v^{n+1}),\label{eq:log-expansion}
\end{equation}
which is valid for fixed $c\in\rr$ and integer $n \geq 0$, with an error term which is uniform in $v$ in any compact set contained on the right half-plane.
Then we get
\begin{align}
&f_3(v) = \bigl( \alpha + \beta + \alpha A - (1 - \alpha \lambda) B \bigr) v + \bigl( \alpha + \beta - \alpha A^2 - (1 - \alpha \lambda) B^2 \bigr) v^2 \\
&\hspace{7cm} + \bigl(\alpha + \beta + \alpha A^3 - (1 - \alpha \lambda) B^3 \bigr) v^3 + E_3(v), \\
&f_2(v) = \sigma^2 \bigl( \tfrac{1-\rho}{\rho} - A - \lambda B\bigr) v + \tfrac{1}{2} \sigma^2 \bigl( \tfrac{1-\rho}{\rho} + A^2 - \lambda B^2\bigr) v^2 + E_2(v), \label{eq:f-asymptotics-0-a} \\ 
&f_1(v) = -\gamma\sigma v + E_1(v), \qquad\qquad f_0(v) = E_0(v),
\end{align}
where the $E_i$'s are error terms which satisfy
\begin{equation}
|E_i(v)|\leq c|v|^{i+1}\label{eq:Ei-errors-a}
\end{equation}
for some constant $c>0$ which can be chosen uniformly in $v$ in a fixed compact set.
Our choices of $\alpha$ and $\beta$ are exactly so that 
\begin{equation}\label{eq:ABeqns-a}
\alpha + \beta + \alpha A - (1 - \alpha \lambda) B = \alpha + \beta - \alpha A^2 - (1 - \alpha \lambda) B^2 = \tfrac{1-\rho}{\rho} - A - \lambda B = 0,
\end{equation}
where for the last one we also used that our choice of $\hat\theta$ satisfies $\frac{1}{1-\hat\theta}+\frac{p\hat\theta\lambda}{q+p\hat\theta}=\rho^{-1}$ (which is equivalent to it), as well as 
\begin{equation}\label{eq:moreABeqns}
\tfrac{1-\rho}{\rho} + A^2 - \lambda B^2=2\gamma^2,\qquad
\alpha + \beta + \alpha A^3 - (1 - \alpha \lambda) B^3=(\gamma\sigma)^3,
\end{equation}
which can be checked after some additional computations.
Then the formulas for $f_2$ and $f_3$ simplify to
\begin{equation}\label{eq:f-asymptotics-a}
f_3(v) = \tfrac{(\gamma\sigma)^3}{3}v^3+E_3(v),\qquad
f_2(v) = (\gamma\sigma)^2v^2+E_2(v).
\end{equation}
In particular, this means that $f_3$ has a critical point at the origin, and also (from \eqref{eq:F-eps-expansion-a} and together with \eqref{eq:f-asymptotics-0-a}) that
\begin{equation}
F_\ep(\hat\theta(1-v))\approx\ep^{-3/2}\ft\tfrac{(\gamma\sigma)^3}{3}v^3+\ep^{-1}\fx(\gamma\sigma)^2v^2+\ep^{-1/2}\gamma\sigma (\zeta-u)v.\label{eq:Feps-approx-a}
\end{equation}
After a further change of variables $v\longmapsto(\gamma\sigma)^{-1}\ep^{1/2}v$, this becomes exactly the exponent in the integral defining $\fT_{-\ft,\fx}(\zeta,u)$ in according to \eqref{eq:fTadjoint} and \eqref{eq:FT2}, and suggests that the following limit should hold:

\begin{prop}\label{prop:S-eps-converges}
For any $M>0$ one has
\begin{equation}\label{eq:S-eps-converges}
\fT^\ep_{-\ft,\fx}(\zeta,u) \xrightarrow[\ep\to0]{} \fT_{-\ft,\fx}(\zeta,u)
\end{equation}
uniformly over $|\zeta-u|\leq M$.
\end{prop}

We will prove this result in Appdx. \ref{sec:trcl-estim}.
Note that we have stated the limit in a stronger form than the pointwise convergence which we have been discussing; we do this because the proof is not any harder.

\vskip6pt
\paragraph{\emph{Formula and asymptotic expansion for $\bar\fT^\ep_{-\ft,-\fx}(\zeta,u)$}}
\enspace 
Proceeding similarly for the operator in \eqref{eq:barS-rescaled} we may write, based on \eqref{eq:SN},
\[\textstyle\SN_{(0, n]}(y,x) %
= - \left(\frac{1-\hat\theta}{\hat\theta}\right)^{n_j} \frac{1}{2\pi\I}\oint_{\gamma_{1}} \d w\,\frac{w^{x-y + n - 1}}{\hat\theta^{x-y} (1-w)^{n}}(1+p w /q)^{-t} \left( \frac{q + p w}{q + p \hat\theta}\right)^{\kappa(n)},\]
where $\gamma_1$ is any contour including only the pole at $1$ (and not the one at $-q/p$; this choice of contour coincides with the one in \eqref{eq:SN} after a having changed variables $w\longmapsto pw/q$), so that
\begin{equation}\label{eq:barSetp-exp-a}
\textstyle\bar \fT^\ep_{-\ft, -\fx}(u',u) = -\gamma\sigma\tts\ep^{-1/2} \frac{1}{2\pi\I}\oint_{\gamma_{1}}\d w\, \frac{1}{\hat\theta}\tts e^{\bar F_\eps(w)} = \gamma\sigma\ep^{-1/2}\frac{1}{2\pi\I}\oint_{\Gamma_{1 / \hat\theta - 1}}\d v\, e^{\bar F_\eps(\hat\theta(1 + v))},
\end{equation}
where $\Gamma_{1 / \hat\theta - 1}$ includes only the pole at $1 / \hat\theta - 1 > 0$ (and not the one at $-1-q/(p\theta)$) and is oriented negatively (thus absorbing the minus sign in the middle expression), and where
\begin{equation}
\bar F_\eps(\hat\theta(1 + v)) = (x-y + n + 1) \log(1 + v) - n \log(1 - A v) - (t - \kappa(n)) \log( 1 + B v), 
\end{equation}
Observe that here we changed variables $w\longmapsto\hat\theta(1+v)$ instead of $w\longmapsto\hat\theta(1-v)$ as for $\fT^\ep_{-\ft, \fx}$, which ensures that the contour can be taken to be contained in the right half plane.
But note $\bar F_\eps(\hat\theta(1 + v))$ equals $- F_\eps(\hat\theta(1 - v))$ from \eqref{eq:123-a} but defined via $x_j$ and $n_j$, so from \eqref{eq:F-eps-expansion-a} we get
\begin{equation}\label{eq:bar-Feps-expansion-a}
\bar F_\eps(\hat\theta(1 + v)) = \ep^{-3/2}\ft \bar f_3(v) + \ep^{-1}\fx_j \bar f_2(v) + \ep^{-1/2}(u'- u) \bar f_1(v) + \bar f_0(v)
\end{equation}
with $\bar f_i(v) = - f_i(-v)$.
In particular this means that, as in the previous case, $\bar f_3$ has a critical point at the origin, and also, as in \eqref{eq:Feps-approx-a}, that $\bar F_\ep(\hat\theta(1+v))\approx\ep^{-3/2}\ft\tfrac{(\gamma\sigma)^3}{3}v^3-\ep^{-1}\fx(\gamma\sigma)^2v^2+\ep^{-1/2}\gamma\sigma (\zeta-u)v$, which after changing variables $v\longmapsto(\gamma\sigma)^{-1}\ep^{1/2}v$ coincides with the exponent in the integral defining $\fT_{-\ft,-\fx}(\zeta,u)$.
This suggests that the following limit should hold:

\begin{prop}\label{prop:barS-eps-converges}
For any $M>0$ one has
\begin{equation}\label{eq:barS-eps-converges}
\bar\fT^\ep_{-\ft,\fx}(\zeta,u) \xrightarrow[\ep\to0]{} \fT_{-\ft,\fx}(\zeta,u).
\end{equation}
uniformly over $|\zeta-u|\leq M$.
\end{prop}

We will prove this result in Appdx. \ref{sec:trcl-estim}.

\vskip6pt
\paragraph{\emph{Limit of $\bar\fT^{\ep,\epi(X_0)}_{\ft,-\fx}(\zeta,u)$}}
\enspace 
We end this part by explaining how the limit of $\bar\fT^{\ep,\epi(X_0)}_{\ft,-\fx}(\zeta,u)$ arises.
We have
\begin{align}\label{eq:hitrescaled}
&\bar\fT^{\ep,\epi(X_0)}_{\ft,-\fx}(\zeta,u)=\gamma\sigma\ep^{-1/2}(1+\theta)^{t} \SN_{n}^{\epi(X_0)}(y, x)\\
&\qquad=\gamma\sigma\ep^{-1/2}(1+\theta)^{t}\ee_{\Bpl_0=y}\!\left[\SN_{(\taupl,n_j]}(\Bpl_\taupl, x_j)\uno{\tau<n_j}\right]\\
&\qquad=\ee_{\fB^\ep_0=\zeta}\!\left[\gamma\sigma\ep^{-1/2}(1+\theta)^{t}\SN_{(\taupl,n_j]}(\gamma\sigma\ep^{-1/2}\fB^\ep_{\sigma^{-2}\ep\taupl}, x_j+\rho^{-1}\taupl)\uno{\taupl<n_j}\right]
\end{align}
where the scaled random walk $(\fB^\ep_{\fy})_{\fy\geq0}$ is given by
\begin{equation}
\fB^\ep_{\fy}\coloneqq\gamma^{-1}\sigma^{-1}\ep^{1/2}(\Bpl_{\sigma^2\ep^{-1}\fy}+\rho^{-1}\sigma^2\ep^{-1}\fy).\label{eq:scaledBpl}
\end{equation}
This walk converges to a Brownian motion $\fB(\fx)$ with diffusivity $2$; this is obtained by studying the associated transition probabilities $\Qpl_{(n,n']}(x,x')$ using the same argument we used for the term $Q_{(n,n']}(x,x')$ (the only difference between $Q$ and $\Qpl$ in this context is that the arrangement of sequential and parallel particles is shifted by $1$, but this makes no difference in the argument).
On the other hand, the scaled hitting time $\sigma^2\ep^{-1}\taupl$ is nothing but the hitting time by $\fB^\ep_\fy$ of the scaled curve $\gamma^{-1}\sigma^{-1}\ep^{1/2}(X_0(\sigma^{2}\ep^{-1}\fy)+\rho^{-1}\sigma^{2}\ep^{-1}\fy)$ which, by \eqref{eq:TASEPini}, converges to $-\fh^-_0$.
This means that, as $\ep\to0$, this hitting time becomes the hitting time $\ftau^{\epi}$ by $\fB$ to the epigraph of $-\fh^-_0$.
All of this together with \eqref{eq:barS-eps-converges} leads to
\begin{equation}
\bar\fT^{\ep,\epi(X_0)}_{\ft,-\fx}(\zeta,u)
\xrightarrow[\ep\to0]{}\fT^{\epi(-\fh_0^{-})}_{-\ft,-\fx}(\zeta,u)
\coloneqq\ee_{\fB_0=\zeta}\!\left[\fT_{-\ft,-\fx-\ftau^{\epi}}(\fB_{\ftau^{\epi}},u)\right],\label{eq:346}
\end{equation}
where $\ftau^{\epi}$ is now the hitting time of the epigraph of $-\fh_0^-$.
Of course, a full proof of this limit, even at the pointwise limit level, necessitates some additional estimates on $\SN_{(0,n_j]}^{\epi(X_0^\ep)}$ in order to justify the joint convergence of this kernel and of the random walk and hitting time after rescaling.
This can be implemented as in \cite[Appdx. B.2]{fixedpt} once the necessary control on $\SN_{(0,n_j]}$ has been obtained, but is beyond the level of detail which we are providing here.

\subsubsection{Conclusion}

Putting the above computations together we deduce that the limiting kernel $\bar\fK$ defined in \eqref{eq:barKlim} is given by
\begin{equation}
\bar\fK(\fx_i,u_i;\fx_j,u_j)=-e^{(\fx_i-\fx_j)\partial^2}\uno{\fx_i>\fx_j}+(\fT_{-\ft,\fx_i})^*\fT^{\epi(-\fh_0^-)}_{-\ft,-\fx_j}.%
\end{equation}
The right hand side is an ``upside down'' version of $\fK^{\hypo(\fh_0)}_{\ft,\uptext{ext}}$: one has $\fK^{\hypo(\fh_0)}_{\ft,\uptext{ext}}(\fx_i,u_i;\fx_j,u_j)=\bar\fK^*(\fx_i,-u_i;\fx_j,-u_j)$, which also implies 
\begin{equation}\label{eq:reflect}
 \det(\fI-\bar\P_{-\fa}\bar\fK\bar\P_{-\fa})=\det(\fI-\P_{\fa}\fK^{\hypo(\fh_0)}_{\ft,\uptext{ext}}\P_{\fa}),
\end{equation}
see \cite[Sec. 3.3]{fixedpt} for the details.
In view of \eqref{eq:fixedptdet}, this leads to the following statement:

\vs
\begin{quote}\it 
  For $X_t$ the mixed sequential/parallel version of right Bernoulli TASEP, with blocks of $a$ sequential particles followed by blocks of $b$ parallel particles, one has
  \[-\gamma^{-1}\sigma^{-1}\ep^{1/2}\Big(X_{\ep^{-3/2}\ft}\big(\alpha\ep^{-3/2}\ft-\sigma^2\ep^{-1}\fx\big)-\beta\ep^{-3/2}\ft-\rho^{-1}\sigma^2\ep^{-1}\fx\Big)\longrightarrow\fh(\ft,\fx;\fh_0)\]
  (with $\fh(\ft,\fx;\fh_0)$ the KPZ fixed point started at $\fh_0$) for fixed $\ft>0$ and in the sense of finite dimensional distributions, with the parameter choices specified in \eqref{eq:paramset} and \eqref{eq:thetaset} and initial data with particle density $\rho$ and satisfying \eqref{eq:TASEPini}.
\end{quote}
\vs

The arguments presented here provide strong support for the validity of this statement.
As explained above, a full proof of the convergence would entail a non-trivial upgrade of Propositions \ref{prop:S-eps-converges} and \ref{prop:barS-eps-converges} to establish the limit \eqref{eq:barKlim} in trace norm.

Setting $\lambda=0$ and $\lambda=1$, the above setting recovers the sequential and parallel cases.
In the case of sequential TASEP the scaling parameters simplify somewhat: one has 
\begin{equation}\label{eq:seq-params}
  \textstyle\theta=\frac{p(1-\rho)}{q},\quad\alpha=\frac{pq\rho^2}{(1-p\rho)^2},\quad\beta=\frac{p(1-2\rho+p\rho^2)}{(1-p\rho)^2},\quad\gamma=\sqrt{\frac{1-\rho}{2\rho^2}}\qand \sigma=\frac{\sqrt{2}(pq)^{1/3}\rho^{2/3}(1-\rho)^{1/6}}{1-p\rho}.
\end{equation}
For parallel TASEP the parameters still look rather complicated, but do simplify in the case of $\rho=1/2$: 
\[\textstyle\theta=\frac{1-\sqrt{q}}{\sqrt{q}},\quad\alpha=\frac{1-\sqrt{q}}{2},\quad\beta=0,\quad\gamma=q^{1/4}\qand\sigma=2^{-1/3}p^{1/3}q^{-1/12}.\]
Both cases match the scaling that has previously appeared in the literature for the standard case $\rho=1/2$.

\subsection{TASEP with a block of caterpillars at the bulk of the system}\label{sec:caterpillar-block}

As a final example, we consider a situation where we modify the memory length of a macroscopic block of particles.
Here we will restrict to the simplest possible version of this setting, where the block of caterpillars is placed at the bulk of the system, that is, far from both the first particles and those lying near the characteristic line for the system.
We describe the situation next; for some comments on other possibilities see the discussion at the end of this section.

The system will be based on discrete time TASEP with sequential update and a fixed jump parameter $p$, at a fixed choice of density $\rho\in(0,1)$.
Given these two values, we fix the choices of parameters $\alpha$, $\beta$, $\gamma$, $\sigma$ and $\theta$ as in the previous section with $\lambda=0$, i.e., through \eqref{eq:seq-params}.
We also assume that the initial data $X_0$ satisfies \eqref{eq:TASEPini}.
The analysis from the previous section then tells us that we should look at the system at time $t=\ep^{-3/2}\ft$ and focus on particles with labels near $\alpha\ep^{-3/2}\ft$.

Now we fix an $\alpha'\in(0,\alpha)$ and a $b>0$, and we place a block of 
\[M=b\ep^{-1}\]
caterpillars with length $L=\kappa+1$ starting at label $\alpha'\ep^{-3/2}\ft$ (with the same jump parameter $p$).
In other words, we choose the parameters of the system to be $v_i=p/q$ for all $i$ and
\[\kappa_i=1+(\kappa-1)\uno{\alpha'\ep^{-3/2}\leq i<\alpha'\ep^{-3/2}+M}.\]
Recall that in order to apply our results, $X_0$ has to satisfy $X_0(i)-X_0(i+1)\geq\kappa_i$, which in our case imposes the condition 
\begin{equation}\label{eq:X0blockcondition}
  X_0(i)-X_0(i+1)\geq\kappa\qquad\uptext{for }\alpha'\ep^{-3/2}\leq i<\alpha'\ep^{-3/2}+b\ep^{-1}.
\end{equation}
We will simply assume that our initial data satisfies this assumption.
As a simple example, one could consider $d$-periodic initial data $X_0(i)=-di$, $i\geq1$, (corresponding to density $\rho=1/d$), which satisfies this condition for any $\kappa\leq d$.

The introduction of particles with longer memory introduces a delay in the system (since the tail of a caterpillar is located at the position where the head was $L-1$ units of time in the past), which by the TASEP dynamics should induce a similar delay on the particles to the right of the perturbation.
Hence if we observe the system at time $t=\ep^{-3/2}\ft$ and focus on particles with indices $n_i=\alpha\ep^{-3/2}\ft-\sigma^2\ep^{-1}\fx_i$ as in \eqref{eq:tn-scaling}, we should expect that the effective time during which these particles evolve is reduced by an amount of order $\ep^{-1}$, in view of our choice of $M$.
If this is the case, then we should expect particle $n_i$ to lie slightly to the left of the position $\beta\ep^{-3/2}\ft+\rho^{-1}\sigma^2\ep^{-1}$ predicted (up to order $\ep^{-1/2}$) by the scaling \eqref{eq:tn-scaling}, and the displacement should be of order $\ep^{-1}$.
But by KPZ universality it is natural to expect that the system will still converge to the KPZ fixed point, if we modify the scaling slightly in order to take the perturbation into account.
Our goal is to work this out.
To this end, we modify the scaling from \eqref{eq:tn-scaling} as follows: for fixed $\ft>0$ and $\fx_i,\fa_i\in\rr$, $i\in\set{m}$, we use
\begin{equation}\label{eq:sc-long1}
t=\ep^{-3/2}\ft,\qquad n_i=\alpha\ep^{-3/2}\ft-\sigma^2\ep^{-1}\fx_i, \qquad r_i= \beta\ep^{-3/2}\ft+\ep^{-1}(\rho^{-1}\sigma^2\fx_i-\mu)-\gamma\sigma\ep^{-1/2}\fa_i
\end{equation}
for the parameters and
\begin{equation}
x_i=\beta \ep^{-3/2}\ft+\ep^{-1}(\rho^{-1}\sigma^2\fx_i-\mu)+\gamma\sigma\ep^{-1/2}u_i\label{eq:xi-scaling-long1}
\end{equation}
for the kernel variables, with $\mu$ a parameter to be determined.

Consider first the term $Q_{(n_i, n_j]}(x_i,x_j)$ appearing in \eqref{eq:kernel-explicit} for $n_i<n_j$.
In view of our choices, $Q_{(n_i, n_j]}(x_i,x_j)$ only uses the parameters $v_i$ and $\kappa_i$ associated to labels $i$ at distance of order $\ep^{-1}$ from $\alpha\ep^{-3/2}\ft$, so since $\alpha'<\alpha$, for small $\ep$ it simply corresponds to the same kernel as for the sequential case.
Therefore, as in the previous example, the conclusion is that $\gamma\sigma\ep^{-1/2}Q_{(n_i,n_j]}(x_i,x_j)$ converges as $\ep\to0$ to $e^{(\fx_i-\fx_j)\p^2}\tsm(u_i,u_j)$ by the central limit theorem.

Next we need to study the scaled kernel $\gamma\sigma\ep^{-1/2}(\SM_{-n_i})^*\SN^{\epi(X_0^\ep)}_{(0,n_j]}(x_1,x_2)$.
Proceeding as before, we focus first on the limit of $\gamma\sigma\ep^{-1/2}(\SM_{-n_i})^*(x_i,y)$ with $y=\gamma\sigma\ep^{-1/2}u$, for which we get as in \eqref{eq:Snihats}
\[\textstyle\SM_{-n_i}(y, x_i) = \left(\frac{\hat\theta}{1-\hat\theta}\right)^{n_i}\frac{1}{2\pi\I}\oint_{\gamma_{\hat r}}\d w\,\frac{\hat\theta^{x_i - y}(1-w)^{n_i}}{w^{x_i - y + n_i + 1}} (1+pw/q)^t \left( \frac{q + p\hat\theta}{q + pw}\right)^{M\kappa}\]
(where we used the fact that $n_i>\alpha'\ep^{-3/2}\ft+M$ for small $\ep$).
Continuing the argument in the same way leads to
\[\textstyle\gamma\sigma \ep^{-1/2} (1+\theta)^{-t} \SM_{-n_i}(y, x_i) 
= \gamma\sigma\ep^{-1/2} \frac{1}{2\pi\I}\oint_{\Gamma}\d v\, e^{F_\eps(\hat\theta(1 - v))},\]
where $\Gamma$ is a negatively oriented circle of radius $\hat\rrin/\hat\theta$ centered at $1$ and
\begin{equation}%
F_\eps(\hat\theta(1 - v)) = \ep^{-3/2}\ft f_3(v) + \ep^{-1}\fx_i f_2(v) +\ep^{-1}\tilde f_2(v)+ \ep^{-1/2}(u_i-u) f_1(v) + f_0(v),
\end{equation}
with
\begin{equation}%
\begin{aligned}
&f_3(v) = - (\alpha + \beta)\log(1 - v) + \alpha \log(1 + A v) + \log( 1 - B v), \\
&f_2(v) = - \tfrac{1-\rho}{\rho}\sigma^2\log(1 - v) - \sigma^2\log(1 + A v),\quad\tilde f_2(v)=-\mu\log(1-v)-b\kappa\log(1-Bv),\\
&f_1(v) = \gamma \sigma\log(1 - v), \qquad\qquad f_0(v) = -1,
\end{aligned}
\end{equation}
where $A$ and $B$ are defined as above (with $\lambda=0$).
Except for $\tilde f_2$, which is new, these are the same functions as in the sequential (i.e., $\lambda=0$) case of the previous example, and so they satisfy the same expansions as those in \eqref{eq:f-asymptotics-0-a}/\eqref{eq:f-asymptotics-a}.
Proceeding in a similar way for $\tilde f_2$ (using \eqref{eq:log-expansion}) gives
\[\tilde f_2(v)=-(\mu-b\kappa B)v-\tfrac12(\mu-b\kappa B^2)v^2+E_2(v).\]
Hence we need to choose
\begin{equation}
  \label{eq:mu-choice}
  \mu=b\kappa B=\tfrac{p(1-\rho)}{1-p\rho}b\kappa
\end{equation}
Under this choice we get $\tilde f_2(v)=-\fb(\gamma\sigma)^2v^2+E_2(v)$ with
\begin{equation}
  \fb=\tfrac1{2(\gamma\sigma)^2}(\mu-b\kappa B^2)=\tfrac{pq(1-\rho)}{2(\gamma\sigma)^2(1-p\rho)^2}b\kappa=\tfrac{(pq)^{1/3}}{2\rho^{1/3}(1-\rho)^{1/3}}b\kappa.\label{eq:fb-choice}
\end{equation}

Continuing as in the previous example (see \eqref{eq:Feps-approx-a}), the (formal) conclusion of all this is that
\[\gamma\sigma \ep^{-1/2} (1+\theta)^{-t} \SM_{-n_i}(y, x_i)
\xrightarrow[\ep\to0]{}\fT_{-\ft,\fx_i-\fb}(u_i,u).\]
The same argument gives
\[\gamma\sigma \ep^{-1/2} (1+\theta)^{t} \SN_{n_j}(y, x_j)
\xrightarrow[\ep\to0]{}\fT_{-\ft,-\fx_j+\fb}(u,u_i).\]
Now consider $\SN_{n_j}^{\epi(X_0^\ep)}(y, x_j)=\ee_{\Bpl_0=y}[\SN_{(\taupl,n_j]}(\Bpl_{\taupl},x_j)\uno{\taupl<n_j}]$.
Note that in the hitting problem defining this kernel in \eqref{eq:SN-epi}, the walk $B^+$ starts at the origin.
Moreover, starting at the origin, and for the first $\alpha'\ep^{-3/2}\ft$ number of steps, the walk has homogeneous jumps coinciding with those of the sequential case.
But the walk $B^+$ and the curve $X_0$ it has to hit are being rescaled diffusively at an $\ep^{-1}$ timescale, so the whole hitting problem is effectively a hitting problem for the homogenous walk\footnote{Here the condition $\fh_0(x)\leq c_1|x|+c_2$ for the $\UC$ initial condition $\fh_0$ and the fact that $X_0$ converges in $\UC$ to $\fh_0$, together with the definition of convergence appearing in \cite[Sec. 3]{fixedpt}, ensure that the walk actually hits the epigraph of $X_0$ at a time of order $\ep^{-1}$ with high probability so that the the remaining part of the hitting problem does not contribute in the limit. This is actually contained in the arguments appearing in \cite{fixedpt}, and we omit the details.}.
In any case, the conclusion is that we are then in essentially the same setting as in the sequential TASEP case corresponding to Sec.~\ref{sec:caterpillars-mix} with $\lambda=0$ and then, proceeding exactly as in \eqref{eq:hitrescaled}, we get
\[\gamma\sigma_\ep \ep^{-1/2}(1+\theta)^{t} \SN_{(0,n_j]}^{\epi(X_0^\ep)}(y, x_j)
\xrightarrow[\ep\to0]{}\fT^{\epi(-\fh_0^{-})}_{-\ft,-\fx_j+\fb}(u,u_j)
.\]

Putting everything together we see that the rescaled kernel $K_t(n_i,x_i;n_j,x_j)$ converges to
\begin{equation}
\bar\fK(\fx_i,\cdot;\fx_j,\cdot)\coloneqq-e^{(\fx_i-\fx_j)\p^2}\uno{\fx_i>\fx_j}+\fT_{-\ft,\fx_i-\fb}^*\fT^{\epi(-\fh_0^-)}_{-\ft,-\fx_j+\fb}.\label{eq:barKlong1-a}
\end{equation}
Thus, as in the previous example, this provides strong support for the following statement (a full proof of which would require again a non-trivial upgrade to trace class convergence of the kernels):

\vs
\begin{quote}\it
For $X_t$ the version of right Bernoulli TASEP with sequential update where particles with labels between $\alpha'\ep^{-3/2}\ft$ and $\alpha'\ep^{-3/2}\ft+\kappa b\ep^{-1}-1$ are have memory $\kappa$, and assuming that the initial condition $X_0$ satisfies \eqref{eq:TASEPini} and \eqref{eq:X0blockcondition}, we have
\begin{equation}\label{eq:longscalinglimit}
-\sigma^{-1}\ep^{1/2}\Big(X_{\ep^{-3/2}\ft}(\alpha\ep^{-3/2}\ft-\sigma^2\ep^{-1}\fx)-\beta\ep^{-3/2}\ft-(\rho^{-1}\sigma^2\fx-\mu)\ep^{-1}\Big)\longrightarrow\fh(\ft,\fx-\fb;\fh_0)
\end{equation}
for fixed $\ft>0$, in the sense of finite dimensional distributions, with the parameter choices specified in \eqref{eq:seq-params}, \eqref{eq:mu-choice} and \eqref{eq:fb-choice}, and where the limit on the right hand side is the KPZ fixed point with initial condition $\fh_0$.
\end{quote}
\vs

In words, a block of $b\ep^{-1}$ particles with memory $\kappa$ placed in the bulk has the effect of shifting the expected position of particles along the characteristic direction by $-\frac{p(1-\rho)}{1-p\rho}b\kappa\ep^{-1}$, and correspondingly, the process describing the limiting fluctuations experiences a spatial shift of $-\frac{(pq)^{1/3}}{2\rho^{1/3}(1-\rho)^{1/3}}b\kappa$.

If instead the block of caterpillars has size of order $\ep^{-1/2}$, a similar analysis shows that the particle positions are now shifted by an amount of that order, while the limiting process has no spatial shift.
If the block of caterpillars is placed along the characteristic direction defined by $n_i=\alpha\ep^{-3/2}\ft$ (i.e., $\alpha'=\alpha$ in the notation above) the situation is similar: the same parts of the kernel are affected in the limit, but the overall effect depends now on the values of $\fx_i$, which determine how much of the perturbed block particle $n_i$ sees.
This is why we chose $\alpha'<\alpha$, since $\alpha'=\alpha$ leads to the same kind of shift, but with additional, unnecessary complications.
Finally, if the block is placed at the beginning of the system, the situation is different, as the perturbation now affects the distribution of the random walk $\Bpl$ involved in the hitting problem.
We leave the analysis of this last case for future work.

\appendix

\section{Convergence of the mixed TASEP kernels}\label{sec:trcl-estim}

The goal of this section is to prove Props. \ref{prop:S-eps-converges} and \ref{prop:barS-eps-converges}.
The basic method which we are going to use is a steep descent analysis.
The idea is that, given a complex integral $\int_\gamma \d v\, e^{f(v)}$ with $f$ having a critical point at $v_0$, one deforms the original contour $\gamma$ to \emph{steep descent contour}, which we will still call $\gamma$, which is so that $\gamma$ passes through $v_0$ and $\Re f(v)$ decreases along the contour as $v$ goes away from $v_0$ (up to some other point $v_1\neq v_0$ where the contour is closed), so that the integral can be estimated based on the size of $|e^{f(v_0)}|$.
Throughout this part we will use the notation and parameters choices specified in Sec. \ref{sec:caterpillars-mix}.

We begin by specifying the basic contour which we will employ in our steep descent analysis of the kernels.
The contour will depend on some parameters,
\[\phi\in(\tfrac{\pi}6,\tfrac{\pi}2),\qquad\ell>0,\qqand\psi\in(0,\pi-\phi),\]
which will be chosen differently depending on which kernel we are studying and depending on the parameters of the model.
The basic contour which we will use is given as 
\[\mathcal{C}=\mathcal{C}_1\cup \mathcal{C}_2\]
with $\mathcal{C}_1$ and $\mathcal{C}_2$ defined as follows (we omit their dependence on $v_0$, $\ell$, and $\phi$).
First, we set
\begin{equation}\label{eq:C-circ-contours}
\mathcal{C}_1 = \textstyle \{r e^{-\I \phi}\!: r \in [0,\ell]\} \cup \{r e^{\I \phi}\!: r \in [0,\ell]\},
\end{equation}
oriented with increasing imnaginary part.
Next we let $\tilde\Gamma$ be a circle centered at some $R_0>0$ and going through the endpoints $\ell\tts e^{\pm\I \phi}$ of $\mathcal{C}_1$. 
We denote the radius of this circle by $R$.
Then we let $\mathcal{C}_2$ be the part of $\tilde\Gamma$ lying to the right of $\mathcal{C}_1$, oriented clockwise (see Fig.~\ref{fig:contours}).
The choice of $R_0$ is not free, instead we want it to be so that the angle between the rays connecting the center of $\tilde\Gamma$ with the origin and with $\ell\tts e^{\I \phi}$ is precisely the parameter $\psi$;
the law of sines, applied to the triangle built out of those three points, yields 
\[R = \tfrac{\sin\phi}{\sin \psi}\ell\qqand R_0 = \tfrac{\sin (\pi-\phi-\psi)}{\sin\psi}\ell,\]
while $\mathcal{C}_2$ can be described parametrically as
\begin{equation}\label{eq:C2}
\mathcal{C}_2 = \{R_0 + R e^{-\I \vartheta}\!: \vartheta \in [\psi-\pi, \pi-\psi]\}.
\end{equation}

In order to deform the contours in \eqref{eq:123-a} and \eqref{eq:barSetp-exp-a} we need to make sure that no singularities of the integrand are crossed:
\begin{itemize}
\item For \eqref{eq:123-a} there is only a pole at $v=1$, which will remain inside the contour if 
\begin{equation}\label{eq:contour-condition}
R_0+R>1.
\end{equation}
\item The case of \eqref{eq:barSetp-exp-a} needs a bit more care, because the contour now needs to enclose the singularity at $1/\hat\theta-1=1/A$, which can be arbitrarily large depending on the choice of parameters.
In any case, what we need is to ensure that
\begin{equation}\label{eq:contour-bar-condition}
\textstyle R_0+R=(\frac{\sin\phi}{\sin{\psi}}+\frac{\sin(\pi-\phi-\psi)}{\sin\psi})\ell>A^{-1}
\end{equation}
which, given $\phi$ and $\psi$, will always hold if $\ell$ is large enough.
\end{itemize}

\begin{figure}[t]
\centering
\includegraphics[width=0.4\textwidth]{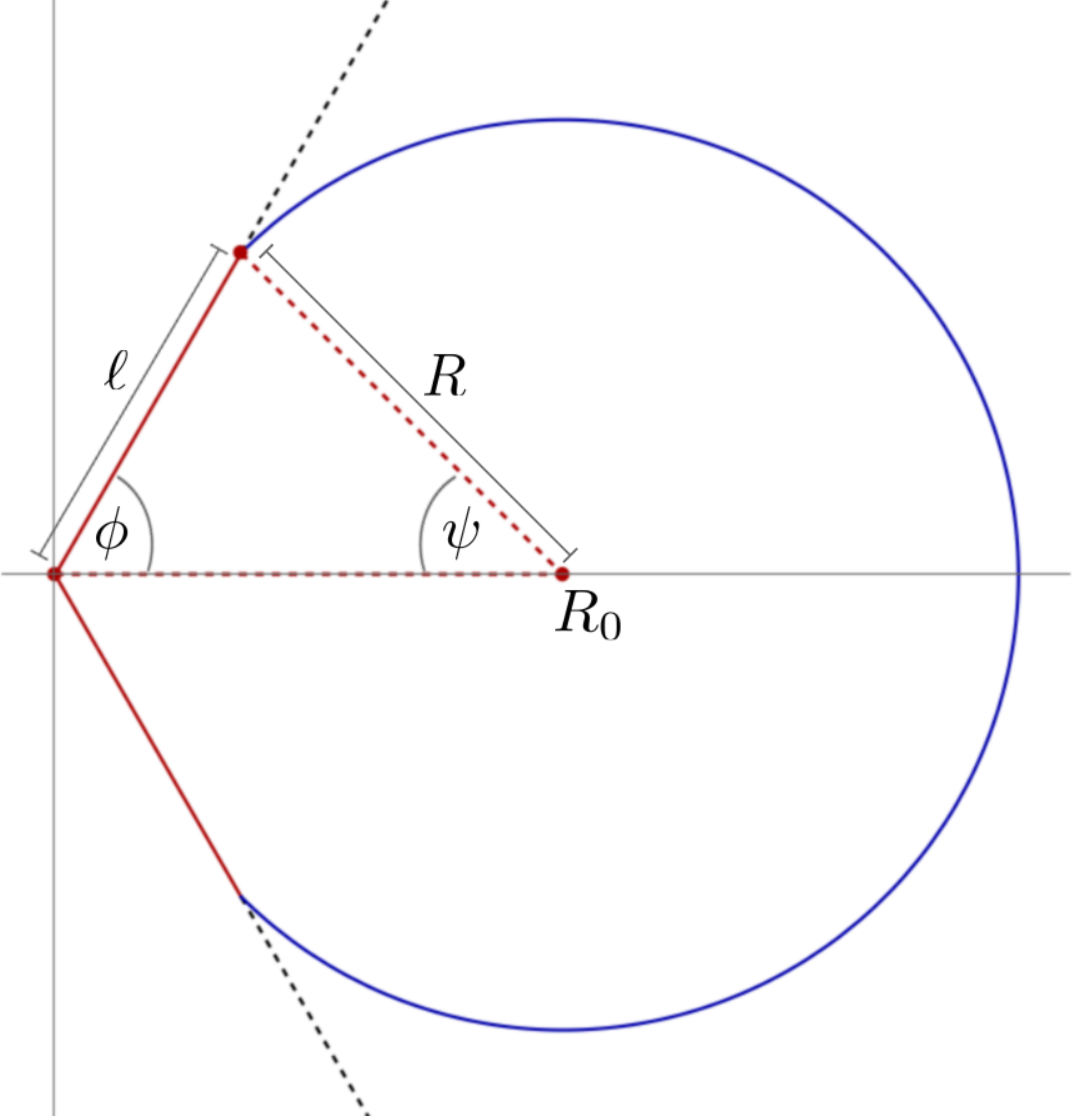}
\caption{The contour $\mathcal{C}=\mathcal{C}_1 \cup \mathcal{C}_2$ chosen for the steep descent analysis. $\mathcal{C}_1$ is the red part of the two rays departing off the origin, while $\mathcal{C}_2$ is the blue arc of the circle $\tilde\Gamma$.
The radius $R$ of the circle and the coordinate $R_0$ of its center can be computed based on the length $\ell$ of the left side of the triangle and the values of the two bottom angles $\phi$ and $\psi$.}
\label{fig:contours}
\end{figure}

We turn now to the proof of the two results, which we will handle together.

\begin{proof}[Proof of Props.  \ref{prop:S-eps-converges} and \ref{prop:barS-eps-converges}.]
Recall that the functions $F_\ep(\hat\theta(1-v))$ and $\bar F_\ep(\hat\theta(1+v))$ appearing in \eqref{eq:123-a} and \eqref{eq:barSetp-exp-a} are given in \eqref{eq:F-eps-expansion-a} and \eqref{eq:bar-Feps-expansion-a}, and that their leading order terms $\ep^{-3/2}\ft f_3(v)$ and $\ep^{-3/2}\ft\bar f_3(v)$ have a critical point at the origin.
Throughout the proof we assume that $|\zeta-u|\leq M$ for a fixed choice of $M>0$.

In order to use the steep descent method we need to show that $\Re f_3(v)$ and $\Re\bar f_3(v)$ have a maximum at $v = 0$ and decay strictly as $v$ departs from the origin along the contour $\mathcal{C}_1 \cup \mathcal{C}_2$ in both the upper and lower half-planes.
We will do this separately for the two cases, and after that conclude the proof.
For clarity, we organize the proof in four steps.

\vskip2pt

\noindent\underline{Step 1: analysis of $\Re f_3(v)$.}
\enspace
In this part we choose $\phi=\pi/4$ and $\ell=\sqrt{2}/(1+\sqrt{3})$.
This and the relation between $R_0$ and $\psi$ yield $R_0+R=\frac{\sqrt2}{1+\sqrt{3}}(\frac{\sqrt{2}}{2\sin{\psi}}+\frac{\sin(3\pi/4-\psi)}{\sin\psi})$, which is easily checked to be larger than $1$ (so that \eqref{eq:contour-condition} holds) for $\psi<\pi/3$.

Let $\mathcal{C}_i^+$ be the part of $\mathcal{C}_i$ lying on the upper half-plane. 
Then by symmetry it is enough to check the decay only on the contour $\mathcal{C}^+_1 \cup \mathcal{C}^+_2$.
On $\mathcal{C}^+_1$ we use the identity $\Re \log (1 + c e^{\I \phi}) = \log |1 + c e^{\I \phi}| = \frac{1}{2} \log (1 + 2 c \cos \phi + c^2)$ to write  
\begin{align}\label{eq:Re-analysis}
\tfrac{\d}{\d r} \Re f_3(r e^{\I \pi / 4}) = (\alpha + \beta) g(- r) + \alpha A\ts g(A r) - (1 - \alpha \lambda) B\tts g(-B r)
\end{align}
with $g(x) = \tfrac{\sqrt{2} + 2 x}{2(1 + \sqrt{2}x + x^2)}$, which yields, after some simplification,
\begin{multline}\label{eq:Re-analysis2}
\tfrac{\d}{\d r} \Re f_3(r e^{\I \pi / 4}) = \tfrac{p q \theta^2 r^2}{\sqrt{2}(1+\theta)^2 (p - q \theta)^2 (q(p + q \lambda) \theta^2 + 2 p q (1 - \lambda) \theta + p(q + p \lambda))} \\
\times \tfrac{(1+\theta) (q \theta - p) - \theta (q - p + 3q \theta) r^2 + \sqrt2 q\theta^2 r^3}{(1-\sqrt{2}r+r^2) (1+\sqrt{2}Ar+A^2r^2) (1- \sqrt{2}B r+ B^2 r^2)}.
\end{multline}
All the factors, except the numerator of the second fraction, are clearly strictly positive.
The latter may be written as $(\sqrt2 r^3 - 3 r^2 + 1) q \theta^2 + (r^2 - 1) (p-q) \theta - p$. 
For any $r\in[0,\ell]$ this parabola with respect to $\theta$ is convex, and thus it is strictly negative for $\theta \in [0, \frac{p}{q}(1-\rho)]$ (which is enough for us thanks to \eqref{eq:theta-bd}), as it is so at $\theta = 0$ and $\theta = \frac{p}{q}(1-\rho)$.
This shows that $\tfrac{\d}{\d r} \Re f_3(r e^{\I \pi / 4}) < 0$ for $r \in (0, \ell\tts]$ and it vanishes at $r = 0$.

Similarly, on $\mathcal{C}^+_2$ we have 
\begin{align}
\tfrac{\d}{\d \vartheta} \Re f_3(R_0 + R e^{-\I \vartheta}) &= - \Bigl( \tfrac{(\alpha + \beta) a_1}{1 - 2 a_1 \cos \vartheta + a_1^2} + \tfrac{\alpha a_2}{1 + 2 a_2 \cos \vartheta + a_2^2} - \tfrac{(1 - \alpha \lambda) a_3}{1 - 2 a_3 \cos \vartheta + a_3^2} \Bigr)\tsm\sin\vartheta\\
&= - \tfrac{\eta(\cos \vartheta)}{(1 - 2 a_1 \cos \vartheta + a_1^2) (1 + 2 a_2 \cos \vartheta + a_2^2) (1 - 2 a_3 \cos \vartheta + a_3^2)} \sin\vartheta,\label{eq:steep}
\end{align}
with $a_1 = \frac{R}{1-R_0}$, $a_2 = \frac{A R}{1+A R_0}$, $a_3 = \frac{B R}{1-B R_0}$ and 
\begin{multline}
\eta(x) = (\alpha + \beta) a_1 (1 + 2 a_2 x + a_2^2) (1 - 2 a_3 x + a_3^2) + \alpha (1 - 2 a_1 x + a_1^2) a_2 (1 - 2 a_3 x + a_3^2) \\
- (1 - \alpha \lambda) (1 - 2 a_1 x + a_1^2) (1 + 2 a_2 x + a_2^2) a_3. \label{eq:function-eta}
\end{multline}
The denominator in \eqref{eq:steep} is strictly positive, so since we need the expression to be strictly negative for $\vartheta \in [\psi - \pi, 0)$ and since $\sin(\vartheta)<0$ there, it is enough to show that the parabola $\eta(x)$ satisfies $\eta(x) < 0$ for $x\in[-1,1  ]$.
From the definition of the $a_i$'s and the scaling parameters, the coefficient of $x^2$ in $\eta(x)$ is $4 a_1 a_2 a_3(1-\alpha\lambda - \beta)$, which is strictly positive for $R_0\in(0,1)$ (because so are the $a_i$'s and since clearly $\alpha\lambda+\beta<1$), 
so the parabola is convex in this region.
Recall that, with our choice of contour, $R_0 = \frac{\sin(\frac{3\pi}{4}-\psi)}{\sin\psi}\frac{\sqrt{2}}{1+\sqrt{3}}$, so $R_0 < 1$ is equivalent to $\psi > \frac{\pi}{6}$. 
Hence, since the parabola is strictly convex for $R_0<1$, $\eta(x) < 0$ will follow for $x\in[-1,1]$ if we are able to choose the angle $\psi \in (\frac{\pi}{6}, \frac{\pi}{2})$ so that $\eta(-1),\eta(1)\leq0$.
For this we write
\begin{equation}
\eta(1) = \tfrac{\xi(R, R_0)}{(1 - R_0)^2 (1 + (1 - R_0) \theta)^2 (p - q (1 - R_0) \theta)^2 (p q (1 + \theta)^2 + (p - q \theta)^2 \lambda)},
\end{equation}
where the denominator is strictly positive and $\xi(R, R_0)$ is a polynomial with respect to $R$ and $R_0$ satisfying 
\begin{equation}
\textstyle
\xi(\frac2{1+{\sqrt 3}}, 1) = -\frac {16pq\theta^2 ((6 + 4\sqrt {3})p + (p-q)\sqrt {3}\theta + 2(3+\sqrt {3})\theta^2 q)} {(1+\sqrt {3})^5}<0
\end{equation}
(the last inequality following from the strict positivity of the expression in parentheses for $\theta \in [0, \frac{p}{q}]$, which can be easily checked).
As a consequence, for $R_0$ sufficiently close to $1$ and $R$ sufficiently close to $\frac{2}{1+\sqrt 3}$, we have $\xi(R, R_0) < 0$. Now $R_0=\frac{\sin(\frac{3\pi}{4}-\psi)}{\sin\psi}\frac{\sqrt{2}}{1+\sqrt{3}}$ equals $1$ at $\psi=\frac{\pi}{6}$, so $R_0$ can be taken to be as close to $1$ as we need by taking $\psi$ to be larger than but close to $\frac{\pi}{6}$.
This also makes $R$ arbitrarily close to $\frac{2}{1+\sqrt 3}$, and we deduce then that $\eta(1)<0$ with this choice of $\psi$.
A very similar computation shows that we also have $\eta(-1)<0$ for $\psi$ sufficiently close to $\frac{\pi}{6}$.
We have shown thus that, with these choices, $\eta(x) < 0$ for $x \in [-1,1]$, which implies $\tfrac{\d}{\d \vartheta} \Re f_3(R_0 + Re^{-\I \vartheta}) < 0$.
This finishes proving the desired decay of $\Re f_3(v)$ along $\mathcal{C}_1\cup \mathcal{C}_2$.

\vskip2pt

\noindent\underline{Step 2: analysis of $\Re \bar f_3(v)$.}\enspace
We turn now to studying $\bar f_3(v)$.
We will specify the choice of contour parameters $\phi$, $\psi$ and $\ell$ a bit later, but we point out that in the argument it will be possible to take $\ell$ as large as we want, so that condition \eqref{eq:contour-bar-condition} will be met.

Consider first $\bar f_3(v)$ on the contour $\mathcal{C}_1^+$.
Proceeding as in \eqref{eq:Re-analysis2} we get
\begin{multline}\label{eq:Re-analysis3}
\tfrac{\d}{\d r} \Re \bar f_3(r e^{\I \phi}) = -\tfrac{p q \theta^2 r^2}{(1+2\cos(\phi)r+r^2) (1-2A\cos(\phi)r+A^2r^2) (1 + 2B\cos(\phi)r+ B^2 r^2)}\\
\times\tfrac{\zeta(r,\phi,\theta)}{(1+\theta)^2(p-q\theta)^2(pq(1+\theta)^2+(p-q\theta)^2\lambda)},
\end{multline}
with
\begin{multline}
\zeta(r,\phi,\theta) = (p+(q-p)\theta+q\theta^2)\cos(3\phi)
   +(-p+2(q-p)\theta+3q\theta^2)\cos(2\phi)r\\
   +((q-p)\theta+3q\theta^2)\cos(\phi)r^2+ qr^3.
\end{multline}
All factors in \eqref{eq:Re-analysis3} other than $\zeta(r,\phi,\theta)$ are clearly positive, so we want to show that $\zeta(r,\phi,\theta)\geq0$ for all $r\geq0$ and a suitable choice of $\phi$.
Note first that $\zeta(0,\phi,\theta)=-(p-q\theta)(1+\theta)\cos(3\phi)$, which is positive for $\phi\in(\pi/6,\pi/2)$, so it is enough to show that $\zeta(r,\phi,\theta)$ is non-decreasing in $r\geq0$.
Now $\tfrac{\d}{\d r}\zeta(r,\phi,\theta)$ is a parabola in $r$ 
with roots at 
\[\textstyle r_{\pm}=\frac {-((2 + 3\theta) q-1)\cos (\phi) \pm \sqrt {((2+3\theta)q-1)^2\cos(\phi)^2 - 3 q ((1 + \theta) (1 + 3\theta) q - 1 - 2 \theta)\cos (2\phi)}} {3\theta q}.\]
Assume first that $\theta<\frac{1-2q}{3q}$.
We claim that in this case we may choose $\phi$ so that the roots $r_\pm$ are not real.
If this is the case then $\zeta(r,\phi,\theta)$ has no critical points in $r$, and since its leading coefficient in $r$ is positive, it has to be increasing (in $r$, everywhere).
Writing $x=\cos(\phi)^2$ and using $\cos(2\phi)=2\cos(\phi)^2-1$, the roots are not real when 
\[((2+3\theta)q-1)^2x -3 q ((1 + \theta) (1 + 3\theta) q - 1 - 2 \theta)(2x-1)<0.\]
At $x=0$, the left hand side equals $3 q ((1 + \theta) (1 + 3\theta) q - 1 - 2 \theta)$, which is a convex parabola in $\theta$ with roots at $\theta_{\pm}=\frac{1-2q\pm\sqrt{1-q+q^2}}{2q}$; $\theta_-<0$ while $\theta_+>\theta$ by our assumption on $\theta$, so the left hand side is negative at $x=0$, and thus it is also negative for $x$ close enough to $0$, which means $\phi$ close enough to $\pi/2$.
This gives the desired decay for the exponent when $\theta<\frac{1-2q}{3q}$.

Next consider the case $\theta\geq\frac{1-2q}{3q}$.
If the roots $r_\pm$ are not real, we have the desired decay as above.
On the other hand, if the roots are real, we claim that we may choose $\phi$ so that $r_+\leq0$, in which case $\zeta(r,\phi,\theta)$ has no critical points for $r\geq0$ and then, as above, it is increasing there.
Our condition on $\theta$ yields $(2 + 3\theta) q-1\geq0$, so $r_+\leq0$ will hold provided that $((1 + \theta) (1 + 3\theta) q - 1 - 2 \theta)\cos (2\phi)>0$.
And the latter condition can be ensured by choosing $\phi\in(\pi/6,\pi/4)$ when $(1 + \theta) (1 + 3\theta) q \geq 1 + 2 \theta$ and $\phi\in(\pi/4,\pi/2)$ otherwise.
This choice of $\phi$ then ensures the desired decay whenever $\theta\geq\frac{1-2q}{3q}$.

We have shown thus that, with the specified choice of $\phi$, $\tfrac{\d}{\d r} \Re \bar f_3(r e^{\I \phi})\leq0$ for $r\geq0$ as needed.
Note that this argument does not place any restriction on $\ell$, which can thus be taken to be as large as we need.

Next we study $\bar f_3(v)$ on the contour $\mathcal C_2^+$.
Similarly to \eqref{eq:steep}, we have
\begin{equation}
\tfrac{\d}{\d \vartheta} \Re \bar f_3(R_0 + R e^{-\I \vartheta}) = \tfrac{\bar\eta(\cos \vartheta)}{(1 - 2 \bar a_1 \cos \vartheta + \bar a_1^2) (1 + 2 \bar a_2 \cos \vartheta + \bar a_2^2) (1 - 2 \bar a_3 \cos \vartheta + \bar a_3^2)} \sin\vartheta\label{eq:steep2}
\end{equation}
with $\bar a_1 = -\frac{R}{1+R_0}$, $\bar a_2 = \frac{A R}{A R_0-1}$, $\bar a_3 = -\frac{B R}{1+B R_0}$, and with $\bar\eta(x)$ given by the parabola $\eta$ defined in \eqref{eq:function-eta}, but with the $a_i$'s replaced by the new $\bar a_i$'s (this is just minus the right hand side of \eqref{eq:steep}, with $R_0$ and $R$ replaced by $-R_0$ and $-R$).
Proceeding as above, we need to show that $\bar\eta(x)>0$ for $x\in[-1,1]$.
From the above discussion, we may assume that $\phi$ is either close to $\pi/4$ or close to $\pi/2$.
We will also assume that $\psi$ is close to $\pi/6$.

The contour for the case $\phi$ close to $\pi/4$ is just a scaled version of the contour which we employed in the analysis of \eqref{eq:function-eta}.
As in that case (but now with a general choice of $\ell$), $\phi$ close to $\pi/4$ implies that we can take $R_0$ and $R$ to be as close as needed, respectively, to $R_0=\frac{1+\sqrt{3}}{\sqrt{2}}\ell$ and $R=\ell$.
Evaluating at these points, a computation shows that
\[\bar\eta(x)=pq^2\theta^4\big(\tfrac1{8}(43+27\sqrt3)+(11+6\sqrt3)x+(5+3\sqrt3)x^2\big)\ell^6+\CO(\ell^5),\]
and since the expression in parenthesis is positive for all $x\in[-1,1]$, the right hand side is bounded below by a positive constant, for all such $x$, if $\ell$ is large enough.
So we choose, as we may, $\ell$ to be as large as needed for this to hold, and then choose $(\phi,\psi)$ to be as close to $(\pi/4,\pi/6)$ as necessary so that a slightly worse (but still positive) lower bound holds.
Similarly, if $\phi$ is close to $\pi/2$ then $R_0$ and $R$ are close, respectively, to $R_0=\sqrt{3}\ell$ and $R=2\ell$.
Evaluating at these points, another computation shows that
\[\bar\eta(x)=pq^2\theta^4\big(49\sqrt3+168x+48\sqrt3x^2\big)\ell^6+\CO(\ell^5),\]
and the same argument allows us to obtain the desired decay.
We deduce that, with these choices, $\tfrac{\d}{\d \vartheta} \Re \bar f_3(R_0 + Re^{-\I \vartheta}) < 0$.
This finishes proving the required decay of $\Re \bar f_3(v)$ along $\mathcal C_1\cup \mathcal C_2$.

\vskip2pt

\noindent\underline{Step 3: bound and convergence for $\fT^{\ep}_{-\ft,\fx}$.}
\enspace
For $\delta \in (0,\ell)$, we define 
\begin{equation}
\mathcal{C}_1^\delta=\{v \in \mathcal C_1\!:|v| \leq \delta\}\label{eq:C1-delta}
\end{equation}
and write \eqref{eq:123-a} (with $\Gamma$ replaced by $\mathcal C_1\cup\mathcal C_2$) as
\begin{equation}\label{eq:Airy-prelimit-1}
\textstyle\fT^\ep_{\ft,\fx}(\zeta,u) = \textstyle \gamma\sigma\ep^{-1/2} \frac{1}{2\pi\I}\int_{\mathcal C^\delta_1}\d v\ts e^{F_\eps(\hat\theta(1 - v))} + \gamma\sigma\ep^{-1/2} \frac{1}{2\pi\I}\int_{(\mathcal C_1 \cup \mathcal C_2) \setminus \mathcal C^\delta_1}\d v\ts e^{F_\eps(\hat\theta(1 - v))}.
\end{equation}
The second integrand in \eqref{eq:Airy-prelimit-1} is bounded in absolute value by a constant (which is independent of $\delta$, $\eps$, $u$ and $\zeta$) times 
\begin{equation}\label{eq:Airy-prelimit-1-bound}
\textstyle \ep^{-1/2} \sup_{v \in (\mathcal C_1 \cup \mathcal C_2) \setminus C^\delta_1} e^{\Re F_\eps(\hat\theta(1 - v))}.
\end{equation}
We recall \eqref{eq:F-eps-expansion-a} and the fact that $\Re f_3(v)$ decays along $\mathcal C_1\cup \mathcal C_2$. 
It means that in the exponent in the preceding expression, $\Re f_3(v)$, attains its maximum at the boundary points of the contour, which are $\delta e^{\pm\I \pi / 4}$. 
Moreover, \eqref{eq:f-asymptotics-a} yields $\Re f_3(\delta e^{\pm\I \pi / 4}) \leq - c_1 \delta^3$ for some $c_1 > 0$ and $\delta$ sufficiently small. The real parts of the functions $f_2$ and $f_0$ can be bounded by a constant $c_2 > 0$. Then the preceding expression is estimated by 
\begin{equation}\label{eq:Airy-prelimit-1-bound-1}
\textstyle \ep^{-1/2} e^{- c_1 \ep^{-3/2}\ft \delta^3} e^{c_2 (\ep^{-1}|\fx| + \ep^{-1/2} M + 1)}
\end{equation}
(recall that we are assuming $|u- \zeta| \leq M$), which vanishes as $\eps \to 0$. %

Using \eqref{eq:F-eps-expansion-a} and the asymptotics \eqref{eq:f-asymptotics-0-a} and \eqref{eq:f-asymptotics-a}, the first integral in \eqref{eq:Airy-prelimit-1} becomes
\begin{equation}\label{eq:S-convergence-0}
\textstyle \gamma\sigma\eps^{-1/2} \frac{1}{2\pi\I}\int_{\mathcal C^\delta_1}\d v\, e^{\eps^{-3/2} \frac{(\gamma\sigma)^3}{3}\ft v^3 + \eps^{-1}(\gamma\sigma)^2\fx v^2 + \eps^{-1/2}\gamma\sigma(u - \zeta   ) v + g_\eps(v)},
\end{equation}
with 
\begin{equation}
g_\eps(v) = \eps^{-3/2}\tfrac{\ft}{3}E_3(v) + \eps^{-1}\fx E_2(v) + (u-\zeta)\eps^{-1/2}E_1(v) + E_0(v).\label{eq:g-eps}
\end{equation}
Rescaling the variable $v$ by $\eps^{1/2}/(\gamma\sigma)$, we can further rewrite this integral as
\begin{equation}\label{eq:S-convergence-2}
\textstyle \frac{1}{2\pi\I}\int_{\eps^{-1/2}\gamma\sigma\mathcal C^\delta_1}\d v\, e^{\frac{\ft}{3}v^3 + \fx v^2 - (u - \zeta)v} + \frac{1}{2\pi\I}\int_{\eps^{-1/2}\gamma\sigma\mathcal C^\delta_1}\d v\, e^{\frac{\ft}{3}v^3 + \fx v^2 + (u - \zeta)v} \bigl(e^{g_\eps(\eps^{1/2} v/(\gamma\sigma))} - 1\bigr).
\end{equation}
We focus now on the second integral in \eqref{eq:S-convergence-2}.
Note that for $v\in\eps^{-1/2}\gamma\sigma\mathcal  C^\delta_1$ we have by definition of the error terms $E_i$ and the assumption $|u-\zeta|\leq M$ that $|g_\eps(\eps^{1/2}v/(\gamma\sigma))|\leq c_3\tts\ep^{1/2}(1+|v|^4)$ for some suitable $c_3>0$.
On the other hand, for $w = x + y \I$ with $x, y \in \rr$ one has
\begin{equation}\label{eq:exponential-bound}
|e^w - 1| \leq |e^{y \I} (e^x - 1)| + |e^{y \I} - 1| \leq |x| e^{|x|} + |y| \leq (|x| + |y|) e^{|x|} \leq \sqrt 2 |w| e^{|w|}.
\end{equation}
Thus the second integral in \eqref{eq:S-convergence-2} is bounded in absolute value by a constant multiple of
\begin{equation}
\textstyle\ep^{1/2}\int_{\eps^{-1/2}\gamma\sigma\mathcal C^\delta_1}\d v\, (1+|v|^4) e^{\Re(\frac{\ft}{3}v^3 + \fx v^2 + (u - \zeta )v) + \ep^{1/2}c_3(1+|v|^4)}.\label{eq:S-convergence-3}
\end{equation}
Since $\ep^{1/2}|v|\leq\delta$ on the contour $\eps^{-1/2}\mathcal C^\delta_1$, the second term in the exponent in the integrand is bounded by $c_3(1+\delta|v|^3)$.
If we parametrize now the contour as $v = re^{\pm \I \pi / 4}$ for $r\in[0,\ep^{-1/2}\gamma\sigma\delta]$, then in terms of $r$ the integrand is bounded by 
\[(1+r^4)e^{-\frac{\ft}{3\sqrt{2}}r^3 + \frac{1}{\sqrt2} (u - \zeta) r + c_3(1+\frac{\delta}{\sqrt2}r^3)}.\]
Taking $\delta$ small enough so that $-\frac{\ft}{3} + c_3\delta<0$ ensures that this expression is integrable on $r \in [0, \infty)$. 
Then the integral in \eqref{eq:S-convergence-3} is absolutely bounded by a constant which is independent of $\eps$, and thus the whole expression vanishes as $\eps \to 0$ thanks to the prefactor $\eps^{1/2}$.

It remains to analyze the first integral in \eqref{eq:S-convergence-2}. 
Letting $\langle$ denote a contour formed by rays going off the origin at angles $\pm\pi/4$, we can write this integral as
\begin{equation}
\textstyle \frac{1}{2\pi\I}\int_{\langle}\d v\, e^{\frac{\ft}{3}v^3 + \fx v^2 + (u - \zeta )v} - \frac{1}{2\pi\I}\int_{(\langle) \setminus (\eps^{-1/2}\gamma\sigma C^\delta_1)}\d v\, e^{\frac{\ft}{3}v^3 + \fx v^2 + (u - \zeta )v}.
\end{equation}
The first term equals $\fT_{-\ft,\fx}(\zeta,u)$ (see \eqref{eq:fTadjoint} and \eqref{eq:FT2}) and the second one is bounded in absolute value by
\begin{equation}
\textstyle \frac{1}{2\pi}\int_{(\langle) \setminus (\eps^{-1/2}\gamma\sigma C^\delta_1)}\d v\, e^{\Re (\frac{\ft}{3}v^3 + \fx v^2 + (u - \zeta)v)}.
\end{equation}
We note that $\Re (\frac{\ft}{3}v^3)$ is strictly negative and decaying as $v \to \infty$ along the contour. Then there are $\eps_0 > 0$ and $c_4 > 0$ such that for any $\eps \in (0, \eps_0)$ the preceding integral is bounded by $e^{- c_4 \eps^{-3/2}}$.

Summarizing the preceding computations, we get 
\begin{equation}
\textstyle \left| \fT^\ep_{-\ft,\fx}(\zeta,u) - \fT_{-\ft,\fx}(\zeta,u)\right| \leq C_\ep,
\end{equation}
for $|u-\zeta| \leq M$, with $C_\ep$ depending on $\ft$, $M$, and $\delta$, and vanishing as $\eps \to 0$.
This yields the desired limit \eqref{eq:S-eps-converges}. 

\vskip2pt

\noindent\underline{Step 4: bound and convergence for $\bar\fT^{\ep}_{-\ft,-\fx}$.}\enspace
The argument is completely analogous to the one we used for the previous case.
The main source of decay in that argument comes from the term $f_3(v)=\frac{(\gamma\sigma)^3}3v^3+E_3(v)$ in the exponent evaluated at $v$ of the form $v=r e^{\I\pi/4}$ and the fact that $\Re((r e^{\I\pi/4})^3)=-\frac1{\sqrt2}r^3$.
For $\bar\fT^{\ep}_{\ft,-\fx}$ the term $f_3(v)$ is replaced by $\bar f_3(v)=-f_3(-v)=\frac{(\gamma\sigma)^3}3v^3-E_3(-v)$, so the same decay is available to us.
The contour now has $\phi$ either close to $\pi/4$ or $\pi/2$, but in both cases we have $\Re((r e^{\I\pi/4})^3)=-c\tts r^3$ for some $c>0$, so the argument goes through with no essential change.
We omit the remaining details.
\end{proof}

\vs

\noindent{\bf Acknowledgements.}
The authors thank Yuchen Liao and Nikos Zygouras for discussions related to our results and the article \cite{bisiLiaoSaenzZigouras}.
KM was partially supported by NSF grant DMS-1953859 (transferred to DMS-2321493). 
DR was supported by Centro de Modelamiento Matem\'{a}tico Basal Funds FB210005 from ANID-Chile and by Fondecyt Grants 1201914 and 1241974.

\vs

\printbibliography[heading=apa]

\end{document}

%% file: catp-ch.tex
\newcommand{\catp}{\raisebox{0.1em}{\tikz[y=0.80pt,x=0.80pt,yscale=-1.0,xscale=1.0,inner sep=0pt,outer sep=0pt]%
{
\path[draw=black,line join=bevel,line cap=round,miter limit=4.00,line width=0.280pt] (-240.5382,165.4074) circle (0.0476cm);
\path[draw=black,fill=white,line join=bevel,line cap=round,miter limit=4.00,line width=0.280pt] (-239.2864,165.0583) circle (0.0476cm);
\path[draw=black,fill=white,line join=bevel,line cap=round,miter limit=4.00,line width=0.280pt] (-238.0008,164.9891) circle (0.0476cm);
\path[draw=black,fill=white,line join=bevel,line cap=round,miter limit=4.00,line width=0.280pt] (-236.6079,164.6444) circle (0.0476cm);
\path[draw=black,fill=white,line join=bevel,line cap=round,miter limit=4.00,line width=0.280pt] (-235.2670,164.3894) circle (0.0476cm);
\path[draw=black,fill=white,line join=bevel,line cap=round,miter limit=4.00,line width=0.280pt] (-234.0488,163.8984) circle (0.0476cm);
\path[draw=black,fill=white,line join=bevel,line cap=round,miter limit=4.00,line width=0.280pt] (-233.9610,163.2756) ellipse (0.0096cm and 0.0141cm);
\path[draw=black,line join=round,line cap=round,miter limit=4.00,line width=0.130pt] (-234.3312,162.2005) .. controls (-234.3312,162.2005) and (-234.7325,160.8832) .. (-235.0441,161.2231);
\path[draw=black,line join=round,line cap=round,miter limit=4.00,line width=0.130pt] (-233.5987,162.2260) .. controls (-233.5987,162.2260) and (-233.1974,160.9087) .. (-232.8857,161.2486);
}}}